\documentclass[12pt]{amsart}
\usepackage[applemac]{inputenc}

\usepackage{fullpage}
\usepackage{
 ulem,
 amsfonts}
 \usepackage{amsmath}
\usepackage{amsmath,esint, amssymb, amsthm}

 \usepackage{mathtools}

\usepackage[colorlinks=true, pdfstartview=FitV, linkcolor=blue, citecolor=blue,
 urlcolor=blue]{hyperref}
\usepackage{color}

\date{}

\newcommand{\CC}{\mathbb{C}}  
\newcommand{\NN}{\mathbb{N}}  
\newcommand{\RR}{\mathbb{R}}  
 
\newcommand{\Rea}{\operatorname{Re}}
\newcommand{\Ima}{\operatorname{Im}}
\theoremstyle{plain}
\newtheorem{definition}{Definition}
\numberwithin{equation}{section}
\newtheorem{theorem}{Theorem}[section]
\newtheorem{proposition}[theorem]{Proposition}
\newtheorem{lemma}[theorem]{Lemma}
\newtheorem{remark}[theorem]{Remark}
\newtheorem*{remarks}{Remarks}

\usepackage[textmath,displaymath,floats,graphics,delayed]{preview}
  \title[]{Doubly connected V-states for the generalized surface quasi-geostrophic equations}
\author[F. de la Hoz]{Francisco de la Hoz}
\address{ Department of Applied Mathematics and Statistics and Operations Research, Faculty of Science and Technology \\
University of the Basque Country UPV/EHU, Barrio Sarriena S/N\\
 48940 Leioa, Spain 
 }
\email{francisco.delahoz@ehu.es}

\author[Z. Hassainia]{ZINEB HASSAINIA}
\address{IRMAR, Universit\'e de Rennes 1 \\ Campus de Beaulieu \\  35 042 Rennes cedex, France}
 \email{zineb.hassainia@univ-rennes1.fr}
 \author[T. Hmidi]{Taoufik Hmidi}
\address{IRMAR, Universit\'e de Rennes 1\\ Campus de
Beaulieu\\ 35~042 Rennes cedex\\ France}
\email{thmidi@univ-rennes1.fr}
 
\begin{document}

\begin{abstract}
In this paper, we prove  the existence of  doubly connected V-states for the generalized SQG equations with $\alpha\in ]0,1[.$ They can be described by countable branches  bifurcating from the annulus  at some explicit "eigenvalues" related to Bessel functions of the first kind. Contrary to Euler equations \cite{H-F-M-V}, we find V-states rotating  with positive and negative angular velocities. At  the end of the paper we discuss some numerical experiments concerning  the limiting V-states. \end{abstract}

\newpage 

\maketitle{}
\tableofcontents
%

\section{Introduction}
The present work deals with the   generalized   surface quasi-geostrophic  equation (gSQG) arising in fluid dynamics and  which describes the evolution  of the potential temperature $\theta$ by the transport  equation:

\begin{equation}\label{sqgch2}
\left\{ \begin{array}{ll}
\partial_{t}\theta+u\cdot\nabla\theta=0,\quad(t,x)\in\RR_+\times\RR^2, &\\
u=-\nabla^\perp(-\Delta)^{-1+\frac{\alpha}{2}}\theta,\\
\theta_{|t=0}=\theta_0.
\end{array} \right.
\end{equation}
Here $u$ refers to the velocity field, $\nabla^\perp=(-\partial_2,\partial_1)$  and $\alpha$ is a real parameter taken in  $]0,2[$. 
The singular operator   $(-\Delta)^{-1+\frac{\alpha}{2}}$ is  of convolution type and defined by,
\begin{equation}\label{Integ1}
(-\Delta)^{-1+\frac{\alpha}{2}} \theta(x)=\frac{C_\alpha}{2\pi}{\int}_{\RR^2}\frac{\theta(y)}{\vert x-y\vert^\alpha}dy,
\end{equation}
with  $C_\alpha=\frac{\Gamma(\alpha/2)}{2^{1-\alpha}\Gamma(\frac{2-\alpha}{2})}$ where $\Gamma$ stands for the gamma function.  This model  was proposed by C\'ordoba { et al.} in \cite{C-F-M-R} as an interpolation between Euler equations and the surface quasi-geostrophic model (SQG)  corresponding to $\alpha=0$ and $\alpha=1$, respectively. The SQG equation was used  by Juckes \cite{Juk} and Held {et al.} \cite{Held} to describe the atmosphere  circulation  near the tropopause. It  was also used   by Lapeyre and Klein \cite{Lap} to track the  ocean dynamics in  the upper layers. We note  that there is a strong  mathematical and physical analogy with the three-dimensional incompressible Euler equations; see \cite{C-M-T} for details. 

In the last few years there has been a growing interest in the mathematical study of these active scalar equations. Special attention has been paid to 
the local well-posedness of classical solutions  which can be performed in various functional spaces. For instance, this was implemented  in the framework of Sobolev spaces  \cite{C-C-C-G-W} by using the commutator theory. Wether or not these solutions are global in time is an open problem except for Euler equations $\alpha=0$. The second  restriction with the  gSQG equation concerns the construction of Yudovich solutions -- known to exist globally in time for  Euler equations \cite{Y} -- which are not at all clear even locally in time.  The main difficulty  is due to the velocity which  is in general singular and scales  below the Lipschitz class. Nonetheless one can say more about this issue for some special class of concentrated vortices. More precisely,  when the initial datum has a vortex patch structure, that is, $\theta_0(x)=\chi_D$ is  the characteristic function of a bounded simply connected smooth domain $D$, then there is a unique local solution in the patch form $\theta(t)=\chi_{D_t}. $ In this case,  the boundary   motion of the domain $D_t$  is described by the contour dynamics formulation; see the papers \cite{Gan,Ro}.  The global persistence of the boundary regularity is only known for $\alpha=0$  according to Chemin's result \cite{C}; for another proof see  the paper of Bertozzi and Constantin \cite{B-C}. Notice that for $\alpha>0$ the numerical experiments carried out in \cite{C-F-M-R} provide strong evidence for the   singularity formation  in finite  time. Let us mention that  the contour dynamics equation remains  locally well-posed when the domain of the initial patch  is assumed to be   multi-connected meaning that  the boundary is composed with  finite number of disjoint smooth Jordan curves.  
  
The main concern of  this work is to  explore  analytically and numerically some special vortex patches called V-states; they correspond to patches which do not change their shapes during the motion. The emphasis will be put on the V-states subject  to uniform rotation  around their center of mass, that is,  $D_t= {\bf R}_{x_0,\Omega t}D$, where $ {\bf R}_{x_0,\Omega t}$  stands for  the planar rotation with center $x_0$ and angle $\Omega t.$ The parameter $\Omega$ is called the angular velocity of the rotating domain. Along the chapter we call these structures rotating patches or simply V-states. Their existence  is of great interest for at least  two reasons: first they provide  non trivial initial data with global existence, and second this might explain the emergence of some ordered structures in the geophysical flows.
This study has been conducted first for  the two-dimensional  Euler equations ($\alpha=0$) a long time  ago  and a number of  analytical and numerical studies are known in   the literature. The first result in this setting   goes back to  Kirchhoff \cite{Kirc} who discovered that an ellipse of semi-axes $a$ and $b$  rotates uniformly  with the  angular velocity $\Omega = ab/(a+b)^2$; see for
instance the references \mbox{\cite[p. 304]{M-B}} and \cite [p. 232]{Lamb}. Till now this is  the only known explicit V-states; however the existence of  implicit examples was established about one century later.  In fact,  Deem and Zabusky \cite{DZ} gave numerical  evidence of  the existence of the 
V-states with $m$-fold symmetry for each  integer $m \geq 2$; remark that the  case $m=2$ coincides with  Kirchhoff's ellipses. To fix the terminology,  a planar 
domain is said  $m$-fold symmetric if it has the same group invariance of a regular polygon with $m$ sides.  Note that at each frequency $m$ these V-states can be seen as a continuous deformation of the disc with respect to  to the angular velocity. An analytical  proof of this fact  was given few years later by 
Burbea in \cite{Bur}. His  approach consists in writing a stationary  problem in the frame of the patch with the conformal mapping of the domain and to look for   the  non trivial solutions  by using the technique of the bifurcation theory.  Quite recently, in   \mbox{ \cite{HMV}}  Burbea's approach was revisited  with  more details and explanations. The boundary regularity of the V-states was also studied  and  it was  shown  to be  of class   $C^\infty$ and convex  close to the disc.

We mention   that  explicit vortex solutions similar to the ellipses  are discovered in the literature for the incompressible Euler equations in the presence of an  external shear flow; see for instance \cite{Chapl,Kida,Neu}. A general review about vortex dynamics can be found in the \mbox{papers  \cite{Ar,New}. }

With regard to  the existence of the simply connected  V-states for the (gSQG) it has been discussed very recently in the papers  \cite{Cor1,Cor,H-H}. In \cite{Cor}, it was shown that the ellipses cannot rotate for any $\alpha\in(0,2)$ and to the authors' best knowledge no explicit example is known in the literature.  Lately,  in \cite{H-H} the last two authors proved the analogous of  Burbea's result  and showed the  existence of the  $m$-folds rotating patches for $\alpha\in]0,1[$.  In addition, the bifurcation from the unit disc occurs at the angular velocities,
$$
\Omega_m^\alpha\triangleq\frac{\Gamma(1-\alpha)}{2^{1-\alpha}\Gamma^2(1-\frac\alpha2)}\bigg(\frac{\Gamma(1+\frac\alpha2)}{\Gamma(2-\frac\alpha2)}-\frac{\Gamma(m+\frac\alpha2)}{\Gamma(m+1-\frac\alpha2)}\bigg),\quad m\geq 2,
$$
where $\Gamma$ denotes the usual  gamma function.The remaining case $\alpha\in[1,2)$  has been explored  and  solved  by Castro, C\'ordoba and  G\'omez-Serrano in \cite{Cor1}. They also show that the V-states are $C^\infty$ and convex close to the discs. To complete these works  we  discussed in this work ( which is forwarded  in  \mbox{Section \ref{Sec-numch1}}) some numerical experiments concerning these V-states and their limiting structures when we go to the end of each branch; new  behaviors will be observed compared to the numerical experiments achieved for Euler \mbox{case \cite{S-Z}.}

We want in   this chapter to learn more about  the V-states but with different topological structure compared to the preceding discussion. More precisely, we propose to scrutinize  rotating patches with only one hole, also called  doubly  connected V-states.  Recall that a patch $\theta_0=\chi_{D}$ is said to be doubly connected if the   domain $D=D_1\backslash D_2$, with $D_1$ and $D_2$ being two simply connected  bounded  domains satisfying $\overline{D_2}\subset D_1.$  This structure is preserved for Euler system globally in time but known to be for short time when  $\alpha\in ]0,1[$ see \cite{C-C-C-G-W,Gan,Ro}.  We notice that compared to the simply connected case the boundaries evolve through  extra nonlinear terms coming from the interaction between the boundaries and therefore  the existence of the V-states is relatively more complicate to analyze. 
This problem is not well studied  from the analytical point of view and recent progress has been  made for Euler equations in the papers \cite{Flierl,HMV2, H-F-M-V}.  In \cite{HMV2}, the authers  proved  the existence of  explicit V-states similar to Kirchhoff ellipses seems to be out of reach. Indeed, it was stated  that if  one of the boundaries of the V-state is a circle then necessarily the other one  should be also a circle. Moreover, if the inner curve is an ellipse then there is no rotation at all.  Another closely related subject is to deal  with some vortex magnitude $\mu$ inside the  domain $D_2$ and try to find explicit rotating patches. This was done by Flierl and Polvani \cite{Flierl} who   proved that confocal ellipses rotate uniformly provided some compatibility relations  are satisfied between the parameter $\mu$ and the semi-axes of the ellipses. We note that another approach based upon complex analysis tools  with a complete discussion can be found \mbox{in \cite{HMV2}.}

Now,  from the equations \eqref{sqgch2} we may easily conclude that the   annulus is  a stationary  doubly connected patch,  and therefore it rotates with any angular velocity  $\Omega$. From this obvious fact, one can wonder wether or not  the bifurcation to nontrivial V-states still happens as for the simply connected case. This has been recently investigated
 in   \cite{H-F-M-V} for Euler equations following basically  Burbea's approach but with more involved calculations.  It was shown that for $b\in(0,1)$ and $m$ being an integer satisfying the inequality
\begin{equation}\label{condeulerc2}
1+b^m-\frac{1-b^2}{2}m<0
\end{equation}
then there exist two curves  of  non-annular $m$-fold doubly connected patches  bifurcating from the annulus $\big\{z; b<|z|< 1\big\}$ at different eigenvalues $\Omega_m^\pm$ given explicitly by the formula
$$
\Omega_m^\pm=\frac{1-b^2}{4}\pm\frac{1}{2m}\sqrt{\Big[\frac{m}{2}{(1-b^2)}-1\Big]^2-b^{2m}}.
$$ 
 Now we come to the main contribution of the current work. We propose to study the doubly connected V-states for the gSQG model \eqref{sqgch2}  when   $\alpha\in ]0,1[$.  Before stating our result we need to make some  notation. We define 
  \begin{equation*}
\Lambda_n(b)\triangleq \frac1b\int_0^{+\infty}{J_n(bt)J_n(t)}\frac{dt}{t^{1-\alpha}},
\end{equation*}
and
\begin{eqnarray*}
\Theta_{n}&\triangleq& \Lambda_1(1)-\Lambda_n(1),
\end{eqnarray*}
where $J_n$ refers to  Bessel function of the first kind. Our result reads as follows.

\begin{theorem}\label{main}
Let  $\alpha\in[0,1[$ and $b\in]0,1[$; there exists $N\in \NN$ with the following property: \\
For each   $m> N$ there exists two curves of  $m$-fold  doubly
connected $V$-states  that bifurcate from the annulus $\big\{z\in \CC, b < |z| < 1
\big\}$ at  the angular velocities
$$
\Omega^{\alpha, \pm}_m\triangleq \frac{1-b^2}{2}\Lambda_1(b)+\frac12(1-b^{-\alpha})\Theta_m\pm\frac12\sqrt{\Delta_m(\alpha,b)},
$$
with
\begin{eqnarray*}
\Delta_m(\alpha,b)&\triangleq & \Big[(b^{-\alpha}+1)\Theta_{m}-(1+b^2)\Lambda_{1}(b)\Big]^2-4b^2\Lambda_{m}^2(b).
\end{eqnarray*}
\end{theorem}

\begin{remarks}\label{rmw1}
\begin{enumerate}
\item The number  $N$ is 
the smallest  integer such that 
\begin{equation}\label{condsqgc2}
\Theta_{N}\geq\frac{1+b^2}{b^{-\alpha}+1}\Lambda_1(b)+\frac{2b}{b^{-\alpha}+1}\Lambda_N(b).
\end{equation}
This restriction appears in the spectral study of the linearized operator and gives only sufficient condition for the existence of the V-states.  
\item As we shall see later in Lemma \ref{lem1} ,  for $\alpha=0$  
we find  the  result of  Euler equations established in \cite{H-F-M-V} and the condition $\eqref{condsqgc2}$ is in accordance with that  given by $\eqref{condeulerc2}.$ 

 \item We can check by using the strict monotonicity of $b\mapsto \Lambda_1(b)$  that for any $b,\alpha\in(0,1),$
$$
\lim_{m\to+\infty}\Omega_m^{\alpha,-}=-b^{-\alpha}\Lambda_1(1)+\Lambda_1(b)<0.
$$
Consequently the corresponding bifurcating curves  generate close to the annulus non trivial  clockwise  doubly connected V-states.  This fact is completely  new compared to what we know for Euler equations or for the simply connected case where the bifurcation occurs at positive angular velocities. The numerical experiments discussed in Section  $\ref{Sec-num}$ reveal  the existence of  non radial stationary patches  for the generalized quasi-geostrophic equations and it would be very interesting to  establish this fact analytically. In a connected subject,  we point out that the last author has shown quite recently in \cite{Hmidi}  that for Euler equations  clockwise convex  V-states reduce to the discs. 

\end{enumerate}
\end{remarks}


Now we shall sketch the proof of  Theorem \ref{main} which is mainly based  upon the bifurcation theory via  Crandall-Rabinowitz's Theorem. The first step is to write down the  analytical equations of the boundaries of the V-states. This can be done for example through the conformal parametrization of the domains $D_1$ and $D_2$: we denote by $\phi_j:\mathbb{D}^c\to D_j^c$  the conformal mappings possessing   the following structure,
$$
\forall \,|w|\geq1,\quad \phi_1(w)=w+\sum_{n\in\NN} \frac{a_n}{w^n},\quad \phi_2(w)= b w+\sum_{n\in\NN} \frac{c_n}{w^n}.
$$
We assume in addition that the Fourier coefficients are real which means that we look only  for  the V-states which are symmetric with respect to the real axis. Moreover using the subordination principle we deduce that $b\in ]0,1[$; the parameter $b$ coincides with the small radius of the annulus that we slightly  perturb. As we shall see later in Section \ref{sec12},  the conformal mappings  are subject to two coupled nonlinear equations  defined as follows: for $j\in\{1,2\}$
\begin{eqnarray*}
F_j(\Omega, \phi_1,\phi_2)(w)&\triangleq& \Ima\Bigg\{ \Big(\Omega\,\phi_j(w)+ S(\phi_2,\phi_j)(w)-S(\phi_1,\phi_j)(w)\Big)\overline{w}\,{\overline{\phi_j'(w)}}\Bigg\}\\
&=&0,\quad\forall w\in \mathbb{T},
\end{eqnarray*}
with
$$
 S(\phi_i,\phi_j)(w)={C_\alpha}\fint_\mathbb{T}\frac{\phi_i'(\tau)}{\vert \phi_j(w)-\phi_i(\tau)\vert^\alpha}d\tau\quad\hbox{and}\quad C_\alpha\triangleq\frac{\Gamma(\alpha/2)}{2^{1-\alpha}\Gamma(\frac{2-\alpha}{2})}\cdot
$$
In order to apply  the bifurcation theory  we should understand the structure of the linearized operator around the trivial solution $(\phi_1,\phi_2)=(\hbox{Id},b\,\hbox{Id})$, corresponding to the annulus with radii $b$ and $1$, 
 and  identify the range of $\Omega $ where this operator has a one-dimensional kernel. The computations of the linear operator $DF(\Omega,\hbox{Id},b\,\hbox{Id})$ with $F=(F_1,F_2)$ in terms of its Fourier coefficients   are long and tricky. They are    connected  to the hypergeometric functions ${}_{2}F_{1}(a,b;c;z)$ simply denoted by $F(a,b;c;z)$ throughout this chapter.  To find compact  formula we use at several steps some algebraic identities described by the contiguous function relations \eqref{f0}-\eqref{f5}. Similarly to the Euler equations \cite{H-F-M-V}  the linearized operator acts as a matrix Fourier multiplier. More precisely, for 
 $${ h_1(w)=\sum_{n\geq1}\frac{a_n}{w^n}}, \quad \displaystyle h_2(w)=\sum_{n\geq1}\frac{c_n}{w^n},$$
 we  obtain the formula,
   \begin{eqnarray}\label{spec11}
\nonumber DF(\Omega,\hbox{Id},b\,\hbox{Id})\big(h_1,h_2\big)(w)=\frac{i}{2}\sum_{n\geq 1}\big(n+1\big)M^\alpha_{n+1}\left( \begin{array}{c}
a_n \\
c_n
\end{array} \right)\Big(w^{n+1}-\overline{w}^{n+1}\Big),
\end{eqnarray}
where the matrix $M_n$ is given for $n\geq2$  by
\begin{equation*}
M^\alpha_{n}\triangleq\begin{pmatrix}
 \Omega-\Theta_{n}+b^2\Lambda_{1}(b) & -b^2\Lambda_{n}(b) \\
 \\
  b\Lambda_n(b) & b\Omega+b^{1-\alpha}\Theta_{n}-b\Lambda_1(b)
\end{pmatrix}.
\end{equation*}
Therefore  the values of $\Omega$ associated to  non trivial kernels are the solutions of a  second-degree polynomial,
\begin{equation}\label{det22}
\hbox{det }M^\alpha_n=0.
\end{equation}
This can be solved  when the discriminant $\Delta_n(\alpha,b)$ introduced in Theorem \ref{main} is positive. The computations of the  dimension of the kernel are more complicate than the cases raised before in the references  \cite{H-H,H-F-M-V}.  The matter reduces to count  the following discrete set
$$
\big\{n\geq2, \hbox{det }M^\alpha_n=0\big\}.
$$
Note that in    \cite{H-H,H-F-M-V}  this set has only one element and therefore the kernel is one-dimensional. This follows from  the  monotonicity of the "nonlinear eigenvalues" sequence $n\mapsto \Omega$ which is not very hard to get  due to the explicit polynomial structure  of the coefficients of the analogous  polynomial  to \eqref{det22}. Unfortunately, in the current situation  this   structure is broken   because the matrix coefficients of  $M_n^\alpha$ are related to Bessel functions.  Therefore  the monotonicity  of the eigenvalues is more subtle and will require more refined analysis. 
This subject will be discussed later with ample   details  in the Subsection \ref{subsec12}. To achieve the spectral study and check the complete assumptions of Crandall-Rabinowitz's Theorem it remains to prove the transversality assumption and check that the image is of co-dimension one. This will be done in Section \ref{sec45} in a straightforward way and without serious difficulties. We also mention that the transversality assumption is obtained only when the discriminant $\Delta_n(\alpha,b)>0$ meaning that there is no crossing roots for  the equation \eqref{det22}.  The proof of the bifurcation will be achieved in Section \ref{sec45}.
%
Next, we shall make few  comments about the statement of the main theorem.
\begin{remarks}


${\bf{1)}}$  For the SQG equation corresponding to $\alpha=1$ the situation is more delicate due to some logarithmic loss. The simply connected case has been achieved recently in \cite{Cor1} by using Hilbert spaces where we take into account this loss. The same approach could lead to the existence of the doubly-connected V-states for the (SQG) equation. For the spectral study, the linearized operator can be obtained as a limit of \eqref{spec11} when $\alpha$ goes to $1.$ More precisely, we get

\begin{eqnarray*}
\nonumber DF(\Omega,\hbox{Id},b\,\hbox{Id})\big(h_1,h_2\big)(w)=\frac{i}{2}\sum_{n\geq 1}\big(n+1\big)M^1_{n+1}\left( \begin{array}{c}
a_n \\
c_n
\end{array} \right)\Big(w^{n+1}-\overline{w}^{n+1}\Big),
\end{eqnarray*}
where the matrix $M_n$ is given for $n\geq2$  by
\begin{equation*}
M^1_{n}\triangleq\begin{pmatrix}
 \Omega-\Theta_{n}+b^2\Lambda_{1}(b) & -b^2\Lambda_{n}(b) \\
 \\
  b\Lambda_n(b) & b\Omega+\Theta_{n}-b\Lambda_1(b)
\end{pmatrix},
\end{equation*}
with
\begin{equation*}
\Theta_n=\frac{2}{\pi}\sum_{k=1}^{n-1}\frac{1}{2k+1}\quad\hbox{and}\quad \Lambda_n(b)=\frac1b\int_0^{+\infty}J_n(bt)J_n(t)dt.
\end{equation*}
\vspace{0,3cm}

${\bf{3)}}$ The  boundary  of the V-states  belongs  to H\"older space $C^{2-\alpha}$. For Euler equations corresponding to $\alpha=0$, we get better result in the simply connected geometry as  it was shown \mbox{in \cite{HMV};} the boundary  is $C^\infty$ and convex when the V-states are close to the circle. The proof in this particular case uses in a deep way the algebraic structure of the kernel according to some recurrence formulae. The extension of this result  to  $\alpha\in ]0,2[$ was done in  \cite{Cor1}.  We expect that the latter approach  could  be also adapted to the {\hbox{$($gSQG$)$}} model for the doubly connected case.
\vspace{0,3cm}

${\bf{4)}}$ In the setting  of the vortex patches   the global existence with smooth boundaries is not known for $\alpha\in]0,2[$. The simply connected V-states discussed in \cite{Cor1, H-H} offer a first class of global solutions which are  periodic in time. We find here a second class of global solutions which are the doubly connected V-states.\vspace{0,3cm}

\end{remarks}

The remainder of the chapter is organized as follows. In the next section, we shall write down the boundary equations through the conformal parametrization. In \mbox{Section $3$,}  we shall introduce and review  some background material on the bifurcation theory and Gauss hypergeometric functions. In Section $4$, we will study the regularity of the nonlinear functionals involved in the boundary equations.  
In Section $5,$ we conduct the spectral study and formulate the suitable assumptions to get a Fredholm operator of zero index. In Section $6$ we prove Theorem \ref{main}. Finally, the last section will be devoted to some numerical experiments dealing with the simply and doubly connected V-states.

\vspace{0,2cm}

{\bf {Notation.}}
We need to fix some notation that will  be frequently used along this chapter.
\begin{enumerate}
\item[$\bullet$] We denote by C any positive constant that may change from line to line.
\item[$\bullet$] For any positive real numbers $A$ and $B$, the notation $A \lesssim B$ means that there exists a
positive constant $C$ independent of $A$ and $B$ such that $A\leq CB$.
\item[$\bullet$] We denote by $\mathbb{D}$ the unit disc. Its boundary, the unit circle, is denoted by  $\mathbb{T}$. 
\item[$\bullet$] Let $f:\mathbb{T}\to \CC$ be a continuous function. We define  its  mean value by,
$$
\fint_{\mathbb{T}} f(\tau)d\tau\triangleq \frac{1}{2i\pi}\int_{\mathbb{T}}  f(\tau)d\tau,
$$
where $d\tau$ stands for the complex integration.
\item[$\bullet$] Let $X$ and $Y$ be two normed spaces. We denote by $\mathcal{L}(X,Y)$ the space of  all continuous linear maps $T: X\to Y$ endowed with its usual strong topology. 
\item[$\bullet$] For a linear operator $T:X\to Y,$ we denote by $N(T)$ and $R(T)$  the kernel and the range of $T$, respectively.  
 \item[$\bullet$] If $Y$ is a vector space and $R$ is a subspace, then $Y/ R$ denotes the quotient space.

\end{enumerate}

\section{Boundary equations }\label{sec12}

Before proceeding further with the consideration of the V-states,  we shall recall Riemann mapping theorem which  is one of the most important results in complex analysis. To restate this result we need to recall the definition of {\it simply connected} domains. Let $\widehat{\mathbb{C}}\triangleq \mathbb{C}\cup\{\infty\}$ denote the Riemann sphere. We  say that a domain $\Omega\subset \widehat{\mathbb{C}}$ is {\it simply connected} if the set $ \widehat{\mathbb{C}}\backslash \Omega$ is connected.

{\it Riemann Mapping Theorem.} Let $\mathbb{D}$ denote the unit open ball and  $\Omega\subset \CC$ be a simply connected bounded domain. Then there is a unique bi-holomorphic map called also conformal map,  $\Phi: \CC\backslash\overline{\mathbb{D}}\to  \CC\backslash\overline{\Omega}$ taking the form
$$
\Phi(z)= az+\sum_{n\in\NN} \frac{a_n}{z^n} \quad \textnormal{with}\quad a>0.
$$
 In this theorem the regularity of the boundary has no effect regarding the existence of the conformal mapping  but  it contributes in the boundary behavior of the conformal mapping, see for instance \cite{Po,WS}. Here, we shall recall the following result. 
\vspace{0,2cm}


{\it  Kellogg-Warschawski's theorem.} It can be found in  \cite{WS} or in  \cite[Theorem 3.6]{Po}.  It asserts that if   the conformal map
$\Phi:\CC\backslash\overline{\mathbb{D}}\to  \CC\backslash\overline{\Omega}$ has a continuous
extension to $\CC\backslash\mathbb{D}$  which is of \mbox{class
$C^{n+1+\beta},$} with
$n\in \NN$  and $ 0<\beta<1$,  then the
boundary $\Phi(\mathbb{T})$ is a Jordan curve of class $C^{n+1+\beta}.$

Next, we shall write down the equation governing the boundary of the doubly connected  V-states.
Let $D=D_1\backslash D_2$ be a doubly connected domain, that is,   $D_1$ and $D_2$ are two simply connected domains with $D_2\subset D_1$. Denote by $\Gamma_1$ and $\Gamma_2$ their boundaries, respectively. Consider the parametrization by the conformal mapping: $\phi_j: \mathbb{D}^c\to D_j^c$ satisfying
$$
\phi_1(z)=z+f_1(z)=z\Big(1+\sum_{n=1}^{\infty}\frac{a_{n}}{z^{n}}\Big),
$$
and
$$
\phi_2(z)=bz+f_2(z)=z\Big(b+\sum_{n=1}^{\infty}\frac{b_{n}}{z^{n}}\Big), \quad 0<b<1.
$$
Now assume that $\theta_0=\chi_D$ is a rotating patch for the model \eqref{sqgch2} then according to \cite{H-H} the boundary equations are given by
\begin{eqnarray}\label{model0}
\nonumber\Omega\Rea\big\{z\overline{z'}\big\}&=& C_\alpha\Ima \Bigg\{\frac{1}{2\pi}\int_{\partial D}\frac{d\zeta}{\vert z-\zeta\vert^\alpha}\overline{z'}\Bigg\},\quad \quad \forall\, z\in \partial D=\Gamma_1\cup\Gamma_2.\\
\quad&=& C_\alpha\Ima\Bigg\{\Big(\frac{1}{2\pi}\int_{\Gamma_1}\frac{d\zeta}{\vert z-\zeta\vert^\alpha}-\frac{1}{2\pi}\int_{\Gamma_2}\frac{d\zeta}{\vert z-\zeta\vert^\alpha}\Big)\overline{z'}\Bigg\},
\end{eqnarray}
where $z^\prime$ denotes a tangent vector to the boundary $\partial D$ at the point $z.$
We shall now rewrite the equations by using the conformal parametrizations $\phi_1$ and $\phi_2$.
First remark that for $w\in \mathbb{T}$ a tangent vector on the boundary $\Gamma_j$ at the point $\phi_j(w)$ is given by
$$
\overline{z'}=-i\overline{w}\, {\overline{\phi_j'(w)}}.
$$
Inserting this into the  equation \eqref{model0} and using the change of variables $\tau=\phi_j(w)$ give
\begin{eqnarray*}
\forall \, w\in \mathbb{T},\quad F_j\big(\Omega,\phi_1,\phi_2\big)(w)=0;\quad j=1,2,
\end{eqnarray*}
with
\begin{eqnarray}\label{g_1}
\nonumber &&F_j\big(\Omega,\phi_1,\phi_2\big)(w)\triangleq \Omega\,\Ima\Big\{\phi_j(w)\overline{w}\,{\overline{\phi_j'(w)}}\Big\}\\
&+&{C_\alpha}\Ima\Bigg\{\bigg(\fint_\mathbb{T}\frac{\phi_2'(\tau)d\tau}{\vert \phi_j(w)-\phi_2(\tau)\vert^\alpha}-\fint_\mathbb{T}\frac{\phi_1'(\tau)d\tau}{\vert \phi_j(w)-\phi_1(\tau)\vert^\alpha}\bigg) \overline{w}\,{\overline{\phi_j'(w)}}\Bigg\}
\end{eqnarray}
and \quad $\displaystyle{C_\alpha\triangleq\frac{\Gamma(\alpha/2)}{2^{1-\alpha}\Gamma(\frac{2-\alpha}{2})}}$. We shall introduce  the functionals
\begin{equation*}
G_j(\Omega,f_1,f_2)\triangleq F_j\big(\Omega,\phi_1,\phi_2\big)\quad j=1,2.
\end{equation*}
Then equations of the V-states become,
\begin{eqnarray}\label{Rotaeq1}
\forall w\in \mathbb{T},\quad G_j\big(\Omega,f_1,f_2\big)(w)=0,\quad j=1,2.
\end{eqnarray}
Now it is easy to ascertain that the annulus is a rotating patch for any $\Omega\in \RR$. Indeed, replacing $\phi_1$ and $\phi_2$ in \eqref{g_1} by $\hbox{Id}$ and $b\,\hbox{Id}$, respectively,   we get
\begin{eqnarray*}
\nonumber F_1\big(\Omega,\hbox{Id},b\,\hbox{Id}\big)(w)&=&{C_\alpha}\Ima\bigg\{b\overline{w}\mathop{{\fint}}_\mathbb{T}\frac{d\tau}{\vert w-b\tau\vert^\alpha}-\overline{w}\mathop{{\fint}}_\mathbb{T}\frac{d\tau}{\vert w-\tau\vert^\alpha}\bigg\}.
\end{eqnarray*}
Using  the change of variables $\tau=w\zeta$ in the two preceding integrals  we find
\begin{eqnarray*}
\nonumber F_1\big(\Omega,\hbox{Id},b\,\hbox{Id}\big)(w)&=&{C_\alpha}\Ima\bigg\{b\mathop{{\fint}}_\mathbb{T}\frac{d\zeta}{\vert 1-b\zeta\vert^\alpha}-\mathop{{\fint}}_\mathbb{T}\frac{d\zeta}{\vert 1-\zeta\vert^\alpha}\bigg\}.
\end{eqnarray*}
Note that each integral in the right side is real since,
\begin{eqnarray*}
\forall a\in (0,1],\quad\overline{\mathop{{\fint}}_\mathbb{T}\frac{d\zeta}{\vert 1-a\zeta\vert^\alpha}} &=& -\mathop{{\fint}}_\mathbb{T}\frac{d\overline{\zeta}}{\vert 1-a\overline{\zeta}\vert^\alpha}\\ &=& \mathop{{\fint}}_\mathbb{T}\frac{d\xi}{\vert 1-a\xi\vert^\alpha}\cdot
\end{eqnarray*}
Therefore we obtain,
\begin{eqnarray*}
\forall \, w\in \mathbb{T},\quad F_1\big(\Omega,\hbox{Id},b\,\hbox{Id}\big)(w)&=&0.
\end{eqnarray*}
Arguing similarly for the second component $F_2$ we get for any $w\in \mathbb{T},$
\begin{eqnarray*}
\nonumber F_2\big(\Omega,\hbox{Id},b\,\hbox{Id}\big)(w)={C_\alpha}\Ima\bigg\{b^{1-\alpha}\mathop{{\fint}}_\mathbb{T}\frac{d\zeta}{\vert 1-\zeta\vert^\alpha}-\mathop{{\fint}}_\mathbb{T}\frac{d\zeta}{\vert b-\zeta\vert^\alpha}\bigg\}= 0,
\end{eqnarray*}
which is the desired result.

%

\section{Tools}
In this section  we shall recall in the first part some simple facts about H\"older spaces on the unit circle $\mathbb{T}$. In the second part we  state   Crandall-Rabinowitz's Theorem which is a crucial tool in the proof of Theorem \ref{main}. We shall also recall some important  properties  of the hypergeometric functions which appear in a natural way in the spectral study of the linearized operator. The last part is devoted to  the computations of some integrals used later in the spectral study. 
\subsection{Functional spaces}
 Throughout this chapter it is more convenient  to think of  $2\pi$-periodic function $g:\RR\to\CC$ as a function of the complex variable $w=e^{i\eta}$. To be more precise, let  $f:\mathbb{T}\to \RR^2$, be a continuous function, then it  can be assimilated to  a $2\pi-$ periodic function $g:\RR\to\RR$ via the relation
$$
f(w)=g(\eta),\quad w=e^{i\eta}.
$$
Hence when  $f$ is  smooth enough we get
$$
f^\prime(w)\triangleq\frac{df}{dw}=-ie^{-i\eta}g'(\eta).
$$  
Because  $d/dw$ and $d/d\eta$ differ only by a smooth factor with modulus one  we shall  in the sequel work with $d/dw$ instead of $d/d\eta$ which appears more suitable in the computations.\\
Moreover, if $f$ has real Fourier coefficients and is of class $C^1$ then  we can easily check that 
\begin{equation}\label{deriv-bar}
{\{\overline{f}\}^\prime}(w)=-\frac{1}{w^2}\overline{f^\prime(w)}.
\end{equation}
Now we shall introduce  H\"older spaces  on the unit circle $\mathbb{T}$.
\begin{definition}
Let 
$0<\gamma<1$. We denote by $C^\gamma(\mathbb{T}) $  the space of continuous functions $f$ such that
$$
\Vert f\Vert_{C^\gamma(\mathbb{T})}\triangleq \Vert f\Vert_{L^\infty(\mathbb{T})}+\sup_{\tau\neq w\in \mathbb{T}}\frac{\vert f(\tau)-f(w)\vert}{\vert \tau-w\vert^\alpha}<\infty.
$$
For any integer $n$, the space $C^{n+\gamma}(\mathbb{T})$ stands for the set of functions $f$ of class $C^n$ whose $n-$th order derivatives are H\"older continuous  with exponent $\gamma$. It is equipped with the usual  norm,
$$
\Vert f\Vert_{C^{n+\gamma}(\mathbb{T})}\triangleq \Vert f\Vert_{L^\infty(\mathbb{T})}+\Big\Vert \frac{d^n f}{dw^n}\Big\Vert_{C^\gamma(\mathbb{T})}.
$$ 
\end{definition}
Recall that the Lipschitz semi-norm is defined by,
$$
\|f\|_{\textnormal{Lip}(\mathbb{T})}=\sup_{\tau\neq w\in\mathbb{T}}\frac{|f(\tau)-f(w)|}{|\tau-w|}\cdot
$$
Now we list some classical properties that will be used later at many places.
\begin{enumerate}
\item For $n\in \mathbb{N}, \gamma\in ]0,1[$ the space $C^{n+\gamma}(\mathbb{T})$ is an algebra.
\item For $K\in L^1(\mathbb{T})$ and $f\in C^{n+\gamma}(\mathbb{T})$ we have the convolution law,
$$
\|K*f\|_{C^{n+\gamma}(\mathbb{T})}\le \|K\|_{L^1(\mathbb{T})}\|f\|_{C^{n+\gamma}(\mathbb{T})}.
$$
\end{enumerate}

 \subsection{Elements of the bifurcation theory}
 We intend now  to recall Crandall-Rabinowitz's Theorem which is a basic tool of the  bifurcation theory and  will be useful in the proof of \mbox{Theorem \ref{main}.} Let $F:\RR\times X\to Y$ be a  continuous function with $X$ and $Y$ being   two Banach spaces. Assume that $F(\lambda, 0)=0$ for any $\lambda$ belonging in a non empty interval $I.$ Whether close to a trivial solution  $(\lambda_0, 0)$ we can find a branch of non trivial solutions of  the equation  $
 F(\lambda,x)=0
 $ is the main concern of  the bifurcation  theory. If this happens we say that we have a bifurcation at the point $(\lambda_0, 0)$. We shall restrict ourselves here to  the classical result of Crandall and Rabinowitz  \cite{CR}. For more general results  we refer the reader to  the book of Kielh\"{o}fer \cite{Kil}. 
 
\begin{theorem}\label{C-R} Let $X, Y$ be two Banach spaces, $V$ a neighborhood of $0$ in $X$ and let 
$
F : \RR \times V \to Y
$
with the following  properties:
\begin{enumerate}
\item $F (\lambda, 0) = 0$ for any $\lambda\in \RR$.
\item The partial derivatives $F_\lambda$, $F_x$ and $F_{\lambda x}$ exist and are continuous.
\item $N(\mathcal{L}_0)$ and $Y/R(\mathcal{L}_0)$ are one-dimensional. 
\item {\it Transversality assumption}: $\partial_\lambda\partial_xF(0, 0)x_0 \not\in R(\mathcal{L}_0)$, where
$$
N(\mathcal{L}_0) = span\{x_0\}, \quad \mathcal{L}_0\triangleq \partial_x F(0,0).
$$
\end{enumerate}
If $Z$ is any complement of $N(\mathcal{L}_0)$ in $X$, then there is a neighborhood $U$ of $(0,0)$ in $\RR \times X$, an interval $(-a,a)$, and continuous functions $\varphi: (-a,a) \to \RR$, $\psi: (-a,a) \to Z$ such that $\varphi(0) = 0$, $\psi(0) = 0$ and
$$
F^{-1}(0)\cap U=\Big\{\big(\varphi(\xi), \xi x_0+\xi\psi(\xi)\big)\,;\,\vert \xi\vert<a\Big\}\cup\Big\{(\lambda,0)\,;\, (\lambda,0)\in U\Big\}.
$$
\end{theorem}

\subsection{Special functions} We shall  give a short  introduction on  the Gauss hypergeometric functions and discuss some  of their basic properties. The formulae listed below will be crucial in the computations of the linearized operator associated to the V-states equations.  Recall that for any real numbers $a,b\in \mathbb{R},\, c\in \mathbb{R}\backslash(-\mathbb{N})$ the hypergeometric function $z\mapsto F(a,b;c;z)$ is defined on the open unit disc $\mathbb{D}$ by the power series
$$
F(a,b;c;z)=\sum_{n=0}^{\infty}\frac{(a)_n(b)_n}{(c)_n}\frac{z^n}{n!}, \quad \forall z\in \mathbb{D}.
$$
Here, $(x)_n$ is the  Pochhammer symbol defined by,
$$
(x)_n = \begin{cases}   1   & n = 0 \\
  x(x+1) \cdots (x+n-1) & n \geq1.
 \end{cases}
 $$
 It is obvious that
\begin{equation}\label{f1s}
(x)_n=x\,(1+x)_{n-1},\quad (x)_{n+1}=(x+n)\,(x)_n.
\end{equation}
 For a future use we recall an integral representation of the  hypergeometric function, for instance see \cite[p. 47]{Rain}. Assume that  $ \textnormal{Re}(c) > \textnormal{Re}(b) > 0,$ then 
 \begin{equation}\label{integ}
 F(a,b;c;z)=\frac{\Gamma(c)}{\Gamma(b)\Gamma(c-b)}\int_0^1 x^{b-1} (1-x)^{c-b-1}(1-zx)^{-a}~\mathrm dx,\quad |z|<1.
 \end{equation}
 The function $\Gamma: \CC\backslash\{-\NN\} \to \CC$ refers to the gamma function which is the analytic continuation to the negative half plane of the usual gamma function defined on the positive \mbox{half-plane $\{\Rea\, z > 0\}$} by the integral representation:
 $$
 \Gamma(z)=\int_0^{+\infty}t^{z-1}e^{-t}\mathrm dt.
 $$
 It satisfies the relation
\begin{equation}\label{Gamma1}
\Gamma(z+1)=z\,\Gamma(z), \quad \forall z\in \CC \backslash(-\NN).
\end{equation}
From this we deduce the identities
\begin{equation}\label{Poc}
(x)_n=\frac{\Gamma(x+n)}{\Gamma(x)},\quad (x)_n=(-1)^n\frac{\Gamma(1-x)}{\Gamma(1-x-n)},
\end{equation}
provided all the quantities in the right terms are well-defined. Later we need the following values,
 \begin{equation}\label{for1}
 \Gamma(n+1)=n!,\quad \Gamma(1/2)=\sqrt{\pi}.
 \end{equation}
 Another useful identity is the Euler's reflection formula,
\begin{equation}\label{reflection}
\Gamma(1-z) \Gamma(z) = {\pi \over \sin{(\pi z)}},\quad \forall z\notin \mathbb{Z}.
\end{equation}
  Now we shall introduce  the digamma function which is nothing but the logarithmic derivative of the gamma  function and often  denoted by $\digamma$. It is given by
 $$
 \digamma(z)=\frac{\Gamma^\prime(z)}{\Gamma(z)},\quad z\in \CC \backslash(-\NN).
 $$
 The following identity is classical,
\begin{equation}\label{digam}
\forall n\in \NN,\quad \digamma(n+\frac12)=-\gamma-2\ln2+2\sum_{k=0}^{n-1}\frac{1}{2k+1}\cdot
\end{equation}

  When $\Rea (c-a-b)>0 $ then it can be shown that  the  hypergeometric series is absolutely convergent on the closed unit disc  and one has the 
  expression,\begin{equation}\label{id1}
F (a,b;c;1)= \frac{\Gamma(c)\Gamma(c-a-b)}{\Gamma(c-a)\Gamma(c-b)}\cdot 
\end{equation}
 the  proof can be found in \cite[p. 49]{Rain},

 Now recall the 
Kummer's quadratic transformation
\begin{equation}\label{FG}
F\Big(a,b;2b;\frac{4z}{(1+z)^2}\Big)=(1+z)^{2a}F\Big(a,a+\frac12-b;b+\frac12;z^2\Big), \quad\forall\,  z\in [0,1[.
\end{equation}
Next we recall some  contiguous function relations of the  hypergeometric series,  see \cite{Rain}.
\begin{equation}\label{f0}
c(c+1)\,F (a,b;c;z)-c(c+1)\,F (a,b;c+1;z)-ab\,zF (a+1,b+1;c+2;z)=0,
\end{equation}
\begin{equation}\label{f1ch2}
c\,F (a,b;c;z)-c\,F (a+1,b;c;z)+b\,zF (a+1,b+1;c+1;z)=0,
\end{equation}
\begin{equation}\label{f2}
c\,F (a,b;c;z)-c\,F (a,b+1;c;z)+a\,zF (a+1,b+1;c+1;z)=0,
\end{equation}
\begin{equation}\label{f3}
c\,F (a,b;c;z)-(c-b)\,F (a,b;c+1;z)-b\,F (a,b+1;c+1;z)=0,
\end{equation}
\begin{equation}\label{f4}
c\,F (a,b;c;z)-(c-a)\,F (a,b;c+1;z)-a\,F (a+1,b;c+1;z)=0,
\end{equation}
\begin{equation}\label{f34}
b\,F (a,b+1;c;z)-a\,F (a+1,b;c;z)+(a-b)\,F (a,b;c;z)=0,
\end{equation}
\begin{equation}\label{f5}
(b-a)(1-z)\,F (a,b;c;z)-(c-a)\,F (a-1,b;c;z)+(c-b)\,F (a,b-1;c;z)=0.
\end{equation}
Now we close this discussion with recalling Bessel function $J_n$ of the first kind of index $n\in \NN$ and review some important identities. It is defined in the full space $\mathbb{C}$  by the power series
$$
J_n(z)=\sum_{k\geq0}\frac{(-1)^k}{k!\,(n+k)!}\Big(\frac{z}{2}\Big)^{2k+n}.
$$
The following identity called Sonine-Schafheitlin's formula will be very useful later.
\begin{eqnarray}\label{sonine}
\nonumber\int_0^{+\infty}\frac{J_\mu(at)J_\nu(bt)}{t^\lambda}\hbox{d}t &=&\frac{ a^{\lambda-\nu-1}b^\nu\,\Gamma\Big(\frac12\mu+\frac12\nu-\frac12\lambda+\frac12\Big)}{2^\lambda\Gamma(\nu+1)\Gamma\Big(\frac12\lambda+\frac12\mu-\frac12\nu+\frac12\Big)}\\ &\times &F\Big(\frac{\mu+\nu-\lambda+1}{2},\frac{\nu-\lambda-\mu+1}{2};\nu+1;\frac{b^2}{a^2}\Big),
\end{eqnarray}
 provided that  $0<b<a$ and that the integral is convergent. A detailed proof of this result  can be found in \cite[p. 401]{Watson}.
 
\subsection{Basic integrals}
The main goal of this paragraph is to compute explicitly some integrals that will appear later  in the spectral study.
\begin{lemma}\label{lem} Let $\alpha,\, b\in (0,1)$ and $n\in \NN$. 
 Then for any $w\in \mathbb{T}$ we have the following formulae:
\begin{equation}\label{In}
\fint_{\mathbb{T}}\frac{\tau^{n-1}}{|w-b\tau|^\alpha}d\tau\,=\, w^{n}b^n\frac{(\frac\alpha2)_n}{n!}F\Big(\frac\alpha2,n+\frac\alpha2;n+1;b^2\Big).\qquad\;\;
\end{equation}
\begin{align}\label{Jn}
\fint_\mathbb{T}\frac{(w-b\tau)(a w^{n}-c\tau^{n})}{\vert w-b\tau\vert^{\alpha+2}}d\tau =  w^{n+2}b\bigg[&a\Big(1+\frac\alpha2\Big)F\Big(\frac\alpha2,2+\frac\alpha2;2;b^2\Big)\\ &- c\,b^{n}\frac{(1+\frac\alpha2)_{n+1}}{(n+1)!}F\Big(\frac\alpha2,n+2+\frac\alpha2;n+2;b^2\Big)\bigg]\notag.
\end{align}
\begin{align}\label{Kn}
\fint_\mathbb{T}\frac{(\overline{w}-b\overline{\tau})(a\overline{w}^n-c\overline{\tau}^{n})}{\vert w-b\tau\vert^{\alpha+2}}d\tau=\overline{w}^{n}\bigg[&ab\frac\alpha2\,F(\frac\alpha2+1,\frac\alpha2+1;2;b^2)\\ &- c\,b^{n-1}\frac{(1+\frac\alpha2)_{n-1}}{(n-1)!}F\Big(\frac\alpha2,n+\frac\alpha2;n;b^2\Big)\bigg]\notag.\qquad\qquad
\end{align}
\begin{align}\label{Ln}
\fint_\mathbb{T}\frac{(bw-\tau)(a w^{n}-c\tau^{n})}{\vert bw-\tau\vert^{\alpha+2}}d\tau =  -w^{n+2}b^2\bigg[&a\frac\alpha4\Big(\frac\alpha2+1\Big)F\Big(1+\frac\alpha2,2+\frac\alpha2;3;b^2\Big)\\&- c\,b^{n}\frac{(\frac\alpha2)_{n+2}}{(n+2)!}F\Big(1+\frac\alpha2,n+2+\frac\alpha2;n+3;b^2\Big)\bigg].\notag
\end{align}
\begin{align}\label{Mn}
\fint_\mathbb{T}\frac{(b\overline{w}-\overline{\tau})(a\overline{w}^n-c\overline{\tau}^{n})}{\vert bw-\tau\vert^{\alpha+2}}d\tau =-\overline{w}^{n}\bigg[&aF\Big(\frac\alpha2,\frac\alpha2+1;1;b^2\Big)\\ &- c\,b^{n}\frac{(\frac\alpha2)_{n}}{n!}F\Big(\frac\alpha2+1,n+\frac\alpha2;n+1;b^2\Big)\bigg].\notag\qquad\quad
\end{align}

\end{lemma}
\begin{proof}
To prove the first identity we use successively   the change of variables $\tau=w \zeta$ \mbox{and $\zeta=e^{i\eta}$,}
\begin{eqnarray*}
\fint_{\mathbb{T}}\frac{\tau^{n-1}}{|w-b\tau|^\alpha}d\tau&=& w^{n}\fint_{\mathbb{T}}\frac{\zeta^{n-1}}{|1-b\zeta|^\alpha}d\zeta\\
&=&w^{n}\frac{1}{2(1+b)^\alpha\pi}\int_{0}^{2\pi}\frac{e^{in\eta}}{\big(1-\frac{4b}{(1+b)^2}\cos^2(\eta/2)\big)^{\frac{\alpha}{2}}}d\eta.
\end{eqnarray*}
Again by the change of variables $\eta/2\mapsto\eta$   one gets 
\begin{eqnarray*}
\fint_{\mathbb{T}}\frac{\tau^{n-1}}{|w-b\tau|^\alpha}d\tau
&=&w^{n}\frac{1}{(1+b)^{\alpha}\pi}\int_{0}^{\pi}\frac{e^{i2n\eta}}{\big(1-\frac{4b}{(1+b)^2}\cos^2\eta\big)^{\frac{\alpha}{2}}}d\eta.
\end{eqnarray*}
Since $\big|\frac{4b}{(1+b)^2}\cos^2\eta\big|<1$ then  we can use  the   Taylor series 
\begin{eqnarray*}
\Big(1-\frac{4b}{(1+b)^2}\cos^2\eta\Big)^{-\alpha/2} &=& \sum_{m=0}^{\infty}\frac{\big(\alpha/2\big)_m}{m!}\frac{2^{2m}b^m}{(1+b)^{2m}} \cos^{2m}\eta.
\end{eqnarray*}

 Consequently, we get
\begin{eqnarray*}
\fint_{\mathbb{T}}\frac{\tau^{n-1}}{|w-b\tau|^\alpha}d\tau
&=&w^{n}\frac{1}{\pi(1+b)^{\alpha}}\sum_{m=0}^{\infty}\frac{\big(\alpha/2\big)_m}{m!}\frac{2^{2m}b^m}{(1+b)^{2m}}\int_{0}^{\pi}\cos^{2m}\eta e^{i2n\eta} d\eta.
\end{eqnarray*}
We shall now recall the following identity, see for instance \cite[p. 8]{M-O} and \cite[p. 449]{Watson},
 \begin{equation*}\label{lem00}
\int_{0}^{\pi}\cos^{x}(\eta) e^{iy\eta}d\eta=\frac{\pi                                                                                                                                                                                                                                                                                                                                                                                                                                                                                                                                                                                                                                                                                                                                                                                                                                                                                                                                                                                                                                                                                                                                                                                                                                                                                                                                                                                                                                                                                                                                                                                                                                                                                                                                                                                                                                                                                                                                                                                                                                                                                                                                                                                                                                                                                                                                                                                                                        \Gamma(x+1)}{2^{x}\Gamma\Big(1+\dfrac{x+y}{2}\Big)\Gamma\Big(1+\dfrac{x-y}{2}\Big)},\quad \forall\, x>-1,\quad\forall\,  y\in \RR.
\end{equation*}
As it  was pointed before the gamma function  has no real zeros  but simple poles located at $-\NN$ and therefore   the function $\frac{1}{\Gamma}$ admits an analytic continuation on $\CC.$ Apply this formula  with $x=2m$ and $y=2n$ yields,
\begin{equation*}\label{imp}
\frac{1}{\pi}\int_{0}^{\pi}\cos^{2m}(\eta) e^{2in\eta}d\eta=\frac{\Gamma(2m+1)}{2^{2m}\Gamma(m+n+1)\Gamma(m-n+1)}\cdot
\end{equation*}

Hence, 
\begin{eqnarray*}
\fint_{\mathbb{T}}\frac{\tau^{n-1}}{|w-b\tau|^\alpha}d\tau
&=&w^{n}\frac{1}{(1+b)^{\alpha}}\sum_{m=n}^{\infty}\frac{\big(\alpha/2\big)_m}{m!}\frac{\Gamma(2m+1)}{\Gamma(m+n+1)\Gamma(m-n+1)}\frac{b^m}{(1+b)^{2m}}\\
 &=&w^{n}\frac{1}{(1+b)^{\alpha}}\sum_{m=0}^{\infty}\frac{\big(\alpha/2\big)_{m+n}}{(m+n)!}\frac{\Gamma(2m+2n+1)}{\Gamma(m+2n+1)\Gamma(m+1)}\frac{b^{m+n}}{(1+b)^{2(m+n)}}\cdot
\end{eqnarray*}
We shall use Legendre's  duplication formula,
$$
\Gamma(z) \; \Gamma\left(z + \frac{1}{2}\right) = 2^{1-2z} \; \sqrt{\pi} \; \Gamma(2z),\quad z\in \mathbb{C}\backslash \{-\NN\},
$$
which gives
\begin{eqnarray*}
\frac{\Gamma(2m+2n+1)}{(m+n)!}&=&\frac{\Gamma(2m+2n+1)}{\Gamma(m+n+1)}\\&=&\frac{2^{2m+2n}}{\sqrt{\pi}}\Gamma(m+n+1/2).
\end{eqnarray*}
Therefore using the identity \eqref{Poc} and  $\Gamma(1/2)=\sqrt{\pi}$ we find
\begin{eqnarray*}
\frac{\Gamma(2m+2n+1)}{(m+n)!}&=&2^{2m+2n}(1/2)_{m+n}.
\end{eqnarray*}
From the elementary fact  $$ (x)_{m+n}=(x)_n(n+x)_m\quad \forall x\in \RR$$ 
one deuces
\begin{eqnarray*}
\fint_{\mathbb{T}}\frac{\tau^{n-1}}{|w-b\tau|^\alpha}d\tau
 &=&w^{n}\frac{(\frac\alpha2)_n(\frac12)_{n}}{(1+b)^{\alpha}(2n)!}\Big(\frac{4b}{(1+b)^2}\Big)^{n}\sum_{m=0}^{\infty}\frac{(n+\frac\alpha2)_m(n+1/2)_{m}}{(2n+1)_mm!}\frac{2^{2m}b^m}{(1+b)^{2m}}\cdot
\end{eqnarray*}
By definition of the  hypergeometric series  we conclude that
\begin{eqnarray*}
\fint_{\mathbb{T}}\frac{\tau^{n-1}}{|w-b\tau|^\alpha}d\tau
 &=&w^{n}\frac{(\frac\alpha2)_n}{n!}\frac{2^{2n}b^n}{(1+b)^{2n+\alpha}}F\Big(n+\frac\alpha2,n+\frac12;2n+1;\frac{4b}{(1+b)^2}\Big)\cdot
 \end{eqnarray*}
Using Kummer's quadratic transformation \eqref{FG} the last identity becomes
\begin{eqnarray*}
\fint_{\mathbb{T}}\frac{\tau^{n-1}}{|w-b\tau|^\alpha}d\tau
 &=&w^{n}b^n\frac{(\frac\alpha2)_n}{n!}F\Big(\frac\alpha2,n+\frac\alpha2;n+1;b^2\Big).
\end{eqnarray*}

This completes the proof of \eqref{In}.

We intend now to compute the second  integral \eqref{Jn}. To this end we use the  change of variables as before,
\begin{eqnarray}\label{jnn1}
I_n &\triangleq &  \fint_\mathbb{T}\frac{(w-b\zeta)(a w^n-c\zeta^{n})}{\vert w-b\zeta\vert^{\alpha+2}}d\zeta\notag \\ &=&w^{n+2}\fint_\mathbb{T}\frac{(1-b\zeta)(a-c\zeta^{n})}{\vert 1-b\zeta\vert^{\alpha+2}}d\zeta\notag\\ &=& w^{n+2}\Big(a\,A_0-c\,A_n\Big)\end{eqnarray}
where
\begin{eqnarray*}
A_n &\triangleq&\fint_\mathbb{T}\frac{\zeta^n d\zeta}{\vert 1-b\zeta\vert^{\alpha+2}}-b\fint_\mathbb{T}\frac{\zeta^{n+1} d\zeta}{\vert 1-b\zeta\vert^{\alpha+2}}.
\end{eqnarray*}
It follows from the identity \eqref{In} that
\begin{align*}
A_n = b^{n+1}\frac{(1+\frac\alpha2)_{n+1}}{(n+1)!}\bigg[ & F\Big(\frac\alpha2+1,n+2+\frac\alpha2;n+2;b^2\Big)\\ &-\frac{(\frac\alpha2+n+2)}{n+2}b^2F\Big(\frac\alpha2+1,n+3+\frac\alpha2;n+3;b^2\Big)\,\,\bigg].
\end{align*}
Then, in view of the formula \eqref{f1ch2} one gets
\begin{eqnarray*}
A_n &=& b^{n+1}\frac{(1+\frac\alpha2)_{n+1}}{(n+1)!}F\Big(\frac\alpha2,n+2+\frac\alpha2;n+2;b^2\Big).
\end{eqnarray*}
Replacing $A_n$ by its expression in \eqref{jnn1} we conclude that
\begin{eqnarray*}
I_n = w^{n+2}\Bigg[ab\big(1+\frac\alpha2\big)F\Big(\frac\alpha2,2+\frac\alpha2;2;b^2\Big)-cb^{n+1}\frac{(1+\frac\alpha2)_{n+1}}{(n+1)!}F\Big(\frac\alpha2,n+2+\frac\alpha2;n+2;b^2\Big)\Bigg].\end{eqnarray*}

We shall now compute  the   integral \eqref{Kn}. We write
\begin{eqnarray*}
K_n &\triangleq & \fint_\mathbb{T}\frac{(\overline{w}-b\overline{\tau})(a\overline{w}^n-c\overline{\tau}^{n})}{\vert 1-b\tau\vert^{\alpha+2}}d\tau\\ &=&\overline{w}^{n} \fint_\mathbb{T}\frac{(1-b\overline{\zeta})(a-c\overline{\zeta}^{n})}{\vert 1-b\zeta\vert^{\alpha+2}}d\zeta\\ &\triangleq& \overline{w}^{n}\Big(aB_0-cB_n\Big) .
\end{eqnarray*}
Using the identity \eqref{In}, $B_0$ can be rewritten as
\begin{eqnarray*}
B_0 &\triangleq&\fint_\mathbb{T}\frac{ d\zeta}{\vert 1-b\zeta\vert^{\alpha+2}}-b\fint_\mathbb{T}\frac{\overline{\zeta} d\zeta}{\vert 1-b\zeta\vert^{\alpha+2}}\\ 
&=&\fint_\mathbb{T}\frac{ d\zeta}{\vert 1-b\zeta\vert^{\alpha+2}}-b\fint_\mathbb{T}\frac{{\zeta^{-1}} d\zeta}{\vert 1-b\zeta\vert^{\alpha+2}}
\\&=& b(1+\frac\alpha2)F\Big(\frac\alpha2+1,2+\frac\alpha2;2;b^2\Big)-bF\Big(\frac\alpha2+1,\frac\alpha2+1;1;b^2\Big).
\end{eqnarray*}
Then, in view of the formula \eqref{f3} we get
\begin{eqnarray*}
B_0 &=& \frac\alpha2bF\Big(\frac\alpha2+1,\frac\alpha2+1;2;b^2\Big).
\end{eqnarray*}
To compute $B_n$ we first observe that by the change of variables $\overline\zeta\mapsto \zeta$ we find
\begin{eqnarray*}
 \fint_\mathbb{T}\frac{\overline{\zeta}^n }{\vert 1-b\zeta\vert^{\alpha+2}}d\zeta=\fint_\mathbb{T}\frac{\zeta^{n-2} }{\vert 1-b\zeta\vert^{\alpha+2}}d\zeta
\end{eqnarray*}
which yields in turn
\begin{eqnarray*}
B_n  &=& A_{n-2}\\ &=& b^{n-1}\frac{(1+\frac\alpha2)_{n-1}}{(n-1)!}F\Big(\frac\alpha2,n+\frac\alpha2;n;b^2\Big).
\end{eqnarray*}
Consequently,
\begin{eqnarray*}
K_n &=& \overline{w}^{n}\bigg[a\,b\frac\alpha2F(\frac\alpha2+1,\frac\alpha2+1;2;b^2)-c\,b^{n-1}\frac{(1+\frac\alpha2)_{n-1}}{(n-1)!}F\Big(\frac\alpha2,n+\frac\alpha2;n;b^2\Big)\bigg] .
\end{eqnarray*}
Concerning the  integral \eqref{Ln} we use a change of variable as before in order to get
\begin{eqnarray}\label{jn1}
P_n &\triangleq &  \fint_\mathbb{T}\frac{(bw-\zeta)(a w^n-c\zeta^{n})}{\vert bw-\zeta\vert^{\alpha+2}}d\zeta\notag \\ &=&w^{n+2}\fint_\mathbb{T}\frac{(b-\zeta)(a-c\zeta^{n})}{\vert b-\zeta\vert^{\alpha+2}}d\zeta\notag\\ &=& w^{n+2}\Big(a\,C_0-c\,C_n\Big),
\end{eqnarray}
with
\begin{eqnarray*}
C_n &\triangleq&b\fint_\mathbb{T}\frac{\zeta^n d\zeta}{\vert b-\zeta\vert^{\alpha+2}}-\fint_\mathbb{T}\frac{\zeta^{n+1} d\zeta}{\vert b-\zeta\vert^{\alpha+2}}\cdot
\end{eqnarray*}
Observe that
$$
|b-\zeta|=|1- b\zeta | \quad  \forall \zeta\in\mathbb{T}.
$$
Then, it follows from the formula \eqref{In} that
\begin{align*}
C_n = b^{n+2}\frac{(1+\frac\alpha2)_{n+1}}{(n+2)!}\bigg[ &(n+2)F\Big(\frac\alpha2+1,n+2+\frac\alpha2;n+2;b^2\Big)\\ &-{\Big(\frac\alpha2+n+2\Big)}F\Big(\frac\alpha2+1,n+3+\frac\alpha2;n+3;b^2\Big)\bigg].
\end{align*}
Using once again the identity \eqref{f3} implies
\begin{eqnarray}\label{mm1}
C_n &=& -b^{n+2}\frac{(\frac\alpha2)_{n+2}}{(n+2)!}F\Big(1+\frac\alpha2,n+2+\frac\alpha2;n+3;b^2\Big).
\end{eqnarray}
Plugging the latter expression of  $C_n$ into  \eqref{jn1} yields
\begin{align*}
P_n = -w^{n+2}\bigg[&ab^2\frac\alpha4\Big(1+\frac\alpha2\Big)F\Big(1+\frac\alpha2,2+\frac\alpha2;3;b^2\Big)\\ &- cb^{n+2}\frac{(\frac\alpha2)_{n+2}}{(n+2)!}F\Big(1+\frac\alpha2,n+2+\frac\alpha2;n+3;b^2\Big)\bigg].
\end{align*}
We shall now move to the computation of the last  integral \eqref{Mn},
\begin{eqnarray*}
Q_n &\triangleq &  \fint_\mathbb{T}\frac{(b\overline{w}-\overline{\tau})(a\overline{w}^n-c\overline{\tau}^{n})}{\vert bw-\tau\vert^{\alpha+2}}d\tau\\ &=&\overline{w}^{n} \fint_\mathbb{T}\frac{(b-\overline{\zeta})(a-c\overline{\zeta}^{n})}{\vert 1-b\zeta\vert^{\alpha+2}}d\zeta\\ &\triangleq& \overline{w}^{n}\Big(aD_0-c\, D_n\Big).
\end{eqnarray*}
From the identity \eqref{In} we may write
\begin{eqnarray*}
D_0 &\triangleq&b\fint_\mathbb{T}\frac{ d\zeta}{\vert 1-b\zeta\vert^{\alpha+2}}-\fint_\mathbb{T}\frac{\overline{\zeta} d\zeta}{\vert 1-b\zeta\vert^{\alpha+2}}\\ &=& b^2\Big(1+\frac\alpha2\Big)F\Big(\frac\alpha2+1,2+\frac\alpha2;2;b^2\Big)-F\Big(\frac\alpha2+1,\frac\alpha2+1;1;b^2\Big).
\end{eqnarray*}
Therefore  by the  formula \eqref{f1ch2} we obtain
\begin{eqnarray*}
D_0 &=&-F\Big(\frac\alpha2,\frac\alpha2+1;1;b^2\Big).
\end{eqnarray*}
To compute $D_n$ we write through a change of variables, 
\begin{eqnarray*}
D_n &\triangleq&b\fint_\mathbb{T}\frac{\overline{\zeta}^n d\zeta}{\vert 1-b\zeta\vert^{\alpha+2}}\,-\,\fint_\mathbb{T}\frac{\overline{\zeta}^{n+1} d\zeta}{\vert 1-b\zeta\vert^{\alpha+2}}\\ &=& b\fint_\mathbb{T}\frac{\zeta^{n-2} d\zeta}{\vert 1-b\zeta\vert^{\alpha+2}}-\fint_\mathbb{T}\frac{\zeta^{n-1} d\zeta}{\vert 1-b\zeta\vert^{\alpha+2}}\\ &=& C_{n-2},
\end{eqnarray*}
which implies in view of \eqref{mm1}
\begin{eqnarray*}
D_n = - b^{n}\frac{(\frac\alpha2)_{n}}{n!}F\Big(1+\frac\alpha2,n+\frac\alpha2;n+1;b^2\Big).
\end{eqnarray*}
Hence we find
\begin{eqnarray*}
Q_n &=& \overline{w}^{n}\bigg[-a F\Big(\frac\alpha2,\frac\alpha2+1;1;b^2\Big)+c\, b^{n}\frac{(\frac\alpha2)_{n}}{n!}F\Big(\frac\alpha2+1,n+\frac\alpha2;n+1;b^2\Big)\bigg]
\end{eqnarray*}
and therefore  the proof of the lemma is now complete.
\end{proof}

\section{Regularity of the nonlinear functional}
This section is devoted to the regularity study of the  nonlinear functional $G$ introduced \mbox{in \eqref{Rotaeq1}} and which defines the V-states equations  . We shall  check the regularity assumptions required by Crandall-Rabinowitz's Theorem. The computations are vey heavy  and can  be done in a straightforward way without new difficulties compared to the simply connected case treated in the paper \cite{H-H}. Many of the details may be found in that work and will not be reiterated here. Therefore for the sake of concise presentation we   shall study  the new terms involving the interaction between the boundaries. However, regarding the self-induced terms we   only recall the results from the paper  \cite{H-H}.  To begin with, we  introduce the function spaces that we shall use. 
We set,  
\begin{equation}\label{x}
X=C_{ar}^{2-\alpha}(\mathbb{T})\times C_{ar}^{2-\alpha}(\mathbb{T}),\quad Y=H\times H,
\end{equation}
with
$$
C_{ar}^{2-\alpha}(\mathbb{T})=\bigg\{ f\in C^{2-\alpha}(\mathbb{T}); f(w)=\sum_{n\geq 1} a_n\overline{w}^n, w\in \mathbb{T}, a_n\in \RR, n\in\NN^\star\bigg\}
$$
and
$$
H=\bigg\{ g\in C^{1-\alpha}(\mathbb{T}); g(w)=\frac{i}{2}\sum_{n\geq 1} a_n(w^n-\overline{w}^n), w\in \mathbb{T}, a_n\in \RR, n\in\NN\bigg\}.
$$
For  $b\in(0,1)$, let $V$ denote  the product  $B_r\times B_r$, where $B_r$ is  the open ball of $X$ with center $0$ and radius $\displaystyle{r=(1/4)\min\{b,1-b\}}$. We note that this choice is done in order to guarantee that $\phi_1=\hbox{Id}+f_1$ and $\phi_2=b\,\hbox{Id}+f_2$ are conformal for  $f_1, f_2\in B_r$ and to prevent the intersection between the curves $\phi_1(\mathbb{T})$ and $\phi_2(\mathbb{T})$ which represent the boundaries of the V-states.\\
Now recall from \eqref{Rotaeq1} the form   of  the functional $G=(G_1,G_2)$,
 \begin{equation*}\label{g_2}
G_j\big(\Omega,f_1,f_2\big)\triangleq\textnormal{Im}\bigg\{\Big(\Omega\,\phi_j(w)+ S({\phi_2},\phi_j)(w)-S({\phi_1},\phi_j)(w)\Big)\overline{w}{\overline{\phi_j'(w)}}\bigg\},\quad w\in\mathbb{T},\, j=1,2,
\end{equation*} 
where   $S$ is defined  by
\begin{equation}\label{S}
S({\phi_i},\phi_j)(w)= C_\alpha\fint_\mathbb{T}\frac{\phi_i'(\tau)d\tau}{\vert \phi_j(w)-\phi_i(\tau)\vert^\alpha},\quad  i,j=1,2.
\end{equation}
We shall  rewrite $G_j$  as follows,
\begin{equation*}
G_j\big(\Omega,f_1,f_2\big)=L_j\big(\Omega,f_j\big)+N_{j}\big(f_1,f_2\big),\quad  j=1,2,\quad
\end{equation*} 
with
\begin{equation*}
 L_j\big(\Omega,f_j\big)\triangleq\textnormal{Im}\bigg\{\Big(\Omega\,\phi_j(w)+(-1)^jS(\phi_j,\phi_j)(w)\Big)\overline{w}\,{\overline{\phi_j'(w)}}\bigg\},
\end{equation*} 
and
\begin{equation*}
\; N_{j}\big(f_1,f_2\big)\triangleq (-1)^{j-1}\textnormal{Im}\bigg\{ S({\phi_i},\phi_j)(w)\overline{w}\,{\overline{\phi_j'(w)}}\bigg\}, \qquad i\neq j,
\end{equation*} 
usually with  the notation $\phi_1=\hbox{Id}+f_1$, $\phi_2=b\, \hbox{Id}+f_2$.\\

We propose to  prove  the following result concerning the regularity of $G.$
\begin{proposition}\label{reg}
The following holds true.
\begin{enumerate}
\item $G: \RR\times V\to Y$ is well-defined.
\item $G: \RR\times V\to Y$ is of class $C^1.$
\item The partial derivative  $\partial_\Omega DG: \mathbb{R}\times V\to \mathcal{L}(X, Y)$  exists and is continuous.
 \end{enumerate}
\end{proposition}
\begin{proof}

Notice that the terms $L_j,j=1,2$ appears modulo the sign of $(-1)^j$ in the simply connected case discussed in the paper \cite{H-H} and all the computations were done there. Therefore we shall restrict ourselves to recalling just  the results of those computations:
\begin{enumerate}
\item $L_j: \RR\times B_r\to H$ is well-defined.
\item $L_j: \RR\times B_r\to H$ is of class $C^1$.
 \end{enumerate}
Moreover the differential $DL_j$  is given  for $f_j\in B_r, h_j\in C_{ar}^{2-\alpha}(\mathbb{T})$ by 

\begin{align}\label{dlj}
\nonumber  DL_j(\Omega,f_j)h_j(w) =&\nonumber   \textnormal{Im} \Bigg\{\Omega\,\Big(\overline{w}\,{\overline{h_j'(w)}}\,\phi_j(w)+\overline{w}\,{\overline{\phi_j'(w)}}\,h_j(w)\Big)\\ \nonumber &+\nonumber (-1)^j  \overline{w}\,{\overline{h_j'(w)}}\,S({\phi_j},\phi_j)(w)+(-1)^jC_\alpha\overline{w}\overline{\phi_j^\prime(w)}\fint_\mathbb{T}\frac{h_j^\prime(\tau)}{\vert \phi_j(w)-\phi_j(\tau)\vert^\alpha}d\tau\\ &+(-1)^{j+1}{\alpha }C_\alpha\overline{w}\overline{\phi_j^\prime(w)}\fint_\mathbb{T}\frac{\textnormal{Re}\big[\big(\phi_j(w)-\phi_j(\tau)\big)\big(\overline{h_j(w)}-\overline{h_j(\tau)}\big)\big]\phi_j^\prime(\tau)}{\vert \phi_j(w)-\phi_j(\tau)\vert^{\alpha+2}}d\tau\Bigg\}.
\end{align}
In addition,  the partial derivative  $\partial_\Omega DL_j: \mathbb{R}\times B_r\to \mathcal{L}(C_{ar}^{2-\alpha}(\mathbb{T}), H)$  exists and is continuous. It is given by the formula, 
\begin{equation*}
\partial_\Omega DL_j(\Omega,f_j)h_j(w) =\textnormal{Im} \Big\{\overline{w}\,{\overline{h_j'(w)}}\,\phi_j(w)+\overline{w}\,{\overline{\phi_j'(w)}}\,h_j(w)\Big\}.
\end{equation*}
If we prove the regularity properties for  the second part $N_j$ then we can easily deduce that
\begin{eqnarray*}
\partial_\Omega DG_j(\Omega,f_1,f_2)(h_1,h_2)(w)&=& \partial_\Omega DL_j(\Omega,f_j)h_j(w)\\
 &=&\textnormal{Im} \Big\{\overline{w}\,{\overline{h_j'(w)}}\,\phi_j(w)+\overline{w}\,{\overline{\phi_j'(w)}}\,h_j(w)\Big\}.
\end{eqnarray*}
 
 Therefore all the regularity assumptions are satisfied for the terms $L_j$ and to complete the proof of the proposition we should check these assumptions for $N_j$. More precisely, we shall prove  that   $N_{j}: V\to H$ is well-defined and  is of class $C^1.$

{$\bf(1)$} First, we shall prove  that  for $ (f_1,f_2)\in V$ we have $N_{j}(f_1,f_2)\in C^{1-\alpha}(\mathbb{T})$. Because   the space $C^{1-\alpha}(\mathbb{T})$  is an algebra the problem reduces to show that for $i\neq j$, the function 
 $
 S({\phi_i},\phi_j)$ belongs to $C^{1-\alpha}(\mathbb{T})$.
This can be deduced easily from the next  general result. Let $(f_1,f_2)\in V$ and $\phi_1=\hbox{Id}+f_1$, $\phi_2=b\, \hbox{Id}+f_2$ and define the operator 
$$
\mathcal{T}\chi(w)\triangleq \fint_\mathbb{T}\frac{\chi(\tau)}{\vert \phi_j(w)-\phi_i(\tau)\vert^\alpha}d\tau,\quad w\in \mathbb{T}.
$$
 Then
\begin{equation}\label{Singds}
\big\|\mathcal{T}\chi\big\|_{C^{1-\alpha}(\mathbb{T})}\le C\| \phi_j\|_{C^{1-\alpha}(\mathbb{T})}\|\chi\|_{L^\infty (\mathbb{T})}.
\end{equation}
The proof of this inequality will be done  in a straightforward way since as we shall see the kernel is not singular. This is due  to the fact that the inner and the outer boundaries do not intersect. Indeed,  for all $w,\tau\in \mathbb{T}$  we can write
\begin{eqnarray*}
\big\vert\phi_1(w)-\phi_2(\tau)\big\vert &\geq &\vert w-b\tau\vert-\big\vert f_1(w)\big\vert -\big\vert  f_2(\tau)\big\vert\notag \\ &\geq & \big(1-b)-\Vert f_1\Vert_{L^\infty} -\Vert  f_2\Vert_{L^\infty}> \frac{1-b}{2}\cdot
\end{eqnarray*}
The same result remains true if we change $\tau$ by $w$ and therefore we get for $i\neq j$
\begin{eqnarray}\label{low}
\big\vert\phi_i(w)-\phi_j(\tau)\big\vert &\geq &\frac{1-b}{2}\cdot
\end{eqnarray}
 It follows  that 
\begin{eqnarray*}
\big\vert \mathcal{T}\chi(w)\big\vert&\lesssim &\mathop{{\int}}_\mathbb{T}\frac{\big\vert\chi({\tau})\big\vert }{\vert \phi_j({w})-\phi_i({\tau})\vert^\alpha}|d{\tau}| \lesssim  \|\chi\|_{L^\infty(\mathbb{T})},
\end{eqnarray*}
which implies  that
$$
\|T\chi\|_{L^\infty(\mathbb{T})}\le C  \|\chi\|_{L^\infty(\mathbb{T})}.
$$
Next take $w_1\neq w_2\in \mathbb{T}$. Using  the inequality \eqref{low} gives
\begin{eqnarray*}
\big\vert \mathcal{T}\chi(w_1)-\mathcal{T}\chi(w_2)\big\vert&\lesssim & \mathop{{\int}}_\mathbb{T}\bigg\vert\frac{1}{\vert \phi_j({w_1})-\phi_i({\tau})\vert^\alpha}-\frac{1}{\vert \phi_j({w_2})-\phi_i({\tau})\vert^\alpha}\bigg\vert \big\vert\chi({\tau})\big\vert| d{\tau}|\\ &\lesssim & \|\chi\|_{L^\infty(\mathbb{T})}\mathop{{\int}}_\mathbb{T}\Big\vert\vert \phi_j({w_1})-\phi_i({\tau})\vert^\alpha-\vert \phi_j({w_2})-\phi_i({\tau})\vert^\alpha\Big\vert | d{\tau}|\\ &\lesssim &  \|\chi\|_{L^\infty(\mathbb{T})}\big\vert \phi_j({w_1})- \phi_j({w_2})\big\vert,
\end{eqnarray*}
where we have used in the last estimate the following  inequality:
 for $\alpha\in (0,1),$  there exists a constant $C>0,$ such that 
\begin{equation}\label{eq0}
\forall  a,b\in\RR^*_+,\quad \vert a^\alpha-b^\alpha\vert \leq C\frac{\vert a-b\vert}{a^{1-\alpha}+b^{1-\alpha}}\cdot
\end{equation}

 Finally, using the fact that  $\phi_i\in C^{2-\alpha}(\mathbb{T})\hookrightarrow C^{1-\alpha}(\mathbb{T})$  one can conclude that
\begin{equation*}
\big\vert \mathcal{T}\chi(w_1)-\mathcal{T}\chi(w_2)\big\vert\lesssim  \|\chi\|_{L^\infty(\mathbb{T})} \|\phi_j\|_{C^{1-\alpha}(\mathbb{T})}\vert w_1-w_2\vert^{1-\alpha},
\end{equation*}
which is the desired result. Now applying \eqref{Singds} to  the operator $S$ we get
\begin{eqnarray}\label{Sphi}
\nonumber\big\|S({\phi_i},\phi_j)\big\|_{C^{1-\alpha}(\mathbb{T})}&\le& C\| \phi_j\|_{C^{1-\alpha}(\mathbb{T})}\|\phi_i'\|_{L^\infty(\mathbb{T})}\\
&\leq& C\| \phi_j\|_{C^{2-\alpha}(\mathbb{T})}\|\phi_i\|_{C^{2-\alpha}(\mathbb{T})}.
\end{eqnarray}

To complete the  proof of the first point we  shall  verify that the Fourier coefficients of $N_{j}(f_1,f_2)$ belong to $i\RR$. From the definition of the space $X $ in \eqref{x} the mapping $\phi_j$ has real Fourier coefficients and thus the Fourier  coefficients of $\phi'_j$ are real too. Since this  property is stable under the conjugation and the  multiplication the problem reduces to prove that the Fourier coefficients of $S({\phi_i},\phi_j)$ are real.  For this last purpose, we  take the conjugate and make a change  of variables, 
\begin{eqnarray*}
\overline{S({\phi_i},\phi_j)(w)}&=&-C_\alpha\fint_\mathbb{T}\frac{\phi_i'(\overline{\tau})d\overline{\tau}}{\vert \phi_j(\overline{w})-\phi_i(\overline{\tau})\vert^\alpha}\\ &=& C_\alpha\fint_\mathbb{T}\frac{\phi_i'(\zeta)d\zeta}{\vert \phi_j(\overline{w})-\phi_i(\zeta)\vert^\alpha}\\ &=& S({\phi_i},\phi_j)(\overline{w}).
\end{eqnarray*}
This proves that  the Fourier coefficients of the  functions $S(\phi_i,\phi_j)$  are real and the proof of the first part $(1)$ is now achieved.

\vspace{0,5cm}
{\bf (2)} The strategy to get  that $N_{j}$ is of class  $C^1$ on $V$ consists first in checking the existence of its  G\^ateaux derivative. Second we  show that the G\^ateaux derivative is strongly continuous. This will ensure in the same time the existence of Fr\'echet derivative and its continuity. \\
The G\^ateaux derivative of the function  $N_{j}$ at $(f_1,f_2)$ in the direction  $h=(h_1,h_2)\in X$ is given by the formula
\begin{align}\label{gatderiv}
\nonumber DN_{j}(f_1,f_2)h &= D_{f_1}N_{j}(f_1,f_2)h_1\,\,+\,\, D_{f_2}N_{j}(f_1,f_2)h_2\\  &\triangleq \lim_{t\to 0}\frac{1}{t}\Big[N_{j}(f_1+th_1,f_2)-N_{j}(f_1,f_2)\Big]+\lim_{t\to 0}\frac{1}{t}\Big[N_{j}(f_1,f_2+th_2)-N_{j}(f_1,f_2)\Big],
\end{align}
where the limits are taken in the strong topology of $Y$. Thus we shall first prove the existence of these limit in the pointwise sense, that is for every point $w\in\mathbb{T},$ and after check  that these limits exist in the strong topology of $C^{1-\alpha}(\mathbb{T})$.

 Let us first check   for each point $(f_1,f_2)\in  V$ the existence of $D_{f_j} N_{j}(f_1,f_2)$ as a linear and bounded operator, that is, $D_{f_j} N_{j}(f_1,f_2)\in\mathcal{L}(C^{2-\alpha}_{ar}(\mathbb{T}),H)$. With the notation $\phi_1=\hbox{Id}+f_1$ and $\phi_2=b\,\hbox{Id}+f_2$, one has
\begin{align}\label{d1}
\nonumber  D_{f_j} N_{j}(f_1,f_2)h_j(w)=(-1)^{j-1}\textnormal{Im} \bigg\{& \overline{w}\,{\overline{h_j'(w)}} S({\phi_i},\phi_j)(w)\\
&+\overline{w}\,{\overline{\phi_j^\prime(w)}}\frac{d}{dt}_{\Big|t=0}S\big({\phi_i}(w),\phi_j(w)+th_j(w)\big)\bigg\}.
\end{align}
We shall make use of  the following identity: let $A\in \CC^\star$, $ B\in \CC$, $\alpha\in \RR$ and define the function $K:t\mapsto |A+Bt|^\alpha$ which is smooth close to zero, then we have
\begin{equation}\label{der1}
K^\prime(0)=\alpha |A|^{\alpha-2}\textnormal{Re}(\overline{A}B).
\end{equation}
 Combining this formula  with few easy computations  yield
 \begin{alignat*}{2}
\nonumber\frac{d}{dt}_{\Big|t=0}S({\phi_i},\phi_j+th_j)(w)=&-\frac{\alpha }{2}C_\alpha\bigg[&&{\overline{h_j(w)}}\fint_\mathbb{T}\frac{\big(\phi_j(w)-\phi_i(\tau)\big)\phi_i^\prime(\tau)}{\vert \phi_j(w)-\phi_i(\tau)\vert^{\alpha+2}}d\tau\\ & &&+{h_j(w)}\fint_\mathbb{T}\frac{\big(\overline{\phi_j(w)}-\overline{\phi_i(\tau)}\big)\phi_i^\prime(\tau)}{\vert \phi_j(w)-\phi_i(\tau)\vert^{\alpha+2}}d\tau\bigg]
\notag\\ 
\triangleq&-\frac{\alpha}{2}C_\alpha\Big[ &&{\overline{h_j(w)}}{A}_{i}\big(\phi_j\big)(w)+{h_j(w)}{B}_{i}\big(\phi_j\big)(w)\Big].
\end{alignat*}
Therefore,
 \begin{align}\label{gat22}
  D_{f_j} N_{j}(\Omega,f_1,f_2)h_j(w)= (-1)^{j-1}\,\textnormal{Im} \bigg\{& \overline{w}\,{\overline{h_j'(w)}} S\big({\phi_i},\phi_j\big)(w)\notag\\ &-\frac\alpha2C_\alpha\overline{w}\,{\overline{\phi_j^\prime(w)}}\Big[ {\overline{h_j(w)}}{A}_{i}\big(\phi_j\big)(w)+{h_j(w)}{B}_{i}\big(\phi_j\big)(w)\Big]\bigg\}.
\end{align}

Using the algebra structure of  $C^{1-\alpha}(\mathbb{T})$   combined with the estimate \eqref{Sphi}, we get
\vspace{0,2cm}
\begin{equation}\label{ess1}
\big\|D_{f_j} N_{j}(f_1,f_2)h_j\big\|_{C^{1-\alpha}(\mathbb{T})}
\lesssim \Big(1+\| {A}_{i}(\phi_j)\|_{C^{1-\alpha}(\mathbb{T})}+\| {B}_{i}\big(\phi_j\big)\|_{C^{1-\alpha}}\Big)\|h_j\|_{C^{2-\alpha}(\mathbb{T})}.
\end{equation}
It remains to  estimate the terms  ${A}_{i}(\phi_j)$ and ${B}_{i}(\phi_j)$. For the first one, we get by virtue \mbox{of   \eqref{low},}
\begin{eqnarray*}
\big\vert {A}_{i}\big(\phi_j\big)(w)\big\vert&\lesssim & \int_\mathbb{T}\frac{\big\vert\phi_i^\prime(\tau)\big\vert}{\vert \phi_j(w)-\phi_i(\tau)\vert^{\alpha+1}}| d{\tau}|\\ &\lesssim & \|\phi_i'\|_{L^\infty(\mathbb{T})}.
\end{eqnarray*}
Now let $w_1\neq w_2\in \mathbb{T}$,  then
\begin{eqnarray}\label{A}
\nonumber\Big\vert {A}_{i}\big(\phi_j\big)(w_1)-{A}_{i}\big(\phi_j\big)(w_2)\Big\vert & \leq & \int_\mathbb{T}\Big\vert\frac{\overline{\phi_j(w_1)}-\overline{\phi_i(\tau)}}{\vert \phi_j({w_1})-\phi_i({\tau})\vert^{\alpha+2}}-\frac{\overline{\phi_j(w_2)}-\overline{\phi_i(\tau)}}{\vert \phi_j({w_2})-\phi_i({\tau})\vert^{2+\alpha}}\Big\vert  \vert\phi_i'(\tau)\vert | d{\tau}|\\ &\triangleq & \int_\mathbb{T}\big\vert K(w_1,\tau)-K(w_2,\tau)
\big\vert \vert\phi_i'(\tau)\vert | d{\tau}|.
\end{eqnarray}
Few easy computations  show that
\begin{eqnarray}\label{eqk}
\nonumber\big\vert K(w_1,\tau)- K(w_2,\tau)\big\vert 
  & \lesssim & \frac{\big\vert{\phi_j(w_2)}-{\phi_j(w_1)}\big\vert}{\vert \phi_j({w_1})-\phi_i({\tau})\vert^{\alpha+2}}\\ &+& \frac{\Big\vert \vert \phi_j({w_2})-\phi_i({\tau})\vert^{\alpha+2}-\vert \phi_j({w_1})-\phi_i({\tau})\vert^{\alpha+2}\Big\vert}{\vert \phi_j({w_1})-\phi_i({\tau})\vert^{\alpha+2}\vert \phi_j({w_2})-\phi_i({\tau})\vert^{\alpha+1}}\cdot
\end{eqnarray}
Concerning  the last term we shall use the following inequality whose proof is classical.
\begin{equation}\label{eq1}
\vert a^{k+1+\alpha}-b^{k+1+\alpha}\vert \leq C(k,\alpha)\vert a-b\vert \big( a^{k+\alpha}+b^{k+\alpha}\big),\quad a,b\in\RR_+,k\in \NN^*,\,0<\alpha .
\end{equation}
Hence, we get
\begin{align*}
 \frac{\Big\vert \vert \phi_j({w_2})-\phi_i({\tau})\vert^{\alpha+2}-\vert \phi_j({w_1})-\phi_i({\tau})\vert^{\alpha+2}\Big\vert}{\vert \phi_j({w_1})-\phi_i({\tau})\vert^{\alpha+2}\vert \phi_j({w_2})-\phi_i({\tau})\vert^{\alpha+1}} &\lesssim  \frac{\big\vert{\phi_j(w_2)}-{\phi_j(w_1)}\big\vert}{\vert \phi_j({w_1})-\phi_i({\tau})\vert^{\alpha+2}}\\ &+ \frac{\big\vert{\phi_j(w_2)}-{\phi_j(w_1)}\big\vert}{\vert \phi_j({w_1})-\phi_i({\tau})\vert\vert \phi_j({w_2})-\phi_i({\tau})\vert^{\alpha+1}}.
\end{align*}
Inserting this  in the estimate \eqref{eqk} and  using the inequality \eqref{low} we find
\begin{eqnarray*}
\big\vert K(w_1,\tau)- K(w_2,\tau)\big\vert  & \lesssim & \big\vert{\phi_j(w_2)}-{\phi_j(w_1)}\big\vert \\  & \lesssim &  \| \phi_j\|_{C^{1-\alpha}}\vert w_2-w_1\vert^{1-\alpha}.
\end{eqnarray*}
Now by plugging the latter estimate into  \eqref{A} one gets
\begin{eqnarray*}
\Big\vert {A}_{i}\big(\phi_j\big)(w_1)-{A}_{i}\big(\phi_j\big)(w_2)\Big\vert & \leq&  C\vert w_2-w_1\vert^{1-\alpha}
\end{eqnarray*}
which is the desired result. The estimate of the term ${B}_{i}\big(\phi_j\big)$ can be done in a similar way by observing that
$$
{B}_{i}\big(\phi_j\big)(w)=\fint_{\mathbb{T}} \overline{K(w,\tau)}\phi_i^\prime(\tau)d\tau.
$$
Consequently, from \eqref{ess1} we deduce that
$$
\|D_{f_j} N_{j}(f_1,f_2)h_j\|_{C^{1-\alpha}(\mathbb{T})}\le C\| h_1\|_{C^{2-\alpha}(\mathbb{T})}.
$$
This means  that $D_{f_j} N_j(f_1,f_2)\in \mathcal{L}(C^{1-\alpha}_{ar}(\mathbb{T}),H).$ 

\vspace{0,5cm}
Let us now move  to the computation of $D_{f_i} N_{j}(f_1,f_2)h_i$, for $i\neq j$ when $(f_1,f_2)\in V$ and  $h_i\in C^{2-\alpha}_{ar}(\mathbb{T})$. From the  definition, we obtain the formula
\begin{eqnarray*}
  D_{f_i} N_{j}(f_1,f_2)h_i(w)=(-1)^{j-1} \textnormal{Im} \bigg\{\overline{w}\,\overline{\phi_j^\prime(w)}\frac{d}{dt}_{\Big|t=0}S\big({\phi_i(w)+th_i(w)},\phi_j(w)\big)
 \bigg\}.
 \end{eqnarray*}
Some easy computations combined with the relation \eqref{der1} allow to get
 \begin{align*}\label{s2}
\nonumber\frac{d}{dt}_{\Big|t=0}S\big({\phi_i(w)+th_i(w)},\phi_j(w)\big)  =&\,\, C_\alpha\fint_\mathbb{T}\frac{h_i^\prime(\tau)}{\big\vert\phi_j(w)-\phi_i(\tau)\big\vert^\alpha}d\tau\\
\nonumber&+\frac\alpha2C_\alpha\fint_\mathbb{T}\frac{\big(\overline{\phi_j(w)}-\overline{\phi_i(\tau)}\big){h_i(\tau)}\phi_i^\prime(\tau)}{\vert \phi_j(w)-\phi_i(\tau)\vert^{\alpha+2}}d\tau\\ &+ \frac\alpha2C_\alpha\fint_\mathbb{T}\frac{\big(\phi_j(w)-\phi_i(\tau)\big)\overline{h_i(\tau)}\phi_i^\prime(\tau)}{\vert \phi_j(w)-\phi_i(\tau)\vert^{\alpha+2}}d\tau \notag\\ \triangleq &\,\,  C_i\big(\phi_j,h_i\big)(w)+D_i\big(\phi_j,h_i\big)(w)+ E_i\big(\phi_j,h_i\big)(w)\cdot
\end{align*}
It follows that,
 \begin{equation}\label{gat223}
D_{f_i} N_{j}(f_1,f_2)h_i(w)  = \textnormal{Im} \bigg\{\overline{w}{\overline{\phi_j^\prime(w)}}\Big[{C}_i\big(\phi_j,h_i\big)(w)+{D}_i\big(\phi_j,h_i\big)(w)+ E_i\big(\phi_j,h_i\big)(w)\Big]\bigg\}.
\end{equation}
Since $C^{1-\alpha}(\mathbb{T})$ is an algebra one finds that
$$
\big\|D_{f_j} N_{j}(\Omega,f_1,f_2)h_j\big\|_{C^{1-\alpha}(\mathbb{T})}\lesssim  \| C_i(\phi_j,h_i)\|_{C^{1-\alpha}}+\|D_i(\phi_j,h_i)\|_{C^{1-\alpha}(\mathbb{T})}+\|E_i(\phi_j,h_i)\|_{C^{1-\alpha}(\mathbb{T})}.
$$
The estimate of the term $C_i(\phi_j,h_i)$  follows immediately from \eqref{Singds} and we get,
$$
\| C_i(\phi_j,h_i)\|_{C^{1-\alpha}(\mathbb{T})}\lesssim \|\phi_j\|_{C^{1-\alpha}(\mathbb{T})}\|h^\prime_i\|_{L^\infty(\mathbb{T})}.
$$
For the terms $D_i(\phi_j,h_i)$ and $E_i(\phi_j,h_i)$ we can proceed  similarly  as for $A_i(\phi_j,h_j)$ and we find
$$
 \| D_i(\phi_j,h_i)\|_{C^{1-\alpha}(\mathbb{T})}+ \| E_i(\phi_j,h_i)\|_{C^{1-\alpha}(\mathbb{T})}\lesssim \|h_i\|_{L^\infty(\mathbb{T})}.
$$
Putting together the preceding estimates yields,
$$
\|D_{f_i} N_{j}(f_1,f_2)h_i\|_{C^{1-\alpha}(\mathbb{T})}\le C\| h_i\|_{C^{2-\alpha}(\mathbb{T})}.
$$
This shows  that $D_{f_i} N_j(f_1,f_2)\in \mathcal{L}(C^{1-\alpha}_{ar}(\mathbb{T}),H).$

 \vspace{0,5cm}
\quad To achieve the existence proof  of the G\^ateaux derivatives it remains to check that the convergence in \eqref{gatderiv}  occurs in the strong topology of $C^{1-\alpha}(\mathbb{T}).$ There are many terms to analyze and they  can be treated in a similar way. The computations are straightforward but slightly long and we prefer just to treat a significant term and the remaining ones are quite similar. For example in the first term of the right-hand side of \eqref{d1} we need  to check 
\begin{equation*}\label{conv}
\lim_{t\to 0}S\big({\phi_i},\phi_j+th_j\big)-S\big({\phi_i},\phi_j\big)=0\quad\textnormal{in}\quad C^{1-\alpha}(\mathbb{T}).
\end{equation*}
To simplify the notation we set
$$
T_{ij}(t,w)=S\big({\phi_i},\phi_j+th_j\big)(w)-S\big({\phi_i},\phi_j\big)(w).
$$Let $t>0$ such that $t\|h_j\|_{L^\infty}<(1-b)/2$. Then  by  \eqref{S} we have
\begin{eqnarray*}
T_{ij}(t,w)&=& \fint_\mathbb{T}\Big(\frac{\phi_i'(\tau)}{\vert \phi_j(w)-\phi_i(\tau)+th_j(w)\vert^\alpha}-\frac{\phi_i'(\tau)}{\vert \phi_j(w)-\phi_i(\tau)\vert^\alpha}\Big)d\tau\\ &\triangleq & \fint_\mathbb{T}K(t,w,\tau)\phi_i'(\tau)d\tau.
\end{eqnarray*}
It follows from the inequalities \eqref{low} and \eqref{eq0} that
$$
\big|K(t,w,\tau)\big|\lesssim  t\|h_j\|_{L^{\infty}(\mathbb{T})}
$$
which implies in turn,
\begin{eqnarray*}
\big|T_{ij}(t,w)\big| &\lesssim  t\|h_j\|_{L^{\infty}(\mathbb{T})}.
\end{eqnarray*}
Therefore we get
$$
\lim_{t\to 0}\big\|T_{ij}(t,\cdot)\big\|_{L^\infty(\mathbb{T})}=0.
$$
Now for $w_1\neq w_2\in\mathbb{T}$, we write by the Mean Value Theorem
\begin{eqnarray}\label{t}
\nonumber \big| T_{ij}(t,w_1)-T_{ij}(t,w_2)\big|&\lesssim &\nonumber \int_\mathbb{T}\Big| K(t,w_1,\tau)-K(t,w_2,\tau)\Big||d\tau|\\ &\lesssim &\big| w_1-w_2\big|\int_\mathbb{T}\sup_{w\in \mathbb{T}}\big| \partial_wK(t,w,\tau)\big||d\tau|.
\end{eqnarray}
Observe that $K(t,w,\tau)$ can be rewritten in the integral form
\begin{eqnarray*}
K(t,w,\tau)&=&\int_0^t\partial_s g(s,w,\tau)ds,\quad\textnormal{with}\quad g(t,w,\tau)\triangleq \frac{1}{\vert \phi_j(w)-\phi_i(\tau)+th_j(w)\vert^\alpha}\cdot
\end{eqnarray*}
Thus,
$$
\big|\partial_wK(t,w,\tau)\big|\le\int_0^t\big|\partial_w\partial_s g(s,w,\tau)\big|ds.
$$
In view of the formula \eqref{deriv-bar} we readily obtain
\begin{align*}
\partial_wg(t,w,\tau)=\frac{-\alpha}{2}\bigg[&\big(\phi_j'(w)+th_j^\prime(w)\big)\frac{{\overline{\phi_j(w)}-\overline{\phi_i(\tau)}+t\overline{h_j(w)}}}{\vert \phi_j(w)-\phi_i(\tau)+th_j(w)\big)\vert^{2+\alpha}}
\\
&-\frac{\overline{\phi_j'(w)}+t\overline{h_j^\prime(w)}}{w^2}\frac{\Big({\phi_j(w)-\phi_i(\tau)+th_j(w)\big)}\Big)}{{\vert \phi_j(w)-\phi_i(\tau)+th_j(w)\big)\vert^{2+\alpha}}}\bigg].
\end{align*}
Using straightforward computations combined with the inequality \eqref{low} yield for any $s\in [0,t],$
$$
\Big|\partial_s\partial_wg(s,w,\tau)\Big|\le C.
$$
Hence we get,
$$
\big|\partial_wK(t,w,\tau)\big|\le C|t|.
$$
This implies according to the estimate \eqref{t} that
\begin{eqnarray*}
 \big| T_{ij}(t,w_1)-T_{ij}(t,w_2)\big| &\leq & C| t|\big| w_1-w_2\big|
  \end{eqnarray*}
and consequently,
$$
\lim_{t\to 0}\| T_{ij}(t,\cdot)\|_{C^{1-\alpha}(\mathbb{T})}=0.
$$
This  concludes the desired result.

\vspace{0,4cm}


\quad The next  task  is to show that the G\^ateaux derivatives   are continuous operators from  the neighborhood $V$ into the Banach space $\mathcal{L}(C_{ar}^{1-\alpha}(\mathbb{T}), H )$. 
From the identities \eqref{gat22} and \eqref{gat223} and since $C^{1-\alpha}(\mathbb{T})$ is an algebra the problem amounts  to showing the continuity of the terms
$S({\phi_i},\phi_j)$, $A_i(\phi_j)$, $B_i(\phi_j)$, ${C}_i(\phi_j,h_i)$, ${D}_i(\phi_j,h_i)$ and ${E}_i(\phi_j,h_i)$.
We shall present here the complete details for  the term $S({\phi_i},\phi_j)$, with $i\neq j$ and the other terms can be  dealt via straightforward variations. 
 Set  
 $$\phi_{1}=\hbox{Id}+f_1,\quad\psi_{1}=\hbox{Id}+g_1,\quad \phi_{2}=b\,\hbox{Id}+f_2,\quad \psi_{2}=b\, \hbox{Id}+g_2,
 $$ 
 with $(f_1,f_2)$ and $(g_1,g_2)\in V$. We shall prove  the estimate
$$
\|S({\phi_i},\phi_j)-S({\psi_i},\psi_j)\|_{C^{1-\alpha}}\le C\Big(\Vert f_1-g_1\Vert_{C^{2-\alpha}(\mathbb{T})}+\Vert f_2-g_2\Vert_{C^{2-\alpha}(\mathbb{T})}\Big).
$$
In view of \eqref{S} we may write
\begin{eqnarray}\label{lst}
\nonumber S({\phi_i},\phi_j)(w)-S({\psi_i},\psi_j)(w)&=&\fint_\mathbb{T}\Big(\frac{\phi_i^\prime(\tau)}{\vert \phi_j(w)-\phi_i(\tau)\vert^\alpha}-\frac{\psi_i^\prime(\tau)}{\vert \psi_j(w)-\psi_i(\tau)\vert^\alpha}\Big)d\tau\\
&=&
\fint_\mathbb{T}\tilde{K}(w,\tau)\psi_i'(\tau)d\tau+\fint_\mathbb{T}\frac{\phi_i^\prime(\tau)-\psi_i^\prime(\tau)}{\vert \phi_j(w)-\phi_i(\tau)\vert^\alpha}d\tau,
\end{eqnarray}
with
$$ 
\tilde{K}(w,\tau)\triangleq \frac{1}{\vert \phi_j(w)-\phi_i(\tau)\vert^\alpha}-\frac{1}{\vert \psi_j(w)-\psi_i(\tau)\vert^\alpha}\cdot
$$
The estimate of the last term in \eqref{lst} follows immediately from \eqref{Singds}, that is,
\begin{eqnarray}\label{fst}
\nonumber\Big\|\fint_\mathbb{T}\frac{\phi_i^\prime(\tau)-\psi_i^\prime(\tau)}{\vert \phi_j(\cdot)-\phi_i(\tau)\vert^\alpha}d\tau\Big\|_{C^{1-\alpha}(\mathbb{T})}&\le&\nonumber C\|f_i^\prime-g_i^\prime\|_{L^\infty}\\
&\le&C\|f_i-g_i\|_{{C^{2-\alpha}(\mathbb{T})}}.
\end{eqnarray}
To control the remaining term we  introduce the functional
$$
L(w)\triangleq \fint_\mathbb{T}\tilde{K}(w,\tau)\psi'(\tau)d\tau.
$$
Owing to the inequalities \eqref{low} and \eqref{eq0} one has
\begin{eqnarray}\label{Linf}
\nonumber\vert L(w)\vert &\lesssim & \|{\phi_i}\|_{\textnormal{Lip}(\mathbb{T})}\mathop{{\int}}_\mathbb{T} \frac{\big|\vert \psi_j(w)-\psi_i(\tau)\vert^\alpha-\vert \phi_j(w)-\phi_i(\tau)\vert^\alpha\big|}{\vert \phi_j(w)-\phi_i(\tau)\vert^\alpha\vert \psi_j(w)-\psi_i(\tau)\vert^\alpha}| d{\tau}|\\ 
\nonumber &\lesssim & \| \psi_j-\phi_j\|_{L^\infty(\mathbb{T})}\\ &\lesssim &\|f_j-f_j\|_{L^\infty(\mathbb{T})}.
\end{eqnarray}
Now let  $w_1\neq w_2\in \mathbb{T}$, then we have
\begin{eqnarray}\label{L}
\nonumber \vert L(w_1)-L(w_2)\vert &\lesssim &   \mathop{{\int}}_\mathbb{T}\big\vert\tilde{K}(w_1,\tau)-\tilde{K}(w_2,\tau)\big\vert | d{\tau}|\\ &\lesssim & | w_1-w_2| \int_{\mathbb{T}} \sup_{w\in\mathbb{T}}\big|\partial_w \tilde{K}(w,\tau)\big|| d{\tau}|.
\end{eqnarray}
In view of \eqref{deriv-bar} the derivative of $K(w,\tau)$ with respect to $w$ is given by
\begin{eqnarray*}
\partial_w \tilde{K}(w,\tau)&=& -\frac{\alpha}{2}\Big(\overline{\mathcal{I}(w,\tau)}-\frac{\mathcal{I}(w,\tau)}{w^2}\Big),
\end{eqnarray*}
where
\begin{eqnarray*}
\mathcal{I}(w,\tau)&\triangleq&\overline{\phi_j'(w)}\frac{\phi_j(w)-\phi_i(\tau)}{\vert \phi_j(w)-\phi_i(\tau)\vert^{\alpha+2}}-\overline{\psi_j'(w)}\frac{\psi_j(w)-\psi_i(\tau)}{\vert \psi_j(w)-\psi_i(\tau)\vert^{\alpha+2}}\cdot
\end{eqnarray*}
We shall transform this quantity into, 
$$
\mathcal{I}(w,\tau)= \mathcal{I}_1(w,\tau)+\mathcal{I}_2(w,\tau)+\mathcal{I}_3(w,\tau),
$$
with
$$
\mathcal{I}_1(w,\tau)\triangleq \overline{\phi_j'(w)}\, \frac{({\phi_j}-{\psi_j})(w)-({\phi_i}-{\psi_i})(\tau)}{\vert \phi_j(w)-\phi_i(\tau)\vert^{\alpha+2}},
$$
$$
\mathcal{I}_2(w,\tau)\triangleq \big(\overline{\phi_j'(w)}-\overline{\psi_j'(w)}\big)\frac{{\psi}_j(w)-{\psi}_i(\tau)}{\vert \psi_j(w)-\psi_i(\tau)\vert^{\alpha+2}},
 $$
 and
 $$
 \mathcal{I}_3(w,\tau)\triangleq \overline{\phi_j'(w)}\big({\psi}_i(\tau)-{\psi}_j(w)\big)\frac{ \vert \phi_j(w)-\phi_i(\tau)\vert^{\alpha+2}-\vert \psi_j(w)-\psi_i(\tau)\vert^{\alpha+2}}{\vert \phi_j(w)-\phi_i(\tau)\vert^{\alpha+2}\vert \psi_j(w)-\psi_i(\tau)\vert^{\alpha+2}}\cdot
$$
For the first and the second terms one readily gets by Sobolev embeddings
\begin{eqnarray}\label{i1+i2}
\vert\mathcal{I}_1(w,\tau)\vert+\vert \mathcal{I}_2(w,\tau)\vert &\lesssim &\|\phi_i-\psi_i\|_{C^{2-\alpha}(\mathbb{T})}+\|\phi_j-\psi_j\|_{C^{2-\alpha}(\mathbb{T})}.
\end{eqnarray}
To estimate  the last term we shall use the inequality \eqref{eq1} combined with the estimate \eqref{low},
\begin{eqnarray*}
\frac{ \big|\vert \phi_j(w)-\phi_i(\tau)\vert^{\alpha+2}-\vert \psi_j(w)-\psi_i(\tau)\vert^{\alpha+2}\big|}{\vert \phi_j(w)-\phi_i(\tau)\vert^{\alpha+2}\vert \psi_j(w)-\psi_i(\tau)\vert^{\alpha+2}} &\lesssim &\|\phi_j-\psi_j\|_{L^\infty(\mathbb{T})}
\end{eqnarray*}
and consequently,
\begin{eqnarray}\label{i03}
\vert\mathcal{I}_3(w,\tau)\vert\lesssim \|\phi_j-\psi_j\|_{C^{2-\alpha}(\mathbb{T})}.
\end{eqnarray}
Putting together \eqref{i1+i2} and  \eqref{i03} we find, 
$$
\vert\mathcal{I}(w,\tau)\vert\lesssim \|\phi_1-\psi_1\|_{C^{2-\alpha}(\mathbb{T})}+\|\phi_2-\psi_2\|_{C^{2-\alpha}(\mathbb{T})}.
$$
Therefore
$$
\vert\partial_w \tilde{K}(w,\tau)\vert\leq C\big(f_1-g_1\|_{C^{2-\alpha}(\mathbb{T})}+\|f_2-g_2\|_{C^{2-\alpha}(\mathbb{T})}\big).
$$
Inserting this  inequality into the estimate \eqref{L}   we get
\begin{eqnarray*}
\nonumber \vert L(w_1)-L(w_2)\vert  &\leq & C\big(f_1-g_1\|_{C^{2-\alpha}(\mathbb{T})}+\|f_2-g_2\|_{C^{2-\alpha}(\mathbb{T})}\big)| w_1-w_2|^{1-\alpha} .
\end{eqnarray*}
Putting together the last estimate with the estimate \eqref{fst}, we obtain
$$
 \|S(\phi_i,\phi_j)-S(\psi_i,\psi_j)\|_{C^{1-\alpha}(\mathbb{T})}\lesssim \|f_1-g_1\|_{C^{2-\alpha}(\mathbb{T})}+\|f_2-g_2\|_{C^{2-\alpha}(\mathbb{T})}.
 $$
 This concludes the proof of the Proposition \refeq{reg}.
\end{proof}

\section{Spectral study}

The main goal of  this section is to perform a spectral study of the linearized operator of  $G$  at the annular solution $(\Omega,0,0)$ and denoted by the differential $D G (\Omega,0,0)$.  In particular, we shall identify the values $\Omega$  for which the kernel of  $DG(\Omega,0,0)$ is not trivial leading to what we call the dispersion relation. Therefore  the next step is to look  among the "nonlinear eigenvalues" $\Omega$  those corresponding to one-dimensional kernels which is an important assumption in  Crandall-Rabinowitz's Theorem.   This task is very  complicate compared to the previous cases discussed in \cite{Bur, H-H,H-F-M-V}. This is due to the  multiple parameters $\alpha, b$ and $m$ in this problem and especially to the  nonlinear and implicit structure of the coefficients appearing in  the  dispersion relation. We will be   able to    validate only a sufficient, but still  a satisfactory result compared to Euler equation, with a restriction on the symmetry of the V-states. This will be deeply discussed in the  Subsection \ref{subsec12} devoted to the  monotonicity of the eigenvalues.

\subsection{Linearized operator}  
We propose  to compute explicitly   the differential $DG(\Omega,0,0)$ and  show that it acts as a Fourier multiplier.
 Since $G=(G_1,G_2)$ then for  given $(h_1,h_2) \in X,$ we have
\begin{eqnarray*}
DG(\Omega,0,0)(h_1,h_2)&=& \begin{pmatrix} D_{f_1}G_1(\Omega,0,0)h_1+D_{f_2}
G_1(\Omega,0,0)h_2 \\*[4pt] D_{f_1} G_2(\Omega,0,0)h_1+D_{f_2}
G_2(\Omega,0,0)h_2
\end{pmatrix}.
\end{eqnarray*}
where we recall the function spaces
$$
X=C_{ar}^{2-\alpha}(\mathbb{T})\times C_{ar}^{2-\alpha}(\mathbb{T}),
$$
and
$$
C_{ar}^{2-\alpha}(\mathbb{T})=\bigg\{ f\in C^{2-\alpha}(\mathbb{T}); f(w)=\sum_{n\geq 1} a_n\overline{w}^n, w\in \mathbb{T}, a_n\in \RR, n\in\NN^\star\bigg\}.
$$

Putting together the formulas \eqref{dlj}, \eqref{gat22} and \eqref{gat223} with $j=1$ and $j=2$,  where we replace  $\phi_1$ by $\hbox{Id}$ and $\phi_2$ by $b\,\hbox{Id}$ we get 
\begin{eqnarray}\label{df001}
DG_1(\Omega,0,0)h(w)= \Omega\,\mathcal{L}_0\big(h_1\big)(w)- C_\alpha\, \mathcal{L}_1\big(h_1\big)(w)+C_\alpha\,\mathcal{L}_2\big(h_1,h_2)(w),\quad
\end{eqnarray}
\begin{eqnarray}\label{df002}
\qquad\;\, DG_2(\Omega,0,0)h(w)=  b\Big(\Omega\,\mathcal{L}_0\big(h_2\big)(w)+b^{-\alpha}C_\alpha\,\mathcal{L}_1\big(h_2\big)(w)- C_\alpha\,\mathcal{L}_3\big(h_1,h_2\big)(w)\Big),
\end{eqnarray}
with
\begin{equation*}
\mathcal{L}_0\big(h_j\big)(w)\triangleq \textnormal{Im}\Big\{\overline{h_j'(w)}+\overline{w}{h_j(w)}\Big\},\qquad \qquad \qquad \qquad \qquad\quad
\end{equation*}
\begin{align*}
\mathcal{L}_1\big(h_j\big)(w)\triangleq \textnormal{Im}\bigg\{&\overline{w}{\overline{h_j'(w)}}\fint_\mathbb{T}\frac{d\tau}{\vert w-\tau\vert^\alpha}+\overline{w}\fint_\mathbb{T}\frac{h_j'(\tau)}{\vert w-\tau\vert^\alpha}d\tau\\
&-{\alpha\overline{w} }\fint_\mathbb{T}\frac{\Rea\big[(w-\tau)\big(\overline{h_j(w)}-\overline{h_j(\tau)}\big)\big]}{\vert w-\tau\vert^{\alpha+2}}d\tau\bigg\},
\end{align*}

\begin{align*}
\mathcal{L}_2\big(h_1,h_2\big)(w)\triangleq \Ima\bigg\{&b\overline{w}\overline{h_1^\prime(w)}\fint_\mathbb{T}\frac{d\tau}{\vert w-b\tau\vert^\alpha}+\overline{w}\fint_\mathbb{T}\frac{h_2'(\tau)d\tau}{\vert w-b\tau\vert^\alpha}\notag\\ &- {\alpha b\overline{w}}\fint_\mathbb{T}\frac{\Rea\big[(w-b\tau)\big(\overline{h_1(w)}-\overline{h_2(\tau)}\big)\big]d\tau}{\vert w-b\tau\vert^{\alpha+2}}\bigg\},
\end{align*}
and 
\begin{align*}
\mathcal{L}_3\big(h_1,h_2\big)(w)\triangleq \Ima\bigg\{&b{\overline{w}\,{\overline{h_2^\prime(w)}}}\fint_\mathbb{T}\frac{d\tau}{\vert bw-\tau\vert^\alpha}+\overline{w}\fint_\mathbb{T}\frac{h_1'(\tau)}{\vert bw-\tau\vert^\alpha}d\tau\\ &- {\alpha }\overline{w}\fint_\mathbb{T}\frac{\Rea\big[(bw-\tau)\big(\overline{h_2(w)}-\overline{h_1(\tau)}\big)\big]}{\vert bw-\tau\vert^{\alpha+2}}d\tau\bigg\}.
\end{align*}
We shall now  compute the Fourier series of the mapping $w \mapsto DG(\Omega,0,0)(h_1,h_2)(w)$ with
$$
h_1(w)=\sum_{n=1}^{\infty}a_n\overline{w}^n\quad\textnormal{and}\quad h_2(w)=\sum_{n=1}^{\infty}c_n\overline{w}^n, \quad w\in\mathbb{T},
 $$
 where  $a_n$ and $c_n$ are real for all $n\in \NN^\star$. This is summarized in the following lemma.
 \begin{lemma}\label{lem0}Let $\alpha\in (0,1)$ and $b\in(0,1)$.
 We set
\begin{equation}\label{Pi}
\Lambda_n(b)\triangleq \frac1b\int_0^{+\infty}\frac{J_n(bt)J_n(t)}{t^{1-\alpha}}dt,
\end{equation}
and
\begin{eqnarray*}
\Theta_{n}&\triangleq& \Lambda_1(1)-\Lambda_n(1),
\end{eqnarray*}
where $J_n$ refers to the Bessel function of the first kind. 
Then, we have
\begin{eqnarray}\label{dff}
DG(\Omega,0,0)\big(h_1,h_2\big)(w)=\frac{i}{2}\sum_{n\geq 1}\big(n+1\big)M^\alpha_{n+1}\left( \begin{array}{c}
a_n \\
c_n
\end{array} \right)\Big(w^{n+1}-\overline{w}^{n+1}\Big).
\end{eqnarray}
where the matrix $M_n$ is given for $n\geq2$  by
\begin{equation}\label{matrix}
M^\alpha_{n}\triangleq\begin{pmatrix}
 \Omega-\Theta_{n}+b^2\Lambda_{1}(b) & -b^2\Lambda_{n}(b) \\
  b\Lambda_n(b) & b\Omega+b^{1-\alpha}\Theta_{n}-b\Lambda_1(b)
\end{pmatrix}.
\end{equation}
The   determinant  of this matrix  is given  by
\begin{eqnarray}\label{cdet}
\textnormal{det}\big(M^\alpha_{n}\big)&=&  \Big(\Omega-\Theta_{n}+b^2\Lambda_{1}(b)\Big)\Big(b\Omega+b^{1-\alpha}\Theta_{n}-b\Lambda_1(b)\Big)+b^3\Lambda_{n}^2(b)\cdot \end{eqnarray}
 \end{lemma} 

 \begin{proof}
 First we shall compute  $DG_1(\Omega,0,0)(h_1,h_2)$. For this goal 
we  start with calculating  the  term  $\mathcal{L}_0\big(h_1(w)\big)$ of the right-hand side of \eqref{df001}  which is easy compared to the other terms. Thus by straightforward computations we obtain
\begin{eqnarray}\label{l0}
\mathcal{L}_0\big(h_1\big)(w)&=&\textnormal{Im}\bigg\{\sum_{n\geq1}\Big( a_n \overline{w}^{n+1}-na_n w^{n+1}\Big)\bigg\}\\
&=& \frac{i}{2}\sum_{n\geq1}(n+1)a_n\Big(w^{n+1}-\overline{w}^{n+1}\Big)\notag.
\end{eqnarray}
The computation of the second term $\mathcal{L}_1\big(h_1\big)(w)$ was done in the paper \cite{H-H} dealing with  the simply connected domain. It is given by 
\begin{eqnarray*}
C_\alpha\mathcal{L}_1\big(h_1\big)(w)=
  \frac{i}{2}\sum_{n\geq 1}a_n\big(n+1\big)\frac{\Gamma(1-\alpha)}{2^{1-\alpha}\Gamma^2(1-\frac\alpha2)} \bigg(\frac{\Gamma(1+\frac\alpha2)}{\Gamma(2-\frac\alpha2)}-\frac{\Gamma(n+1+\frac\alpha2)}{\Gamma(n+2-\frac\alpha2)}\bigg)\Big(w^{n+1}-\overline{w}^{n+1}\Big).
\end{eqnarray*}
We shall later establish  the identity \eqref{theta} which gives here
\begin{eqnarray}\label{l1}
C_\alpha\,\mathcal{L}_1\big(h_1\big)(w)&=&  \frac{i}{2}\sum_{n\geq 1}a_n\big(n+1\big)\Theta_{n+1}\Big(w^{n+1}-\overline{w}^{n+1}\Big).
\end{eqnarray}
To compute the term $\mathcal{L}_2\big(h_1,h_2\big)(w)$ we first  split it into two parts as follows,
\begin{eqnarray}\label{ll2}
\mathcal{L}_2\big(h_1,h_2\big)(w)&=&  \textnormal{Im}\Big\{\hbox{I}_1(w)+\hbox{I}_2(w)\Big\},
\end{eqnarray}
with
$$
\hbox{I}_1(w)\triangleq \frac{b\overline{h_1'(w)}}{w}\fint_\mathbb{T}\frac{d\tau}{\vert w-b\tau\vert^\alpha}- \frac{\alpha b}{2w}\fint_\mathbb{T}\frac{(w-b\tau)\big(\overline{h_1(w)}-\overline{h_2(\tau)}\big)}{\vert w-b\tau\vert^{\alpha+2}}d\tau
$$
and
$$
\hbox{I}_2(w)\triangleq \frac{1}{w}\fint_\mathbb{T}\frac{h_2'(\tau)}{\vert w-b\tau\vert^\alpha}d\tau-\frac{\alpha b}{2w}\fint_\mathbb{T}\frac{(\overline{w}-b\overline{\tau})\big(h_1(w)-h_2(\tau)\big)}{\vert w-b\tau\vert^{\alpha+2}}d\tau.
$$
By using the Fourier expansions of  $h_1$ and $h_2$ we get,
$$
\hbox{I}_1(w) =-b\sum_{n\geq 1}na_nw^{n}\fint_\mathbb{T}\frac{d\tau}{\vert w-b\tau\vert^\alpha} -\frac{\alpha }{2}b\overline{w}\sum_{n\geq 1}\fint_\mathbb{T}\frac{(w-b\tau)(a_nw^{n}-c_n\tau^{n})}{\vert w-b\tau\vert^{\alpha+2}}d\tau.
$$
Then by applying the formula \eqref{In} with $n=1$ to the first term  and the formula \eqref{Jn} to the second term we find
\begin{alignat}{2}\label{I1}
\hbox{I}_1(w) =& -\frac{\alpha}{2}b^2\sum_{n\geq 1}na_nw^{n+1} && F\Big(\frac\alpha2,1+\frac\alpha2;2;b^2\Big)\notag\\ &-\frac{\alpha }{2}b^2\sum_{n\geq 0}w^{n+1}\Bigg[  a_n&&\Big(1+\frac\alpha2\Big)F\Big(\frac\alpha2,2+\frac\alpha2;2;b^2\Big)\notag\\  &  &&- c_nb^{n}\frac{(1+\frac\alpha2)_{n+1}}{(n+1)!}F\Big(\frac\alpha2,n+2+\frac\alpha2;n+2;b^2\Big)\Bigg]\notag\\\triangleq&  -\frac{\alpha}{2}b^2\sum_{n\geq 0}\big(\,a_n\gamma_n+&&c_n\delta_n\big)w^{n+1},
\end{alignat}
where we have used in the last equality  the notation,
$$
\gamma_n\triangleq\Big(1+\frac\alpha2\Big)F\Big(\frac\alpha2,2+\frac\alpha2;2;b^2\Big)+nF\Big(\frac\alpha2,1+\frac\alpha2;2;b^2\Big)
$$
and
$$
\delta_n=-b^{n}\frac{(1+\frac\alpha2)_{n+1}}{(n+1)!}F\Big(\frac\alpha2,n+2+\frac\alpha2;n+2;b^2\Big).
$$
 Similarly, the second term $\hbox{I}_2(w)$ may be written in the form,
\begin{eqnarray*}
\hbox{I}_2(w)  &=&-{\overline{w}}\sum_{n\geq 0}nc_n\fint_\mathbb{T}\frac{\overline{\tau}^{n+1}d\tau}{\vert w-b\tau\vert^\alpha}-\frac{\alpha}{2}b\overline{w}\sum_{n\geq 0}\fint_\mathbb{T}\frac{(\overline{w}-b\overline{\tau})(a_n\overline{w}^{n}-c_n\overline{\tau}^{n})}{\vert w-b\tau\vert^{\alpha+2}}d\tau.
\end{eqnarray*}
Using the elementary fact 
\begin{equation}\label{fct}
\fint_\mathbb{T}\frac{\overline{\tau}^{n+1}d\tau}{\vert w-b\tau\vert^\alpha}=\overline{\fint_\mathbb{T}\frac{{\tau}^{n-1}d\tau}{\vert w-b\tau\vert^\alpha}}
\end{equation}
combined with the formulas  \eqref{In} and \eqref{Kn} we obtain
\begin{alignat}{2}\label{i2}
\hbox{I}_2(w) =&-\sum_{n\geq 1}nc_n\frac{(\frac\alpha2)_n}{n!}&&\overline{w}^{n+1}b^{n}F\Big(\frac\alpha2,n+\frac\alpha2;n+1;b^2\Big)\notag\\ &-\frac{\alpha }{2}\sum_{n\geq 1}\overline{w}^{n+1}\bigg[&&a_nb^2\frac\alpha2F\Big(\frac\alpha2+1,\frac\alpha2+1;2;b^2\Big)-c_nb^{n}\frac{(1+\frac\alpha2)_{n-1}}{(n-1)!}F\Big(\frac\alpha2,n+\frac\alpha2;n;b^2\Big)\bigg]\notag\\ =&-\frac{\alpha }{2}\sum_{n\geq 1}\overline{w}^{n+1}\Bigg[&&a_nb^2\frac\alpha2F(\frac\alpha2+1,\frac\alpha2+1;2;b^2)\notag\\ & &&+c_n\frac{(1+\frac\alpha2)_{n-1}}{(n-1)!}b^{n}\bigg(F\Big(\frac\alpha2,n+\frac\alpha2;n+1;b^2\Big)-F\Big(\frac\alpha2,n+\frac\alpha2;n;b^2\Big)\bigg)\Bigg].\end{alignat}
Thus owing to the formula \eqref{f0} applied with $a=\frac\alpha2, b=\frac\alpha2+n$ and $c=n$ one gets
\begin{eqnarray}\label{i22}
\hbox{I}_2(w)&=&
-\frac{\alpha }{2}b^2\sum_{n\geq 1}\Big(\alpha_na_n+\beta_nc_n\Big)\overline{w}^{n+1},
\end{eqnarray}
where $\alpha_n$ and $\beta_n$ are defined by,
$$
\alpha_n\triangleq \frac\alpha2F\Big(\frac\alpha2+1,\frac\alpha2+1;2;b^2\Big)
$$
and
$$
\beta_n\triangleq -\frac{(\frac\alpha2)_{n+1}}{(n+1)!}b^{n}F\Big(\frac\alpha2+1,n+1+\frac\alpha2;n+2;b^2\Big).
$$
Inserting   the identities \eqref{i22} and \eqref{I1} into \eqref{ll2} we find
\begin{eqnarray*}
\mathcal{L}_2\big(h_1,h_2\big)(w)&=&-\frac\alpha2 b^2\textnormal{Im}\bigg\{\sum_{n\geq 1}\big(a_n\gamma_n+c_n\delta_n\big)w^{n+1}+\sum_{n\geq 1}\big(a_n\alpha_n+c_n\beta_n\big)\overline{w}^{n+1} \bigg\}\notag\\ &=&i \frac\alpha4 b^2\sum_{n\geq 1}\Big(w^{n+1}-\overline{w}^{n+1}\Big)\Big[a_n\big(\gamma_n-\alpha_n\big)+c_n\big(\delta_n-\beta_n\big)\Big].
\end{eqnarray*}

To compute $\gamma_n-\alpha_n$ we shall use the formula \eqref{f34}  which gives 
\begin{eqnarray*}
\gamma_n-\alpha_n &=& \Big(1+\frac\alpha2\Big)F\Big(\frac\alpha2,2+\frac\alpha2;2;b^2\Big)-\frac\alpha2F\Big(\frac\alpha2+1,\frac\alpha2+1;2;b^2\Big)+n\,F\Big(\frac\alpha2,1+\frac\alpha2;2;b^2\Big)\\ &=&(n+1)\, F\Big(\frac\alpha2,1+\frac\alpha2;2;b^2\Big).
\end{eqnarray*}
Similarly we have
\begin{align*}
\delta_n-\beta_n = 
-b^{n}\frac{(1+\frac\alpha2)_{n}}{(n+1)!}\bigg[&\Big(n+1+\frac\alpha2\Big)F\Big(\frac\alpha2,n+2+\frac\alpha2;n+2;b^2\Big)\\ &-\frac\alpha2F\Big(\frac\alpha2+1,n+1+\frac\alpha2;n+2;b^2\Big)\bigg]\end{align*}
and therefore using once again the  identity  \eqref{f34} we find
$$
\delta_n-\beta_n=-b^{n}\frac{(1+\frac\alpha2)_{n}}{n!}F\Big(\frac\alpha2,n+1+\frac\alpha2,n+2,b^2\Big).
$$
Consequently the Fourier expansion  of $\mathcal{L}_2\big(h_1,h_2\big)$ is described by the formula 
\begin{align*}
\mathcal{L}_2\big(h_1,h_2\big)(w)=\frac{i}{2}b^2\sum_{n\geq 1}(n+1)\bigg[&a_n\frac\alpha2 F\Big(\frac\alpha2,1+\frac\alpha2,2,b^2\Big)\\ &-c_nb^{n}\frac{(\frac\alpha2)_{n+1}}{(n+1)!}F\Big(\frac\alpha2,n+1+\frac\alpha2,n+2,b^2\Big)\bigg]\Big(w^{n+1}-\overline{w}^{n+1}\Big).\notag
\end{align*}
By virtue of  the identity \eqref{lambdaf} we get
\begin{align}\label{l22}
C_\alpha\,\mathcal{L}_2\big(h_1,h_2\big)(w)=\frac{i}{2}\sum_{n\geq 1}(n+1)\bigg[a_nb^2\Lambda_1(b)-c_nb^2\Lambda_{n+1}(b)\bigg]\Big(w^{n+1}-\overline{w}^{n+1}\Big).
\end{align}
Finally inserting  \eqref{l0}, \eqref{l1} and \eqref{l22} into \eqref{df001} we find

\begin{align}\label{df1f}
\nonumber DG_1(\Omega,0,0)(h_1,h_2)(w)=\frac{i}{2}\sum_{n\geq 1}&\big(n+1\big) \Big[ a_n\Big(\Omega-\Theta_{n+1}+b^2\Lambda_1(b)\Big)-c_nb^2 \Lambda_{n+1}(b)\Big]\\ &\times \Big(w^{n+1}-\overline{w}^{n+1}\Big) .
\end{align}

Next, we shall move to the computations of $DG_2(\Omega,0,0)(h_1,h_2)$ defined in \eqref{df002}.  The first two terms are done in the preceding step and therefore it remains just to compute 
the  term $\mathcal{L}_3\big(h_1,h_2\big)$. It may be splitted into two terms,
\begin{eqnarray}\label{l2}
\mathcal{L}_3\big(h_1,h_2\big)(w)=\textnormal{Im}\Big\{\tilde{\hbox{I}}_1(w)+\tilde{\hbox{I}}_2(w)\Big\},
\end{eqnarray}
with
$$
\tilde{\hbox{I}}_1(w)\triangleq \frac{\overline{h_2'(w)}}{bw}\fint_\mathbb{T}\frac{d\tau}{\vert bw-\tau\vert^\alpha}-\frac{\alpha }{2w}\fint_\mathbb{T}\frac{(bw-\tau)\big(\overline{h_2(w)}-\overline{h_1(\tau)}\big)}{\vert bw-\tau\vert^{\alpha+2}}d\tau
$$
and
$$
\tilde{\hbox{I}}_2(w)\triangleq \frac{1}{w}\fint_\mathbb{T}\frac{h_1'(\tau)d\tau}{\vert bw-\tau\vert^\alpha}-\frac{\alpha }{2w}\fint_\mathbb{T}\frac{(b\overline{w}-\overline{\tau})\big(h_2(w)-h_1(\tau)\big)}{\vert bw-\tau\vert^{\alpha+2}}d\tau.
$$
To compute the first term $\tilde{\hbox{I}}_1(w)$ we write
\begin{eqnarray*}
\tilde{\hbox{I}}_1(w)&=& -\sum_{n\geq 1}n\frac{c_n}{b}w^n\fint_\mathbb{T}\frac{d\tau}{\vert bw-\tau\vert^\alpha}-\frac{\alpha }{2w}\sum_{n\geq 1}\fint_\mathbb{T}\frac{(bw-\tau)\big(c_nw^n-a_n\tau^n\big)}{\vert bw-\tau\vert^{\alpha+2}}d\tau.
\end{eqnarray*}
Thus applying successively the formula \eqref{In} to the first term with $n=1$  and the \mbox{formula \eqref{Ln}} to the second one we get
\begin{alignat}{2}\label{tI1}
\tilde{\hbox{I}}_1(w) =&-\frac\alpha2\sum_{n\geq 1}\,n\,{c_n}\, &&w^{n+1}F\Big(\frac\alpha2,1+\frac\alpha2;2;b^2\Big)\notag\\ &+\frac\alpha2\sum_{n\geq 1}w^{n+1}&&\bigg[c_nb^2\frac\alpha4\big(1+\frac\alpha2\big)F\Big(1+\frac\alpha2,2+\frac\alpha2;3;b^2\Big)\notag\\ & &&- a_nb^{n+2}\frac{(\frac\alpha2)_{n+2}}{(n+2)!}F\Big(1+\frac\alpha2,n+2+\frac\alpha2;n+3;b^2\Big)\bigg]\notag \\ =& -\frac{\alpha }{2}\,\sum_{n\geq 1}\big(\,a_n&&\tilde{\gamma}_n+c_n\,\tilde{\delta}_n\,\big){w}^{n+1}\end{alignat}
with
$$
\tilde{\gamma}_n\triangleq b^{n+2}\frac{(\frac\alpha2)_{n+2}}{(n+2)!}F\Big(1+\frac\alpha2,n+2+\frac\alpha2;n+3;b^2\Big)
$$
and
$$
\tilde{\delta}_n\triangleq nF\Big(\frac\alpha2,1+\frac\alpha2;2;b^2\Big)-b^2\frac\alpha4\Big(1+\frac\alpha2\Big)F\Big(1+\frac\alpha2,2+\frac\alpha2;3;b^2\Big).
$$ 
As to  the term $\tilde{\hbox{I}}_2(w)$ we write
\begin{eqnarray*}
\tilde{\hbox{I}}_2(w)&=&-\frac{1}{w}\sum_{n\geq 1}na_n\fint_\mathbb{T}\frac{\overline{\tau}^{n+1}d\tau}{\vert bw-\tau\vert^\alpha}-\frac{\alpha }{2w}\sum_{n\geq 1}\fint_\mathbb{T}\frac{(b\overline{w}-\overline{\tau})\big(c_n\overline{w}^{n}-a_n\overline{\tau}^{n}\big)d\tau}{\vert bw-\tau\vert^{\alpha+2}}\cdot
\end{eqnarray*}
Owing to \eqref{fct} and using the  formulae \eqref{In} and \eqref{Mn}, one gets
\begin{align}\label{tI2}
\tilde{\hbox{I}}_2(w) =&-\sum_{n\geq 1}na_n\overline{w}^{n+1}b^n\frac{(\frac\alpha2)_n}{n!}F\Big(\frac\alpha2,n+\frac\alpha2;n+1;b^2\Big)\notag\\ &-\frac\alpha2\sum_{n\geq 1} \overline{w}^{n+1}\bigg[-c_nF\Big(\frac\alpha2,\frac\alpha2+1;1;b^2\Big)+a_nb^{n}\frac{(\frac\alpha2)_{n}}{n!}F\Big(\frac\alpha2+1,n+\frac\alpha2;n+1;b^2\Big)\bigg]\notag\\ =&- \frac{\alpha }{2}\sum_{n\geq 1}\big(a_n\tilde{\alpha}_n+c_n\tilde{\beta}_n\big)\overline{w}^{n+1},
\end{align}
with
$$
\tilde{\alpha}_n\triangleq\frac{(1+\frac\alpha2)_{n-1}}{n!}b^{n}\bigg[n\,F\Big(\frac\alpha2,n+\frac\alpha2;n+1;b^2\Big)+\frac\alpha2F\Big(1+\frac\alpha2,n+\frac\alpha2;n+1;b^2\Big)\bigg] $$
and
$$
\tilde{\beta}_n\triangleq-F\Big(\frac\alpha2,\frac\alpha2+1;1;b^2\Big).
$$
Now  inserting   the  identities \eqref{tI1} and \eqref{tI2} into \eqref{l2} we find
\begin{eqnarray}\label{df}
\mathcal{L}_3\big(h_1,h_2\big)(w)&=&-\frac\alpha2 \textnormal{Im}\bigg\{\sum_{n\geq 1}\big(a_n\tilde{\gamma}_n+c_n\tilde{\delta}_n\big)w^{n+1}+\sum_{n\geq 1}\big(a_n\tilde{\alpha}_n+c_n\tilde{\beta}_n\big)\overline{w}^{n+1} \bigg\}\notag\\ &=&i\frac\alpha4 \sum_{n\geq 1}\Big[a_n\big(\tilde{\gamma}_n-\tilde{\alpha}_n\big)+c_n\big(\tilde{\delta}_n-\tilde{\beta}_n\big)\Big]\Big(w^{n+1}-\overline{w}^{n+1}\Big).
\end{eqnarray}
From  the foregoing expressions for $\tilde{\gamma}_n$ and $\tilde{\alpha}_n$ one may write,
\begin{alignat*}{2}
\tilde{\gamma}_n-\tilde{\alpha}_n &= b^{n+2}\frac{(\frac\alpha2)_{n+2}}{(n+2)!} &&F\Big(1+\frac\alpha2,n+2+\frac\alpha2;n+3;b^2\Big)\\ &-\frac{(1+\frac\alpha2)_{n-1}}{n!}b^{n}&&\bigg[nF\Big(\frac\alpha2,n+\frac\alpha2;n+1;b^2\Big)+\frac\alpha2F\Big(1+\frac\alpha2,n+\frac\alpha2;n+1;b^2\Big)\bigg]\\ &=\,\,\frac{(1+\frac\alpha2)_{n-1}}{n!}b^{n}&&\bigg[b^2\frac{\frac\alpha2\big(\frac\alpha2+n\big)\big(\frac\alpha2+1+n\big)}{(n+2)(n+1)}F\Big(1+\frac\alpha2,n+2+\frac\alpha2;n+3;b^2\Big)\\ & &&- nF\Big(\frac\alpha2,n+\frac\alpha2;n+1;b^2\Big)-\frac\alpha2F\Big(1+\frac\alpha2,n+\frac\alpha2;n+1;b^2\Big)\bigg].
\end{alignat*}
Hence using   the formula \eqref{f0} with $a=\frac\alpha2, b=n+1+\frac\alpha2$ and $c=n+1$ yields 
\begin{align*}
\tilde{\gamma}_n-\tilde{\alpha}_n  =\,\,\frac{(1+\frac\alpha2)_{n-1}}{n!}b^{n}\bigg[&\big(\frac\alpha2+n\big)F\Big(\frac\alpha2,n+1+\frac\alpha2;n+1;b^2\Big)\\ &-\big(\frac\alpha2+n\big)F\Big(\frac\alpha2,n+1+\frac\alpha2;n+2;b^2\Big)\\ &-\,\, nF\Big(\frac\alpha2,n+\frac\alpha2;n+1;b^2\Big)-\frac\alpha2F\Big(1+\frac\alpha2,n+\frac\alpha2;n+1;b^2\Big)\bigg].
\end{align*}
Applying the formula \eqref{f34} with $a=\frac\alpha2, b=\frac\alpha2+n$ and $c=n+1$  we get
\begin{eqnarray*}
\big(\frac\alpha2+n\big)F\Big(\frac\alpha2,n+1+\frac\alpha2;n+1;b^2\Big)-\frac\alpha2F\Big(1+\frac\alpha2,n+\frac\alpha2;n+1;b^2\Big)=nF\Big(\frac\alpha2,n+\frac\alpha2;n+1;b^2\Big).
\end{eqnarray*}
This implies,
\begin{eqnarray*}
\tilde{\gamma}_n-\tilde{\alpha}_n =-\frac{(1+\frac\alpha2)_{n}}{n!}b^{n}F\Big(\frac\alpha2,n+1+\frac\alpha2;n+2;b^2\Big).
\end{eqnarray*}
Using  the expressions of $\tilde{\delta}_n$ and $\tilde{\beta}_n$ combined with the identity \eqref{f0} applied with $a=\frac\alpha2, b=1+\frac\alpha2$ and $c=1$ we find the compact formula
\begin{alignat*}{2}
\tilde{\delta}_n-\tilde{\beta}_n &= \, && nF\Big(\frac\alpha2,1+\frac\alpha2;2;b^2\Big)-b^2\frac\alpha4\big(1+\frac\alpha2\big)F\Big(1+\frac\alpha2,2+\frac\alpha2;3;b^2\Big)\\ & &&+F\Big(\frac\alpha2,\frac\alpha2+1;1;b^2\Big)\\  &=\, && (n+1)F\Big(\frac\alpha2,1+\frac\alpha2;2;b^2\Big).\end{alignat*}
Putting together the preceding identities allows to write
\begin{align*}
\mathcal{L}_3\big(h_1,h_2(w)\big)&=i\frac\alpha4 \sum_{n\geq 1}(n+1)\Big(w^{n+1}-\overline{w}^{n+1}\Big)\\
&\times\bigg[c_n\,F\Big(\frac\alpha2,1+\frac\alpha2;2;b^2\Big)- a_nb^{n}\frac{(1+\frac\alpha2)_{n}}{(n+1)!}F\Big(\frac\alpha2,n+1+\frac\alpha2,n+2,b^2\Big)\bigg].
\end{align*}
According to  the identity \eqref{lambdaf} we get
\begin{align*}
C_\alpha\,\mathcal{L}_3\big(h_1,h_2(w)\big)= \frac{i}{2} \sum_{n\geq 1}(n+1)\Big(c_n\Lambda_1(b)-a_n\Lambda_{n+1}(b)\Big)\Big(w^{n+1}-\overline{w}^{n+1}\Big).
\end{align*}
Finaly, inserting the  preceding identity and the expressions \eqref{l1} and \eqref{l2} into \eqref{df002} one can readily verify that
\begin{align}\label{df2f}
\nonumber DG_2(\Omega,0,0)(h_1,h_2)(w)=  \frac{i}{2}\sum_{n\geq 1}&\big(n+1\big)\bigg[a_nb\Lambda_{n+1}(b)+ c_n\Big(b\Omega+b^{1-\alpha}\Theta_{n+1}-b\Lambda_{1}(b)\Big)\bigg]\\ &\times \Big(w^{n+1}-\overline{w}^{n+1}\Big).
\end{align}
This concludes the proof of the Lemma \ref{lem0}.
\end{proof}
 \subsection{Asymptotic behavior}
 We shall  collect some useful properties on the asymptotic behavior of the sequences  $(\Theta_n)_n$ and $(\Lambda_n)_n$  introduced in Lemma \ref{lem0}. The study is  done with respect to the parameters  $\alpha$ and $n$. This is summarized in the next lemma.
 \begin{lemma} \label{lem1} Let $\alpha\in (0,1)$ and $b\in(0,1)$. Then the following results hold true.
 \begin{enumerate}
\item For all $n\in\NN^*$, $\Theta_n\geq0$, $\Lambda_n\geq 0$. Moreover, $b\mapsto \Lambda_n(b)$ is strictly increasing, $n\mapsto\Theta_n$  is strictly increasing and   $n\mapsto \Lambda_n(b)$ is strictly decreasing.
\item Let $n\geq 2$, then
$$
\lim_{\alpha\to 0} \Theta_n =\frac{n-1}{2n},\quad \lim_{\alpha\to 0} \Lambda_n(b) =\frac{b^{n-1}}{2n}
$$
and
$$
\lim_{\alpha\to 1} \Theta_n =\frac2\pi\sum_{k=1}^{n-1}\frac{1}{2k+1},\quad \lim_{\alpha\to 1} \Lambda_n(b) =\frac1b\int_0^{+\infty}J_n(bt)J_n(t)dt.
$$

\item  For  $n$ sufficiently large,
 \begin{equation}\label{As3}
\Lambda_n(b)=O(b^{n-1}), \quad \lim_{n\to\infty}\Lambda_n(b)=0,
\end{equation}
\begin{equation}\label{As1}
\Theta_n=\Lambda_1(1)-\big(1-{\alpha}/{2}\big)\Lambda_1(1)\frac{e^{\alpha\gamma+c_\alpha}}{n^{1-\alpha}}+O\Big(\frac{1}{n^{2-\alpha}}\Big), \quad \lim_{n\to\infty}\Theta_n=\Lambda_1(1).
\end{equation}
\item The determinant of the matrix $M^\alpha_n$ introduced in \eqref{matrix} satisfies
\begin{eqnarray}\label{Asdet}
\textnormal{det}(M^\alpha_{n})&=&\mu+\frac{\nu}{n^{1-\alpha}}+O\Big(\frac{1}{n^{2-\alpha}}\Big)\quad \textnormal{and}\quad \lim_{n\to +\infty}\textnormal{det}(M_{n})=\mu,
\end{eqnarray}
with
$$
\mu\triangleq \Big(\Omega-\Lambda_{1}(1)+b^2\Lambda_{1}(b)\Big)\Big(b\Omega+b^{1-\alpha}\Lambda_{1}(1)-b\Lambda_1(b)\Big)
$$
and
$$
\nu\triangleq\big(1-{\alpha}/{2}\big)\Lambda_1(1)\Big(2b^{1-\alpha}\Lambda_1(1)+(b-b^{1-\alpha})\Omega-b(1+b^{2-\alpha})\Lambda_1(b)\Big){e^{\alpha\gamma+c_\alpha}},
$$
with $\gamma$ denotes Euler constant,  $c_\alpha$ is the sum of the series 
$$
c_\alpha\triangleq\sum_{m=1}^\infty\frac{\alpha^{2m+1}}{2^{2m-1}(2m+1)}\zeta(2m+1).
$$
and $s\mapsto \zeta(s)$ is the Riemann zeta function.
 \end{enumerate}
 \end{lemma}

 \begin{remark}
The assertion $(2)$ from the last lemma shows that  the spectrum  is continuous with respect to $\alpha$. In other words, we have
$$
M_n^0\triangleq\displaystyle{\lim_{\alpha\to 0}M_n^\alpha=\begin{pmatrix}
 \Omega-\dfrac{n-1}{2n}+\dfrac{b^2}{2}& -\dfrac{b^{n+1}}{2n} \\
 \dfrac{b^{n+1}}{2n} & b\big(\Omega+\dfrac{n-1}{2n}-\dfrac{1}{2}\big)
\end{pmatrix}}
$$
Hence, by the change of variable $\lambda=1-2\Omega$ we can see that $M_n^0$ is exactly the same matrix obtained in \cite{H-F-M-V}.
However, for  $\alpha=1$  the dispersion relation  established in \cite{HAC} involves  the following matrix    
\begin{equation*}
\widetilde{M_n}=\begin{pmatrix}
\Omega-\dfrac{2}{\pi} \displaystyle{\sum_{k=1}^{n-1}\frac{1}{2k+1}+b^2\Lambda_1(b)}& -b^3\displaystyle{\Lambda_n(b)} \\
\displaystyle{\frac1b\Lambda_n(b)} &\Omega+ \dfrac{2}{b\pi} \displaystyle{\sum_{k=1}^{n-1}\frac{1}{2k+1}-\Lambda_1(b)}
\end{pmatrix}.
\end{equation*}
This discrepancy with the matrix $M_n^1$ is due to the parametrization used in \cite{HAC} for the interior curve. Indeed, in that paper the perturbation of the interior curve is dilated by $b$. Thus with our parametrization we should multiply the second column of $\widetilde{M_n}$ by $b$ and the matrix $\widetilde{M_n}$ becomes
\begin{equation*}
\begin{pmatrix}
\Omega-\dfrac{2}{\pi} \displaystyle{\sum_{k=1}^{n-1}\frac{1}{2k+1}+b^2\Lambda_1(b)}& -b^4\displaystyle{\Lambda_n(b)} \\
\displaystyle{\frac1b\Lambda_n(b)} &b\Omega+ \dfrac{2}{\pi} \displaystyle{\sum_{k=1}^{n-1}\frac{1}{2k+1}-b\Lambda_1(b)}
\end{pmatrix}.
\end{equation*}
As we can easily see, this matrix has exactly the same determinant of  the matrix $M_n^1$ and therefore we find the same dispersion relations.
 \end{remark}
 \vspace{0,5cm}
 Let us now prove  Lemma \ref{lem1}.
 \begin{proof}
{\bf (1)} To study the sign of $\Lambda_n$ we shall make use of  Sonine-Schafheitlin's formula \eqref{sonine} leading to the identity
\begin{equation}\label{lambdaf}
\Lambda_n(b)=\frac{\Gamma(\frac\alpha2)}{2^{1-\alpha}\Gamma(1-\frac\alpha2)} \frac{(\frac\alpha2)_{n}}{n!}b^{n-1}F\Big(\frac\alpha2,n+\frac\alpha2,n+1,b^2\Big),
\end{equation}
which is obviously positive for all $\alpha,b\in(0,1)$. \\
 Let us now prove that the mapping $n\mapsto\Lambda_n(b)$ is decreasing. For this end,  we rewrite the hypergeometric series $F$ appearing in right-hand side of \eqref{lambdaf} according to the identity \eqref{integ}, which yields
 \begin{equation*}
F\Big(\frac\alpha2,n+\frac\alpha2,n+1,b^2\Big)=\frac{\Gamma(n+1)}{\Gamma(n+\frac\alpha2)\Gamma(1-\frac\alpha2)}\int_0^1 x^{n-1+\frac\alpha2} (1-x)^{-\frac\alpha2}(1-b^2x)^{-\frac\alpha2}~\mathrm dx.
 \end{equation*}
From the relation \eqref{Poc}  we get
 \begin{eqnarray}\label{43}
\Lambda_n(b)
 &=& \frac{b^{n-1}}{2^{1-\alpha}\Gamma^2(1-\frac\alpha2)}\int_0^1 x^{n-1+\frac\alpha2} (1-x)^{-\frac\alpha2}(1-b^2x)^{-\frac\alpha2}~\mathrm dx,\quad b\in(0,1).
 \end{eqnarray}
Therefore  it is easily seen that   $b\in(0,1)\mapsto \Lambda_n(b)$ is increasing and $n\in\NN^\star\mapsto \Lambda_n(b)$ is decreasing.  This implies in turn  that $n\in\NN^\star\mapsto \Theta_n=\Lambda_1(1)-\Lambda_n(1)$ is increasing and  thus it should be  positive. Notice that these properties can be also proven from the series \mbox{expansion \eqref{lambdaf}.}
\vspace{0,2cm}

{\bf (2)} 
Passing to the limit in the formula defining $\Lambda_n(b)$ when $\alpha$ goes to one yields
$$
\lim_{\alpha\to 1} \Lambda_n(b) =\frac{1}{b}\int_0^{+\infty}J_n(bt)J_n(t)dt.
$$
As to the second limit, we have
\begin{eqnarray*}
\lim_{\alpha\to 0} \Lambda_n(b) =\frac{1}{b}\int_0^{+\infty}\frac{J_n(bt)J_n(t)}{t}dt.
\end{eqnarray*}
Since $b\in(0,1)$  we can use the following identity,
$$
\int_0^{+\infty}\frac{J_n(bt)J_n(t)}{t}dt=\frac{b^n}{2n},
$$
whose proof can be found for example in \cite[p. 405]{Watson}. Consequently,
$$
\lim_{\alpha\to 0} \Lambda_n(b) =\frac{b^{n-1}}{2n}.
$$
 Now to compute the  limits of $\Theta_n$ when  $\alpha$ goes to the values $0$ and $1$ we shall rewrite $\Theta_n$ by using the identity\eqref{lambdaf} in the form
\begin{eqnarray*}
\Theta_n&\triangleq &\Lambda_1(1)-\Lambda_n(1) \\ &=& \frac{(\frac\alpha2)\Gamma(\frac\alpha2)}{2^{1-\alpha}\Gamma(1-\frac\alpha2)} \bigg[F\Big(\frac\alpha2,1+\frac\alpha2,2,1\Big)- \frac{(\frac\alpha2+1)_{n-1}}{n!}F\Big(\frac\alpha2,n+\frac\alpha2,n+1,1\Big)\bigg].
\end{eqnarray*}
This gives in view of the formula \eqref{id1} and \eqref{Poc},
\begin{equation}\label{theta}
\Theta_n= \frac{\Gamma(1-\alpha)}{2^{1-\alpha}\Gamma^2(1-\frac\alpha2)} \bigg(\frac{\Gamma(n+\frac\alpha2)}{\Gamma(2-\frac\alpha2)}-\frac{\Gamma(n+\frac\alpha2)}{\Gamma(n+1-\frac\alpha2)}\bigg).
\end{equation}
This expression coincides with the "eigenvalues" in the simply connected case, see \cite{H-H}. 
It follows that,
\begin{eqnarray*}
\lim_{\alpha\to 0}\Theta_n&=&\frac12\Big(\frac{\Gamma(1)}{\Gamma(2)}-\frac{\Gamma(n)}{\Gamma(n+1)}\Big)\\
&=&\frac12\Big(1-\frac{(n-1)!}{n!}\Big)\\
&=&\frac{n-1}{2n}\cdot
\end{eqnarray*}
We note that these values coincide with the "eigenvalues" for Euler equations in the simply connected case.
To compute the second limit, we shall introduce for a fixed $n$ the function
$$
\phi_n(\alpha)=\frac{\Gamma(n+\alpha/2)}{\Gamma(n+1-\alpha/2)}\cdot
$$
Therefore we obtain according to \eqref{Gamma1}, \eqref{for1} and the relation $\phi_n(1)=1,$
\begin{eqnarray*}
\lim_{\alpha\to 1}\Theta_n&=&\frac{-1}{\Gamma^2(1/2)}\lim_{\alpha\to 1}\big\{(1-\alpha)\Gamma(1-\alpha)\big\}\lim_{\alpha\to 1}\Big\{\frac{\phi_1(\alpha)-\phi_1(1)}{\alpha-1}-\frac{\phi_n(\alpha)-\phi_n(1)}{\alpha-1}\Big\}\\
&=&\frac{-1}{\pi}\Big\{{\phi_1^\prime(1)}-{\phi_n^\prime(1)}\Big\}.
\end{eqnarray*}
By applying the logarithm function to $\phi_n$ and differentiating with respect to $\alpha$ one obtains the relation
$$
2\frac{\phi_n^\prime(\alpha)}{\phi_n(\alpha)}=\digamma(n+\alpha/2)+\digamma(n+1-\alpha/2).
$$
Now using the fact that $\phi_n(1)=1$ combined with the preceding identity and \eqref{digam}, we find
\begin{eqnarray*}
\lim_{\alpha\to 1}\Theta_n
&=&\frac{-1}{\pi}\Big\{{\digamma(3/2)}-{\digamma(n+1/2)}\Big\}\\
&=&\frac{2}{\pi}\sum_{k=1}^{n-1}\frac{1}{2k+1},
\end{eqnarray*}
which is the desired result.
\vspace{0,5cm}

{\bf (3)-(4)} The asymptotic behavior of $\Lambda_n$ may be easily  obtained from the integral formula \eqref{43}. The proof of \eqref{As1} was done in details in \cite{H-H}. Finally, by combining \eqref{As1}, \eqref{As3} and the expression of $\textnormal{det}(M^\alpha_n)$ given by \eqref{cdet} one can deduce the identity \eqref{Asdet}.  
 \end{proof}

\subsection{Monotonicity of the eigenvalues}\label{subsec12}
In this section we shall discuss some important properties  concerning the monotonicity of the eigenvalues associated to the matrix $M^\alpha_{n}$ already seen  in Lemma \ref{lem0}. This will be crucial in the study of the kernel of the linearized operator $DG(\Omega,0,0)$. Recall that 
\begin{equation*}
M^\alpha_{n}=\begin{pmatrix}
 \Omega-\Theta_{n}+b^2\Lambda_{1}(b) & -b^2\Lambda_{n}(b) \\
  b\Lambda_n(b) & b\Omega+b^{1-\alpha}\Theta_{n}-b\Lambda_1(b)
\end{pmatrix}.
\end{equation*}
The determinant of this matrix given by \eqref{cdet} is a second order polynomial on the variable $\Omega$ and therefore it has two roots depending on all the parameters $n,b$ and $\alpha.$ For our deal it is important to formulate sufficient conditions to avoid  the eigenvalues  crossing in order to guarantee a one-dimensional kernel which is an essential assumption in Crandall-Rabionwitz's Theorem. In what follows we shall use the variable $\lambda\triangleq1-2\Omega$ instead of $\Omega$    in the spirit of the work of \cite{H-F-M-V}. Thus easy computations show that the determinant \eqref{cdet} takes the form,
 \begin{equation}\label{eqch2}
 \textnormal{det}\big(M_n^\alpha\big)=b\Big(\lambda^2-2C_n\lambda+D_n\big),
 \end{equation}
 with
 \begin{equation*}
 C_n\triangleq 1+\big(b^{-\alpha}-1\big)\Theta_n-\big(1-b^2\big)\Lambda_1(b),
 \end{equation*}
 and
 \begin{align}\label{dn}
 D_n\triangleq  -4b^{-\alpha}\Theta_n^2+2\Big[b^{-\alpha}-1+&2\big(1+b^{2-\alpha}\big)\Lambda_1(b)\Big]\Theta_n-4b^2\Big(\Lambda_1^2(b)-\Lambda_n^2(b)\Big)\\ -& 2\big(1-b^2\big)\Lambda_1(b)+1\notag.
 \end{align}
Note that the quantities $\Lambda_n(b)$ and $\Theta_n$ have been  introduced in Lemma \ref{lem0}. It is easy to check through straightforward computations that the reduced discriminant of the second order polynomial appearing in \eqref{eqch2} is given by 
\begin{eqnarray}\label{discriminant1}
\Delta_n&=& \Big((b^{-\alpha}+1)\Theta_{n}-(1+b^2)\Lambda_{1}(b)\Big)^2-4b^2\Lambda_{n}^2(b).
 \end{eqnarray}
Our result reads as follows.
\begin{proposition}  \label{lem2c2}
There exists $N\geq 2$ such that the following  holds true.

\begin{enumerate}
\item  For all $n>N$ we get $\Delta_n>0$ and the equation  \eqref{eqch2} admits two different  real solutions given by
$$
\lambda^\pm_n\triangleq C_n\pm\sqrt{\Delta_n}.
$$
\item The sequences $(\Delta_n)_{n\geq N}$ and $(\lambda_n^+)_{n\geq N}$ are strictly increasing and $(\lambda_n^-)_{n\geq N}$ is strictly decreasing.
\item For all $m>n >N$  we have
$$
\lambda_m^-<\lambda_n^-<\lambda_n^+<\lambda_m^+.
$$
\end{enumerate}
\end{proposition}
\begin{remark}\label{rmq1}
\begin{enumerate}
\item The number  $N$ in $(2)$ of  the previous proposition is the smallest integer satisfying,
\begin{equation}\label{pos}
\Theta_{N}\geq\frac{1+b^2}{b^{-\alpha}+1}\Lambda_{1}(b)+\frac{2b}{b^{-\alpha}+1}\Lambda_{N}(b). 
\end{equation}
\item In the known cases like the simply connected case with $\alpha\in [0,2[$  or the doubly connected case with $\alpha=0$ 
the analysis is more easier because the dispersion relation is a kind of fractional polynomial but in the present case it is highly nonlinear with respect to the frequencies and the parameters $\alpha$ and $b.$ Therefore the program is achieved with only a sufficient condition on the existence of the eigenvalue and which is given by \eqref{pos}. This condition coincides with that given in \cite{H-F-M-V} for $\alpha=0$. 
%
\end{enumerate}
\end{remark}

\begin{proof}
{\bf (1) }  We intend  to discuss the conditions leading to the positivity of the discriminant defined in \eqref{discriminant1} which ensures in turn that the polynomial \eqref{eqch2} has two real solutions. 
We can see that $(\Delta_n)_{n\geq 2}$ can be extended to a smooth function defined on $[1,+\infty[$ as follows
 \begin{eqnarray*}
 \Delta_x&=&\Big((b^{-\alpha}+1)\Theta_{x}-(1+b^2)\Lambda_{1}(b)\Big)^2-4b^2\Lambda_{x}^2(b).
  \end{eqnarray*}
It is strictly positive if and only if 
\begin{equation}\label{cond1ch2}
E_b(x)\triangleq(b^{-\alpha}+1)\Theta_{x}-(1+b^2)\Lambda_{1}(b)-2b\Lambda_{x}(b)> 0,
\end{equation}
or
\begin{equation*}
F_b(x)\triangleq(b^{-\alpha}+1)\Theta_{x}-(1+b^2)\Lambda_{1}(b)+2b\Lambda_{x}(b)< 0.
\end{equation*}
Remark first that for  $x$ and $b$ verifying the inequality \eqref{cond1ch2} we have $F_b(x)>0$. On other hand,  from  Lemma \ref{lem1} one has :
for all  $b\in(0,1)$  the mapping $x\mapsto E_b(x)$ is strictly increasing, continuous and satisfying  
\begin{eqnarray*}
\lim_{x\to+\infty} E_b(x)&=&(b^{-\alpha}+1)\Lambda_1(1)-(1+b^2)\Lambda_{1}(b)\\
&\geq& \big(b^{-\alpha}-b^2)\Lambda_1(1)>0,
\end{eqnarray*}
$$
E_b(1)=-(1+b)^2\Lambda_1(b)<0.
$$
 Consequently the set 
 $$
 \mathcal{I}^\alpha(b)\triangleq \big\{ x>1;\; E_b(x)>0\big\}.
 $$ 
 is connected and takes  the form $]\beta,+\infty[$, with
$$ 
E_b(\beta)=0.
$$
Hence the integer $N$ is chosen as
$$
N=[\beta].
$$
\vspace{0,2cm}

{\bf (2)} To prove that $x\mapsto\Delta_x$ is decreasing we shall compute its  derivative with respect to $x$. Pain computations give,
\begin{eqnarray*}
\partial_x \Delta_x &=& 2(b^{-\alpha}+1)\Big[(b^{-\alpha}+1)\Theta_x-\big(1+b^{2}\big)\Lambda_1(b)\Big]{\partial_x  \Theta_x}-8b^2\Lambda_x(b){\partial_x \Lambda_x(b)}.
\end{eqnarray*}
Since $x\mapsto\Lambda_x(b)$ is decreasing, $x\mapsto\Theta_x$ is increasing  and $\Lambda_x(b)\geq0$  (see (1) from Lemma \ref{lem1}) then we deduce  that
\begin{align*}\label{revch2}
{\partial_x \Delta_x}&> 2(b^{-\alpha}+1)\Big[(b^{-\alpha}+1)\Theta_x-\big(1+b^{2}\big)\Lambda_1(b)\Big]{\partial_x  \Theta_x}\\ &> 2(b^{-\alpha}+1)E_x(b){\partial_x  \Theta_x}.
\end{align*}
Hence, for all $x\in  \mathcal{I}^\alpha(b)$ we have
$$
{\partial_x \Delta_x}>0.
$$
This shows that  $x\mapsto \Delta_x$ is strictly  increasing.\\
Now, recall that 
$$
\lambda_x^+\triangleq 1+(b^{-\alpha}-1)\Theta_x-(1-b^2)\Lambda_1(b)+\sqrt{\Delta_x}.
$$
Using the fact that $x\mapsto\Theta_x$ is increasing (according to Lemma \ref{lem1})  combined  with the increasing property of $x\mapsto \Delta_x$ we get the desired result.

So it remains to establish that the mapping  $x\mapsto\lambda_x^-$ is strictly decreasing. For this aim we calculate  
its derivative   with respect to $x$,
\begin{alignat*}{2}
{\partial_x \lambda_x^-}&=&& (b^{-\alpha}-1)\partial_x\Theta_x-\dfrac{\partial_x{ \Delta_x}}{2\sqrt{\Delta_x}}
\\ &= &&(b^{-\alpha}-1){\partial_x  \Theta_x}-(b^{-\alpha}+1)\Big[\frac{(b^{-\alpha}+1)\Theta_x-\big(1+b^{2}\big)\Lambda_1(b)}{\sqrt{\Delta_x}}\Big]{\partial_x  \Theta_x}+4b^2{\partial_x \Lambda_x(b)}\dfrac{\Lambda_x(b)}{\sqrt{\Delta_x}}\\ &<&& 
b^{-\alpha}\Big[1-\frac{(b^{-\alpha}+1)\Theta_x-\big(1+b^{2}\big)\Lambda_1(b)}{\sqrt{\Delta_x}}\Big]\partial_x  \Theta_x\\ & &&-\Big[1+\frac{(b^{-\alpha}+1)\Theta_x-\big(1+b^{2}\big)\Lambda_1(b)}{\sqrt{\Delta_x}}\Big]\partial_x  \Theta_x.
\end{alignat*}
where we have used in the last inequality the decreasing property of the mapping  $x\mapsto\Lambda_x(b)$   and the fact that $\Lambda_x(b)\geq0$ . Now from the inequality
$$
(b^{-\alpha}+1)\Theta_x-\big(1+b^{2}\big)\Lambda_1(b)>E_x(b)>0,
$$
and  the expression of $\Delta_x$ we deduce that
\begin{alignat*}{2}
{\partial_x \lambda_x^-} &<&& 0.
\end{alignat*}
This gives the desired result.
\vspace{0,5cm}

{\bf (3)} This follows easily from (2) and the obvious fact
$$
\lambda_n^-\leq \lambda_n^+.
$$
\end{proof}
\section{Bifurcation at simple eigenvalues}\label{sec45}
In this section we shall prove Theorem \ref{main} which is deeply related to the spectral study developed in the preceding section combined with Crandall-Rabinowitz's Theorem. To construct the function  spaces where the bifurcation occurs we shall take into account the restriction to  the high  frequencies stated in Proposition \ref{lem2c2} and include the $m$-fold symmetry of the V-states. To proceed,  fix $b \in (0,1)$ and $m>N$, where $N$ is defined in \mbox{Proposition \ref{lem2c2}}  and  Remark \ref{rmq1}.
Set,
$$
X_m=C^{2-\alpha}_{m}(\mathbb{T})\times C^{2-\alpha}_{m}(\mathbb{T}),
$$
where $C^{2-\alpha}_{m}(\mathbb{T})$ is the space of the $2\pi-$periodic functions $f\in C^{2-\alpha}(\mathbb{T})$ whose  Fourier series is given by 
$$
f(w)=\sum_{n=1}^{\infty}a_{n}\overline{w}^{nm-1},\quad w\in\mathbb{T},\quad a_n\in \mathbb{R}.
$$
This space is equipped with its  usual norm. We define the ball of radius $r\in(0,1)$ by
$$
B_r^m=\Big\{f\in X_m, \,\|f\|_{C^{2-\alpha}(\mathbb{T})}\le r\Big\}
$$
and we introduce  the neighborhood of zero,
$$V_{m,r}\triangleq B_r^m\times B_r^m.
$$
The set $V_{m,r}$ is endowed with the  induced topology of the product spaces.\\
 Take  $(f_1,f_2)\in V_{m,r}$ then  the expansions of the associated conformal mappings $\phi_1, \phi_2$
outside the unit disc $\big\{z \in \CC; \, \vert z\vert \geq 1\big\}$ are given successively  by
$$
\phi_1(z)=z+f_1(z)=z\Big(1+\sum_{n=1}^{\infty}\frac{a_{n}}{z^{nm}}\Big)
$$
and
$$
\phi_2(z)=bz+f_2(z)=z\Big(b+\sum_{n=1}^{\infty}\frac{b_{n}}{z^{nm}}\Big).
$$
This structure  provides the $m-$fold symmetry of the associated boundaries $\phi_1(\mathbb{T})$ and $\phi_2(\mathbb{T})$, via the relation
\begin{equation}\label{mfo1}
\phi_j\big(e^{2i\pi/m}z\big)=e^{2i\pi/m}\phi_j(z),\quad j=1,2 \quad\textnormal{and}\quad\vert z\vert\geq1.
\end{equation}
For functions $f_1$ and $f_2$ with small size  the boundaries can be seen as a small perturbation of the boundaries of the  annulus $\big\{z\in \CC;\, b\le |z|\leq 1 \big\}.$
Set
$$
H_m=\Bigg\{g\in C^{1-\alpha}(\mathbb{T}),\, g(w)=i\sum_{n\geq 1}A_{n}\big(w^{mn}-\overline{w}^{mn}\big),\, A_n\in\RR,\, n\in\NN^*\Bigg\}
$$
and define the product space $Y_m$ by
$$
Y_m=H_m\times H_m.
$$
From Proposition \ref{lem2c2} recall the definition of the eigenvalues $\lambda_m^{\pm}$ and the associated angular velocities are
\begin{eqnarray*}
\Omega_m^{\pm}&=&\frac{1}{2}-\frac12\lambda_m^{\pm}\\
&=&\frac12\widehat{C}_m\mp \frac12\sqrt{\Delta_m}
\end{eqnarray*}
with
$$
\Delta_m\triangleq \Big((b^{-\alpha}+1)\Theta_{m}-(1+b^2)\Lambda_{1}(b)\Big)^2-4b^2\Lambda_{m}^2(b)
$$
and
$$
\widehat{C}_m\triangleq (1-b^{-\alpha})\Theta_m+(1-b^2)\Lambda_1(b).
$$
Note that $\Theta_m$ and $\Lambda_m(b)$ were introduced in Lemma \ref{lem0}. The V-states equations are described  in \eqref{g_1} and \eqref{Rotaeq1} which we restate here: for $j\in\{1,2\}$,
$$
F_j\big(\Omega,\phi_1,\phi_2\big)(w)\triangleq G_j\big(\Omega, f_1, f_2\big)(w)=0,\quad \forall\, w\in \mathbb{T}; \quad\hbox{and}\quad G\triangleq(G_1,G_2).
$$
with 
\begin{eqnarray*}
F_j\big(\Omega,\phi_1,\phi_2\big)(w)&\triangleq& \Omega\,\Ima\Big\{\phi_j(w)\overline{w}\,{\overline{\phi_j'(w)}}\Big\}\\
&+&{C_\alpha}\Ima\Bigg\{\bigg(\fint_\mathbb{T}\frac{\phi_2'(\tau)d\tau}{\vert \phi_j(w)-\phi_2(\tau)\vert^\alpha}-\fint_\mathbb{T}\frac{\phi_1'(\tau)d\tau}{\vert \phi_j(w)-\phi_1(\tau)\vert^\alpha}\bigg) \overline{w}\,{\overline{\phi_j'(w)}}\Bigg\}.
\end{eqnarray*}

\vspace{0,5cm}
Now, to apply Crandall-Rabinowitz's Theorem   it suffices to show the following result.
\begin{proposition}\label{prozq}  Let  $N$ be as in the part $(1)$ of Proposition $\ref{lem2c2}$ and   $m\geq N,$ and take  $\Omega\in\big\{\Omega_m^{\pm}\big\}$. Then, the following assertions hold  true. 
\begin{enumerate}
\item There exists $r>0$ such that $G:\mathbb{R}\times V_{m,r}\to Y_m$ is well-defined and of class $C^1$.
\item   The kernel of $DG(\Omega,0,0)$ is one-dimensional  and generated by 
\[v_{0,m}:\,w\in \mathbb{T}\mapsto \left( {\begin{array}{cc}
   \Omega+b^{-\alpha}\Theta_m-\Lambda_1(b)\\
 -\Lambda_m(b)\\
  \end{array} } \right)\overline{w}^{m-1}.
  \]
\item  The range of  $DG(\Omega,0,0)$ is closed  and is  of co-dimension one in $Y_m$. 
\item Transversality assumption: If $\Omega$ is a simple eigenvalue $(\Delta_m> 0)$ then
$$
\partial_\Omega D G(\Omega_m^\pm,0,0)v_{0,m}\not\in R\big(DG(\Omega_m^\pm,0,0)\big)\cdot
$$
\end{enumerate}
\end{proposition} 
\begin{proof}

{\bf (1)} Compared to Proposition \ref{reg} we need just  to check that $G=(G_1,G_2)$ preserves the $m-$fold symmetry and  maps $X_m$ into $Y_m$. For this end, it is sufficient to check  that for  given $(f_1,f_2)\in X_m$, the coefficients  of the Fourier series of $F_j(\Omega, \phi_1,\phi_2) $ vanish at frequencies which are not integer multiple of $m$. This amounts to proving that,
$$
F_j(\Omega, \phi_1,\phi_2) \Big(e^{i\frac{2\pi}{m}}w\Big)=F_j(\Omega, \phi_1,\phi_2)(w), \quad w\in\mathbb{T},\quad j=1,2\cdot
$$
This property is obvious for the first term  $ \displaystyle{\textnormal{Im}\Big\{\overline{w}\overline{\phi_j'(w)}{\phi_j(w)}\Big\}}$. For the two last terms in the expression of  $F_j$ it is enough to check the identity,
$$
\forall w\in \mathbb{T},\quad \Phi_j\big(e^{i\frac{2\pi}{m}}w\big)=e^{i\frac{2\pi}{m}}\Phi_j(w),
$$
with
\begin{eqnarray*} 
\Phi_j(w)&\triangleq& \fint_{\mathbb{T}}\frac{\phi_2'(\tau)}{\vert \phi_j(w)-\phi_2(\tau)\vert^\alpha}d\tau
\end{eqnarray*}
This follows  easily by making  the change of variables $\tau=e^{i2\pi/m}\zeta$ and from  \eqref{mfo1}. Indeed,
\begin{eqnarray*}
\Phi_j\big(e^{i\frac{2\pi}{m}}w\big)&=&e^{i2\pi/m}\fint_{\mathbb{T}}\frac{\phi_2'(e^{i2\pi/m}\zeta)}{\vert \phi_j(e^{i\frac{2\pi}{m}}w)-\phi_2(e^{i2\pi/m}\zeta)\vert^\alpha}d\zeta\\
&=&e^{i\frac{2\pi}{m}}\fint_{\mathbb{T}}\frac{\phi_2'(\zeta)}{\vert \phi_j(w)-\phi_2(\zeta)\vert^\alpha}d\zeta\\
&=&e^{i\frac{2\pi}{m}}\Phi_j(w).
\end{eqnarray*}
 This concludes the proof of the following statement,
$$
(f_1,f_2)\in V_{m,r}\Longrightarrow G(\Omega, f_1,f_2)\in Y_m.
$$

\vspace{0,5cm}
{\bf (2) } We shall describe the kernel of  the linear operator  $DF(\Omega_m^\pm,0,0)$ and show that it is one-dimensional. Let  $h_1,h_2$ be two functions in $C^{2-\alpha}_m\big(\mathbb{T}\big)$ such that
 \begin{equation}\label{ds1}
h_1(w) =\sum_{n=1}^{\infty}{a_{n}}{\overline{w}^{nm-1}}\quad\textnormal{and}\quad h_2(w)= \sum_{n=1}^{\infty}{c_{n}}{\overline{w}^{nm-1}}, \quad w\in\mathbb{T},
 \end{equation}
Recall from Lemma \ref{lem0} the following expression,
\begin{eqnarray}\label{dfff}
DG(\Omega,0,0)\big(h_1,h_2\big)(w)=\frac{i}{2}\sum_{n\geq 1}nm M^\alpha_{nm}\left( \begin{array}{c}
a_n \\
c_n
\end{array} \right)\Big(w^{nm}-\overline{w}^{nm}\Big).
\end{eqnarray}
where the matrix $M_n$ is given for $n\geq2$  by
\begin{equation}\label{matrixxs}
M^\alpha_{n}\triangleq\begin{pmatrix}
 \Omega-\Theta_{n}+b^2\Lambda_{1}(b) & -b^2\Lambda_{n}(b) \\
  b\Lambda_n(b) & b\Omega+b^{1-\alpha}\Theta_{n}-b\Lambda_1(b)
\end{pmatrix}.
\end{equation}
Now, if $\Omega\in\big\{\Omega_m^{\pm} \big\}$ then
\begin{equation*}
\textnormal{det}\big(M_m^\alpha\big)=0.
\end{equation*}
Thus, the kernel of $DG(\Omega,0,0)$ is non trivial and it is  one-dimensional if and only if
\begin{equation}\label{detmn}
\textnormal{det}\big(M_{nm}^\alpha\big)\neq0,\quad \forall\, n\geq2.
\end{equation}
This condition is ensured by the part $(1)$ of the  Proposition \ref{lem2c2}.
Then,  $(h_1,h_2)$ is in the kernel of $DG\big(\Omega,0,0)$ if and only the Fourier coefficients in the identity \eqref{dfff} vanish, namely,
$$
a_n=c_n=0\quad\textnormal{for all}\quad n\geq 2\quad\textnormal{and}\quad (a_1,c_1)\in\textnormal{Ker}M_m^\alpha.
$$
Hence, a generator of Ker$\Big(DG(\Omega,0,0)\Big)$ can be chosen as  the pair of functions
\begin{equation}\label{gen}
w\in \mathbb{T}\mapsto \left( {\begin{array}{cc}
   \Omega+b^{-\alpha}\Theta_m-\Lambda_1(b)\\
 -\Lambda_m(b)\\
  \end{array} } \right)\overline{w}^{m-1},\quad w\in\mathbb{T}.
\end{equation}

\vspace{0,5cm}
{\bf (3) }
We are going to show that for any $m \geq N$ the range $R\big(DG(\Omega,0,0)\big)$ coincides with the space  of the functions $(g_1,g_2)\in C^{1-\alpha}(\mathbb{T})\times C^{1-\alpha}(\mathbb{T})$ such that
\begin{equation}\label{g1g2}
g_{1}(w)=\sum_{n\geq 1}iA_{n}\big(w^{mn}-\overline{w}^{mn}\big),\quad g_{1}(w)=\sum_{n\geq 1 }iC_{n}\big(w^{mn}-\overline{w}^{mn}\big)
\end{equation}
where  $A_{n},C_{n}\in\RR$ for all $n\in\NN^*$ and there exists $(a_{1},c_{1})\in \RR^2$ such that
\begin{equation}\label{cm}
M_{m}\left( \begin{array}{c}
a_{1} \\
c_{1}
\end{array} \right)=\left( \begin{array}{c}
A_{1} \\
C_{1}
\end{array} \right).
\end{equation} 
For the sake of simple notation we remove in this part the parameter  $\alpha$ from $M_n^\alpha.$
The range of operator $DG(\Omega,0,0)$ is obviously included in the space defined above which is closed and of   co-dimension $1$ in $Y_m$. Therefore it remains to check just the converse. Let $g_1$ and $g_2$   be two functions in $C^{1-\alpha}(\mathbb{T})$ with Fourier series expansions as in \eqref{g1g2} and \eqref{cm}. We shall prove that the equation
$$
\frac1m DG(\Omega,0,0)(h_1,h_2)=(g_1,g_2)
$$
admits a solution $(h_1, h_2)$ in the space $X_m,$ where  the Fourier series expansions  of these functions are given in \eqref{ds1}.
Then according to \eqref{dfff}, the preceding equation  is equivalent to
$$
n \,M_{mn}\left( \begin{array}{c}
a_{n} \\
c_{n}
\end{array} \right)=\left( \begin{array}{c}
A_{n} \\
C_{n}
\end{array} \right),\quad \forall n\in\NN^\star.
$$
For $n=1$, the existence follows from the condition \eqref{cm} and therefore we shall only focus on $n\geq2.$ Owing to \eqref{detmn} the sequences $(a_{n})_{n\geq2}$ and $(c_{n})_{n\geq 2}$ are uniquely determined by  the formulae
\begin{equation}\label{m1}
\left( \begin{array}{c}
a_{n} \\
c_{n}
\end{array} \right)=\frac{1}{n}M_{mn}^{-1}\left( \begin{array}{c}
A_{n} \\
C_{n}
\end{array} \right),\quad n\geq2.
\end{equation}
By computing the matrix $M_{mn}^{-1}$  we deduce that for all $n\geq2,$  
\begin{eqnarray*}
a_{n}&=&\frac{b\big(\Omega+b^{-\alpha}\Theta_{nm}-\Lambda_1(b)\big)}{n\,\textnormal{det}\big({M}_{nm}\big)}A_{n}+\frac{b^2\Lambda_{nm}(b)}{n\,\textnormal{det}\big(M_{nm}\big)}C_{n}
\end{eqnarray*}
and
\begin{eqnarray*}
c_{n}&=&-\frac{b\,\Lambda_{nm}(b)}{n\,\textnormal{det}\big(M_{nm}\big)}A_{n}+\frac{\big(\Omega-b\Theta_{nm}+b^2\Lambda_1(b)\big)}{n\,\textnormal{det}\big(M_{nm}\big)}C_{n}.
\end{eqnarray*}
Therefore the proof of $(h_1,h_2)\in X_m$ amounts to showing that  
$$
w\mapsto \left( \begin{array}{c}
h_1(w)-a_1\overline{w}^{m-1} \\
h_2(w)-c_1\overline{w}^{m-1}
\end{array} \right)\in C^{2-\alpha}(\mathbb{T})\times C^{2-\alpha}(\mathbb{T}).
$$
We shall develop the computations only for the first component and the second one can be done in a similar way. 
We set ${\displaystyle \widetilde{h}_1(w)=h_1(w)-a_1\overline{w}^{m-1}}$ and 
$$
 H(w)\triangleq \sum_{n\geq 2} \frac{A_{n}}{n\,\textnormal{det}\big(M_{nm}\big)}{w}^{n},\quad  H_1(w)\triangleq \sum_{n\geq 2} \frac{C_{n}}{n}{w}^{n}\cdot
$$
Then in view of \eqref{As1} the function  $\widetilde{h}_1(w)$ can be rewritten as follows 
\begin{eqnarray*}
\widetilde{h}_1(w) &=&  C_1\, wH(\overline{w}^m)+C_2\, w\big(H*K_1\big)(\overline{w}^m)\\ &+&w\,\big(H*K_2\big)(\overline{w}^m)+b^2w\big(H_1*K_3\big)(\overline{w}^m),
\end{eqnarray*}
with $C_1$ and $C_2$ two constants. The kernels are defined by  
$$
K_1(w)\triangleq \sum_{n\geq 2} \frac{w^n}{n^{1-\alpha}},\quad 
K_2(w)\triangleq \sum_{n\geq 2} O\Big(\frac{1}{n^{2-\alpha}}\Big)w^n,
$$
and
$$
K_3(w)\triangleq \sum_{n\geq 2} \frac{\Lambda_{nm}(b)}{\textnormal{det}\big(M_{nm}\big)}w^n.
$$
The convolution  is understood in the usual one: for two continuous  functions $f,g;\mathbb{T}\to \CC$ we define
$$
\forall w\in \mathbb{T},\quad f*g(w)=\fint_{\mathbb{T}}f(\tau) g({\tau}\overline{w})\frac{d\tau}{\tau}\cdot
$$
Assume for a while that $w\mapsto H(w)$ belongs to $C^{2-\alpha}(\mathbb{T})$. Then by virtue of the classical convolution law $L^1(\mathbb{T})*C^{2-\alpha}(\mathbb{T})\to C^{2-\alpha}(\mathbb{T})$ , it suffices to show that the kernels  $K_1, K_2$ and $K_3$  belong to $L^1(\mathbb{T})$. The second and the third kernels are easy to analyze because the series converge   absolutely,
$$
\Vert K_2\Vert_{L^{\infty}(\mathbb{T})}\lesssim\sum_{n\geq 1}\frac{1}{n^{2-\alpha}}\leq C\cdot
$$
Similarly, owing to \eqref{As3} one has
$$
\Vert K_3\Vert_{L^{\infty}(\mathbb{T})}\lesssim\sum_{n\geq 0}b^{n}\leq C\cdot
$$
and therefore $K_2,K_3\in L^1(\mathbb{T})$. Note that to bound the series we have used the fact that the sequence $\big(\textnormal{det} M_{nm}\big)_{n\geq2}$ does not vanish and converges to a strictly positive number  $D_\infty$ defined \mbox{in \eqref{LimitD}.} It remains to show that $K_1\in L^1(\mathbb{T})$. For this end we shall use the following estimate: for any $\beta\in \big(\alpha,1\big)$
\begin{equation}\label{kern1}
|K_1(e^{i\theta})|\lesssim\frac{1}{\sin^\beta(\frac\theta2)},\quad\forall \theta\in (0,2\pi)\cdot
\end{equation}
which is true for all $\alpha\in[0,1[$ and for a proof we can see  \cite{H-H}.

Now to complete  the reasoning it remains to prove the preceding claim asserting that the function $H$ belongs to the space $C^{2-\alpha}(\mathbb{T})$.
To prove this  we write in view of \eqref{Asdet},
\begin{eqnarray*}
\textnormal{det}(M_{nm})&=&\mu+\frac{\nu}{n^{1-\alpha}}+O\Big(\frac{1}{n^{2-\alpha}}\Big)\\
&=&\mu-\rho_n,
\end{eqnarray*}
where $\mu$ and $\nu$ are two constants (depending on $m$)  and\begin{equation}\label{Ass11}
 \rho_n\triangleq-\frac{\nu}{n^{1-\alpha}}+O\Big(\frac{1}{n^{2-\alpha}}\Big)\cdot
\end{equation} 
This allows to get
 \begin{eqnarray*}
H(w)&=& \sum_{ n\geq 2}\frac{A_{n}}{n\big(\mu-\rho_n\big)}{w}^{n}.
\end{eqnarray*}
Then one may use  the general  decomposition: for $k\in \NN,$
$$
\frac{1}{\mu-\rho_n}=\frac{\mu^{-k-1}\rho_n^{k+1}}{\mu-\rho_n}+\sum_{j=0}^k\mu^{-j-1}{\rho_n^j}
$$
which yields,
 \begin{eqnarray*}
H(w)&=& \mu^{-k-1}\sum_{ n\geq 2}\frac{A_{n}\rho_n^{k+1}}{n\big(\mu-\rho_n\big)}{w}^{n}+\sum_{j=0}^{k}\mu^{-j-1}\sum_{n\geq 2}\frac{A_{n}\rho_n^j}{n}{w}^{n}\\ &\triangleq &  \mu^{-k-1}H_{k+1}({w})+\sum_{j=0}^k\mu^{-j-1}L_j(w)\cdot
\end{eqnarray*}
Since the sequence $(A_{n})_{n\geq 2}$ is bounded then by \eqref{Ass11} we obtain
$$
\bigg\vert \frac{A_{n}\rho_n^{k+1}}{n\big(\mu-\rho_n\big)}\bigg\vert\lesssim \frac{\big\vert\rho_n\big\vert^{k+1}}{n}\lesssim\frac{1}{n^{1+(1-\alpha)(k+1)}}\cdot
$$
Thus for $k$ large enough we get $H_{k+1}\in C^{2-\alpha}(\mathbb{T}).$ Concerning the estimate of $L_j$ we shall restrict the analysis to $j=0$ and $j=1$ and the higher  terms can be treated in a similar way. We write
$$
L_0(w)=\sum_{n\geq 2}\frac{A_{n}}{n}w^n.
$$
Using the Cauchy-Schwarz inequality we get
\begin{eqnarray*}
\Vert L_0\Vert_{L^\infty}&\le & \sum_{n\geq 2}\frac{\vert A_{n}\vert}{n}\\ &\le & \bigg(\sum_{n\geq 1}\frac{1}{n^2}\bigg)^{1/2} \Big(\sum_{n\geq 2}\vert A_{n}\vert^2\Big)^{1/2}\\ & \lesssim &\Vert g_1\Vert_{L^2}.
\end{eqnarray*}
By the embedding $ C^{1-\alpha}\big(\mathbb{T}\big)\hookrightarrow L^\infty\big(\mathbb{T}\big) \hookrightarrow L^2\big(\mathbb{T}\big)$ we conclude that
$$
\Vert L_0\Vert_{L^\infty} \lesssim \Vert g_1\Vert_{1-\alpha}.
$$
It remains to prove that $L_0^\prime\in C^{1-\alpha}(\mathbb{T})$. For this end, one needs first to check that we can differentiate   the series term by term.
Fix $N\geq 1$ and define
$$
L_0^N(w)\triangleq\sum_{n= 2}^N\frac{A_{n}}{n}{w}^{n}.
$$
From Cauchy-Schwarz inequality  we find
\begin{eqnarray*}
\|L_0^N-L_0\|_{L^\infty(\mathbb{T})}&\lesssim &\bigg(\sum_{n\geq N+1}\frac{1}{n^2}\bigg)^{1/2}\Vert g_1\Vert_{1-\alpha}\\ &\lesssim &\frac{\Vert g_1\Vert_{1-\alpha}}{N^{1/2}}\cdot
\end{eqnarray*}
Hence,
\begin{equation}\label{uni1}
\lim_{N\to+\infty}\|L_0^N-L_0\|_{L^\infty(\mathbb{T})}=0.
\end{equation}
Differentiating $L_0^N$ term by term one gets
\begin{eqnarray*}
(L_0^N)^\prime(w)&=&\overline{w}\sum_{n= 2}^N{A_{n}}{w}^{n}\\
&\triangleq &\overline{w}\,  g^N_1(w).
\end{eqnarray*}
Put
$$
g_1^+(w)=\sum_{n\geq2}{A_{n}}{w}^{n},
$$
then  using the continuity of Szeg\"o protection:
 \[\Pi:\sum_{n\in \mathbb{Z}}a_n w^n\mapsto \sum_{n\in \mathbb{N}}a_n w^n
 \]
on H\"{o}lder spaces $C^{1-\alpha}(\mathbb{T})$ for $ \alpha\in (0,1)$ we may conclude that  $g_1^+$ belongs to $C^{1-\alpha}(\mathbb{T})$, (for more details see for example \cite{H-H}). By virtue of a classical result on Fourier series one gets
\begin{equation*}
\lim_{N\to+\infty}\|g_1^N-g_1^+\|_{L^\infty(\mathbb{T})}=0
\end{equation*}
and consequently 
\begin{equation}\label{uni2}
\lim_{N\to+\infty}\|(L_0^N)^\prime-\overline{w}\,  g_1^+\|_{L^\infty(\mathbb{T})}=0.
\end{equation}
Putting together \eqref{uni1} and \eqref{uni2} we deduce  that $L_0$ is differentiable and
$$
L_0^\prime(w)=\overline{w}\,  g_1^+(w),\quad w\in \mathbb{T}.
$$
This concludes that $L_0\in C^{2-\alpha}.$
Now, 
as before,  we can easily get  $L_1\in L^\infty(\mathbb{T})$ and we shall  check that $L_1^\prime \in C^{1-\alpha}\big(\mathbb{T}\big)$.  Arguing in a similar way to $L_0$ we can differentiate term by term the series defining $L_j$ leading to 
$$
L'_1(w)= \overline{w}\sum_{n\geq 2}{A_{n}\rho_n}w^{n}.
$$
Note that with the same kernels $K_1$ and $K_2$ as before one can write 
$$
{w}{L'}_1(w)=-\nu\big(K_1*g_1^+\big)(w)+\big(K_2*g_1^+\big)(w).
$$
Using the fact that $g_1^+$ belongs to $C^{1-\alpha}(\mathbb{T})$ and $K_1,K_2\in L^1(\mathbb{T})$ we obtain  the desired result.\\

\vspace{0,5cm}
{\bf (4) } {\it  The transversality condition.}
Let $\Omega\in \{\Omega_m^\pm\}$ be a simple eigenvalue associated to the frequencies $m$ and $v_{0,m}$ be the generator of the kernel $DG(\Omega,0,0)$ defined in the part $(2)$ of Proposition \ref{prozq}. We shall prove that
$$
{\partial_\Omega}DG(\Omega,0,0)v_{0,m}\notin R\big(DG(\Omega,0,0)\big),
$$
with
\begin{equation*}
v_{0,m}(w)= \left( {\begin{array}{cc}
   \Omega+b^{-\alpha}\Theta_m-\Lambda_1(b)\\
 -\Lambda_m(b)\\
  \end{array} } \right)\overline{w}^{m-1},\quad w\in\mathbb{T}.
\end{equation*}
Differentiating   \eqref{df001} and \eqref{df002}  with respect to $\Omega$ we get 
$$
{\partial_\Omega}DG_1(\Omega,0,0)(h_1,h_2)(w)=\Ima\Big\{\overline{h_1'(w)}+\overline{w}\,{h_1(w)}\Big\}
$$
and
$$
{\partial_\Omega}DG_2(\Omega,0,0)(h_1,h_2)(w)=b\Ima\Big\{\overline{h_2'(w)}+\overline{w}\,{h_2(w)}\Big\}.
$$
Hence,
\begin{eqnarray*}
\partial_\Omega DG(\Omega,0,0)(v_{0,m})(w)&=&\frac m2 i\, \left( {\begin{array}{cc}
   \Omega+b^{-\alpha}\Theta_m-\Lambda_1(b)\\
 -b\Lambda_m(b)\\
  \end{array} } \right)\big(w^{m}-\overline{w}^m\big)\\
  &\triangleq& \frac m2 i\, \, e_m\;\big(w^{m}-\overline{w}^m\big).
\end{eqnarray*}
This pair of functions is in the range of $DG(\Omega,0,0)$ if and only if the vector $e_m\in\RR^2$ is a scalar multiple of one column of the matrix $M_m^\alpha$ seen in \eqref{matrixxs}, which happens if and only if
\begin{equation}\label{haq1}
\Big(\Omega+b^{-\alpha}\Theta_m-\Lambda_1(b)\Big)^2-b^2\Lambda_m^2(b)=0.
\end{equation}
Combining this equation with $\hbox{det}\,M_m^\alpha=0$ we get
$$
 \Big(\Omega-\Theta_{m}+b^2\Lambda_{1}(b)\Big)\Big(\Omega+b^{-\alpha}\Theta_{m}-\Lambda_1(b)\Big)+\Big(\Omega+b^{-\alpha}\Theta_m-\Lambda_1(b)\Big)^2=0.
$$
This yields
$$
\Big(\Omega+b^{-\alpha}\Theta_{m}-\Lambda_1(b)\Big)\Big[2\Omega+(b^2-1)\Lambda_1(b)+(b^{-\alpha}-1)\Theta_m\Big]=0
$$
which is equivalent to 
$$
\Omega+b^{-\alpha}\Theta_{m}-\Lambda_1(b)=0 \quad\hbox{or}\quad 
\Omega=\frac12\Big({\big(1-b^2\big)}\Lambda_1(b)+{\big(1-b^{-\alpha}\big)}\Theta_m\Big).
$$
This first possibility is excluded by \eqref{haq1} because $\Lambda_m(b)\neq0$ and the second one is also impossible  because it corresponds to a multiple eigenvalue which is not the case here. This concludes the proof of Proposition \ref{prozq}.

\end{proof}

 \section{Numerical study of $V$-states}\label{Sec-num}

Even if there is a number of references on the numerical obtention of rotating $V$-states for the vortex patch problem (see for instance \cite{DZ} and \cite{DR}, and more recently \cite{H-F-M-V}), up to our knowledge nothing similar has been done for the quasi-geostrophic problem. Therefore, for the sake of completeness, we will discuss in this section the numerical obtention of $V$-states for the quasi-geostrophic problem in both the simply-connected case and the doubly-connected case. Since the procedure is very similar to that in the vortex patch problem, we will omit some details, which can be consulted in \cite{H-F-M-V}.

\subsection{Simply-connected $V$-states}

We gather the main theoretical arguments from \cite{H-H}. Given a simply-connected domain $D$ with boundary $z(\theta)$, where $\theta\in[0,2\pi)$ is the Lagrangian parameter, and $z$ is counterclockwise parameterized, the contour dynamics equation for the quasi-geostrophic problem is
\begin{equation}
z_t(\theta, t) = \frac{C_\alpha}{2\pi}\int_0^{2\pi}\frac{z_\phi(\phi,t)d\phi}{|z(\phi,t) - z(\theta,t)|^\alpha};
\end{equation}

\noindent from now on, in order not to burden the notation, we will not indicate explicitly the dependence on $t$.

The simply-connected domain $D$ is a $V$-state rotating  with constant angular velocity $\Omega$, if and only if its boundary satisfies the following equation:
\begin{equation}
\label{e:QGcondition0}
\operatorname{Re}\bigg[\bigg(\Omega z(\theta) - \frac{C_\alpha}{2\pi i}\int_0^{2\pi}\frac{z_\phi(\phi)d\phi}{|z(\phi) - z(\theta)|^\alpha}\bigg)\overline{z_\theta(\theta)}\bigg] = 0.
\end{equation}

\noindent However, it is convenient to rewrite \eqref{e:QGcondition0} in the following equivalent form:
\begin{equation}
\label{e:QGcondition0a}
\operatorname{Re}\bigg[\bigg(\Omega z(\theta) - \frac{C_\alpha}{2\pi i}\int_0^{2\pi}\frac{(z_\phi(\phi) - z_\theta(\theta))d\phi}{|z(\phi) - z(\theta)|^\alpha}\bigg)\overline{z_\theta(\theta)}\bigg] = 0.
\end{equation}

\noindent We use a pseudo-spectral method to find $m$-fold $V$-states from \eqref{e:QGcondition0a}. We discretize $\theta\in[0,2\pi)$ in $N$ equally spaced nodes $\theta_i = 2\pi i/N$, $i = 0, 1, \ldots, N-1$. Observe that,
although \eqref{e:QGcondition0} and \eqref{e:QGcondition0a} are trivially equivalent, the addition of $z_\theta(\theta)$ in the numerator cancels the singularity in the denominator; indeed,
\begin{equation}
\label{e:limit0}
\lim_{\phi\to\theta}\frac{z_\phi(\phi) - z_\theta(\theta)}{|z(\phi) - z(\theta)|^\alpha} = 0.
\end{equation}

\noindent Therefore, bearing in mind \eqref{e:limit0}, we can evaluate numerically with spectral accuracy the integral in \eqref{e:QGcondition0a} at a node $\theta = \theta_i$, by means of the trapezoidal rule, provided that $N$ is large enough:
\begin{equation}
\frac{1}{2\pi}\int_0^{2\pi}\frac{(z_\phi(\phi) - z_\theta(\theta_i))d\phi}{|z(\phi) - z(\theta_i)|^\alpha}
\approx \frac{1}{N}\mathop{\sum_{j = 0}^{N-1}}_{j\not=i}\frac{z_\phi(\phi_j) - z_\theta(\theta_i)}{|z(\phi_j) - z(\theta_i)|^\alpha}\cdot
\end{equation}

\noindent Now, in order to obtain $m$-fold $V$-states, we approximate the boundary $z$ as
\begin{equation}
\label{e:z0cos}
z(\theta) = e^{i\theta}\left[1 + \sum_{k = 1}^M a_{k}\cos(m\,k\,\theta)\right], \quad \theta\in[0,2\pi),
\end{equation}

\noindent where the mean radius is $1$; and we are imposing that $z(-\theta) = \overline{ z(\theta)}$, i.e., we are looking for $V$-states symmetric with respect to the $x$-axis. For sampling purposes, $N$ has to be chosen such that $N \ge 2mM+1$; additionally, it is convenient to take $N$ a multiple of $m$, in order to be able to reduce the $N$-element discrete Fourier transforms to $N/m$-element discrete Fourier transforms. If we write $N = m2^r$, then $M = \lfloor (m2^r-1)/(2m)\rfloor = 2^{r-1}-1$.

We introduce \eqref{e:z0cos} into \eqref{e:QGcondition0a}, and approximate the error in that equation by an $M$-term sine expansion:
\begin{equation}
\label{e:V-State0conditions}
\operatorname{Re}\bigg[\bigg(\Omega z(\theta) - \frac{C_\alpha}{2\pi i}\int_0^{2\pi}\frac{(z_\phi(\phi) - z_\theta(\theta))d\phi}{|z(\phi) - z(\theta)|^\alpha} \bigg)\overline{z_\theta(\theta)}\bigg] = \sum_{k = 1}^M b_k\sin(m\,k\,\theta).
\end{equation}

\noindent This last expression can be represented in a very compact way as
\begin{equation}
\label{e:FaO}
\mathcal F_{\alpha,\Omega}(a_1, \ldots, a_M) = (b_1, \ldots, b_M),
\end{equation}

\noindent for a certain $\mathcal F_{\alpha,\Omega}\ : \ \mathbb{R}^{M}\to\mathbb{R}^{M}$.

Remark that, for any value of the parameters $\alpha$ and $\Omega$, we have trivially $\mathcal F_{\alpha,\Omega}(\mathbf 0) = \mathbf 0$, i.e., the unit circumference is a solution of the problem. Therefore, the obtention of a simply-connected $V$-state is reduced to finding numerically a nontrivial root $(a_1, \ldots, a_M)$ of \eqref{e:FaO}. To do so, we discretize the $(M\times M)$-dimensional Jacobian matrix $\mathcal J$ of $\mathcal F_{\alpha,\Omega}$ using first-order approximations. Fixed $|h|\ll1$ (we have chosen $h = 10^{-9}$), we have
\begin{equation}
\label{e:derivative0}
\frac{\partial}{\partial a_1} \mathcal F_{\alpha,\Omega}(a_1, \ldots, a_M) \approx \frac{\mathcal F_{\alpha,\Omega}(a_1 + h, \ldots, a_M) - \mathcal F_{\alpha,\Omega}(a_1, \ldots, a_M)}{h}\cdot
\end{equation}

\noindent Then, the sine expansion of \eqref{e:derivative0} gives us the first row of $\mathcal J$, and so on. Hence, if the $n$-th iteration is denoted by $(a_1, \ldots, a_M)^{(n)}$, then the $(n+1)$-th iteration is given by
\begin{equation}
(a_1, \ldots, a_M)^{(n+1)} = (a_1, \ldots, a_M)^{(n)} - \mathcal F_{\alpha,\Omega}\left((a_1, \ldots, a_M)^{(n)}\right)\cdot [\mathcal J^{(n)}]^{-1},
\end{equation}

\noindent where $[\mathcal J^{(n)}]^{-1}$ denotes the inverse of the Jacobian matrix at $(a_1, \ldots, a_M)^{(n)}$. This iteration converges in a small number of steps to a nontrivial root for a large variety of initial data $(a_1, \ldots, a_M)^{(0)}$. In particular, it is usually enough to perturb the unit circumference by assigning a small value to ${a_1}^{(0)}$, and leave the other coefficients equal to zero. Our stopping criterion is
\begin{equation}
\max\left|\sum_{k = 1}^M b_k\sin(m\,k\,\theta)\right| < tol,
\end{equation}

\noindent where $tol = 10^{-11}$. For the sake of coherence, we change eventually the sign of all the coefficients $\{a_k\}$, in order that, without loss of generality, $a_1>0$.

Before moving forward, let us mention that it is possible to work numerically with other parametrizations than \eqref{e:z0cos}, like for example
\begin{equation}
\label{e:z0conformal}
z(\theta) = e^{i\theta}\Big(1 + \sum_{k = 1}^Ma_{k}e^{-imk\theta}\Big), \quad \theta\in[0,2\pi),
\end{equation}

\noindent where we have used a conformal mapping. However, a caveat should be make here. Indeed, unlike in the vortex patch problem, given a $V$-state $(z, \Omega)$, and $\mu > 0$, $(\mu z, \Omega)$ is no longer a $V$-state, but, from \eqref{e:QGcondition0a}, $(\mu z, \mu^{-\alpha}\Omega)$ is. Therefore, since we bifurcate from the unit circumference at a certain angular velocity $\Omega = \Omega_m^\alpha$, we always obtain, by uniqueness, the same $V$-states up to a scaling that implies also a modification on $\Omega$, irrespectively of the chosen numerical representation of $z$. An equivalent observation can be done for the doubly-connected case, etc.

\subsubsection{Numerical experiments}

According to Theorem 1 in \cite{H-H}, for fixed $\alpha\in(0,1)$ and $m\ge2$, there is a family of $m$-fold $V$-states bifurcating from the unit circumference at the angular velocity
\begin{equation}
\label{e:Oma0}
\Omega_m^\alpha\triangleq \frac{\Gamma(1 - \alpha)}{2^{1-\alpha}\Gamma^2(1 - \frac{\alpha}{2})} \left(\frac{\Gamma(1 + \frac{\alpha}{2})}{\Gamma(2 - \frac{\alpha}{2})} - \frac{\Gamma(m + \frac{\alpha}{2})}{\Gamma(m + 1 - \frac{\alpha}{2})}\right).
\end{equation}

\noindent Given an $\Omega$ slightly smaller than $\Omega_m^\alpha$, it is straightforward to obtain the corresponding $V$-state with the technique described above. Then, we can use that $V$-state as a new initial datum to obtain another $V$-state with smaller $\Omega$, and so on. However, it seems impossible to obtain numerically $V$-states for $\Omega$ strictly larger than $\Omega_m^\alpha$. This means that the bifurcation is pitchfork and this fact  follows from a symmetry argument: if $(\Omega,z)$ is a solution of \eqref{e:QGcondition0} then $(\Omega,-z)$ is a solution too.

Bearing in mind \eqref{e:Oma0}, we are able to obtain $V$-states for an arbitrary large number of symmetries $m$. For instance, in Figure \ref{f:VState0holeAlpha0_5Omegas}, we have plotted simultaneously the $V$-states corresponding to $\alpha = 0.5$, $m = 10$, for $\Omega = 0.5592$ and $\Omega = 0.556, 0.552, \ldots, 0.528$. In all the numerical experiments in this section, we take $N = 256\times m$ nodes. Since, according to \eqref{e:Oma0}, $\Omega_{10}^{0.5} = 0.559238\ldots$, the $V$-state corresponding to $\Omega = 0.5592$, in black, is practically a circumference, as expected. On the other hand, the $V$-state corresponding to $\Omega = 0.528$ is plotted in red. Observe that we have been unable to obtain the $V$-state corresponding to, say, $\Omega = 0.527$; this makes us wonder whether the $V$-state in red might be close from developing some kind of singularity.
\begin{figure}[!htb]
\center
\includegraphics[width=0.5\textwidth, clip=true]{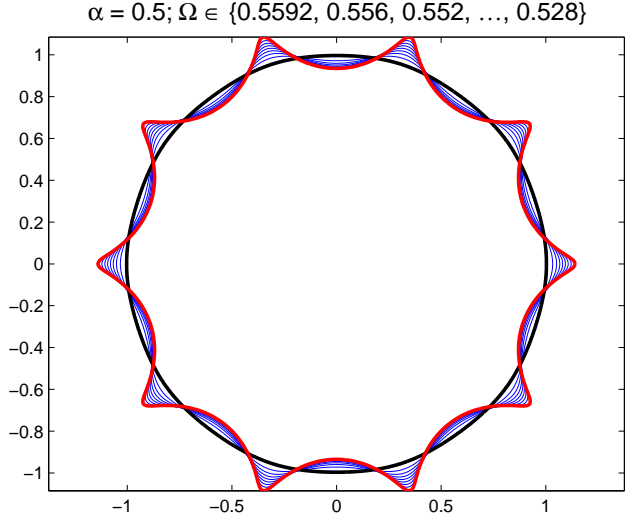}
\caption{\small{10-fold $V$-states corresponding to $\alpha = 0.5$, for different values of $\Omega$.}}
\label{f:VState0holeAlpha0_5Omegas}
\end{figure}

It is an established fact that simply-connected limiting $V$-states do exist for $\alpha = 0$, which corresponds to the vortex patch problem. These $V$-states are obtained after bifurcating from the circumference at $\Omega_m \triangleq (m - 1) / (2m)$, which corresponds to \eqref{e:Oma0} evaluated at $\alpha = 0$, and decreasing $\Omega$ as much as possible, until corner-shaped singularities appear. Furthermore, it has been proved in \cite{Over,WOZ} that the angle at the corners is always $\pi / 2$, irrespectively of the number $m$ of symmetries. Therefore, we are interested in understanding what happens when $\alpha > 0$.

In Figure \ref{f:limitingVstates}, we have plotted $V$-states corresponding to $m = 3, 4, 5, 6, 7$; for the vortex patch problem, and for $\alpha = 0.1$, $\alpha = 0.5$, and $\alpha = 0.9$ (i.e, a value of $\alpha$ rather close to zero, an intermediate value, and a value rather close to one). We have used the smallest possible (four-digit) values of $\Omega$, which are offered in Table \ref{t:LimitingOmegas}, in such a way that the experiments become numerically instable for smaller values of $\Omega$. In order to facilitate the comparison between different $m$, we have plotted $z(\theta) / \max_\theta(|z(\theta)|)$), rather than $z(\theta)$.

\begin{figure}[!htb]
\center
\includegraphics[width=0.5\textwidth, clip=true]{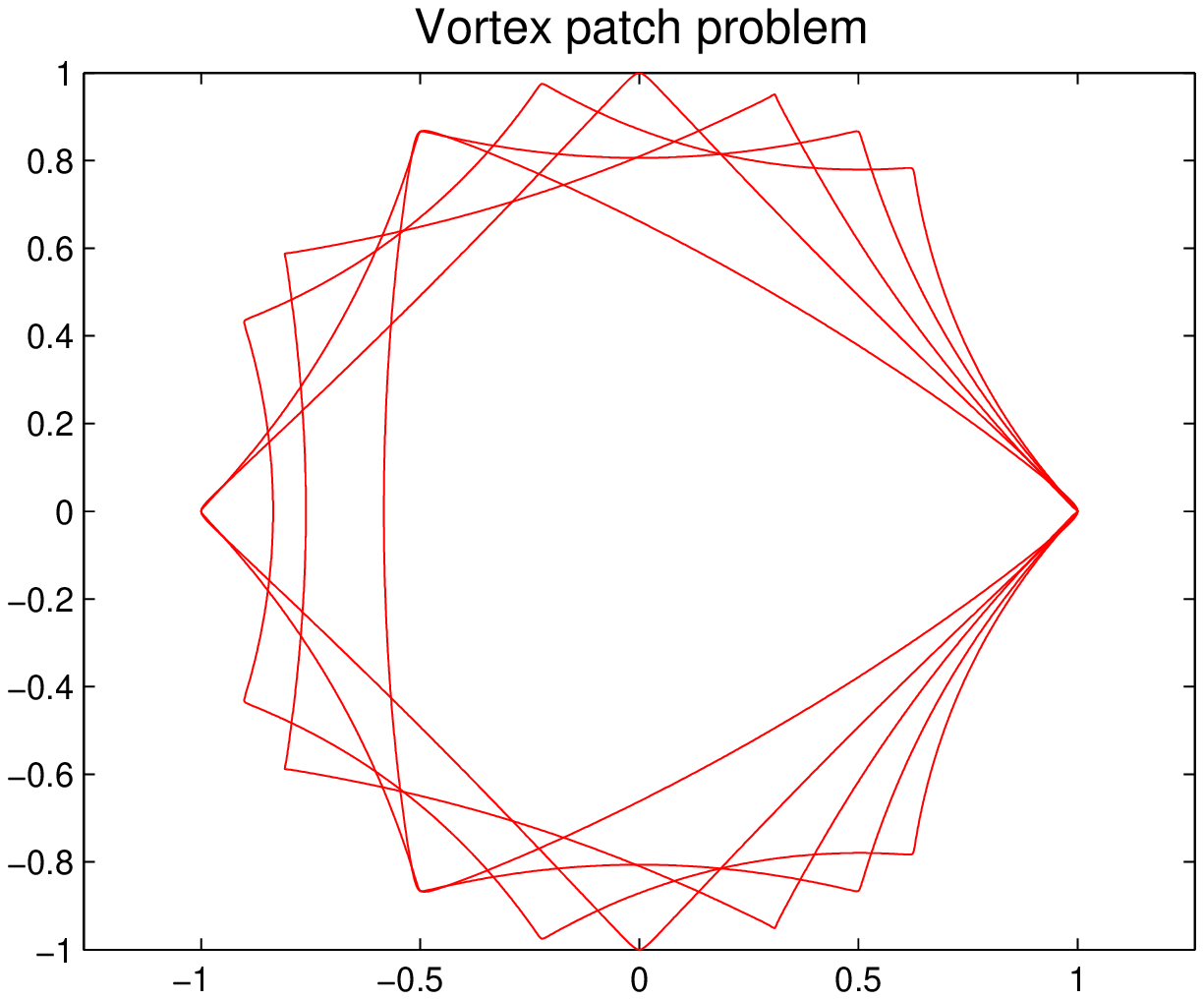}~
\includegraphics[width=0.5\textwidth, clip=true]{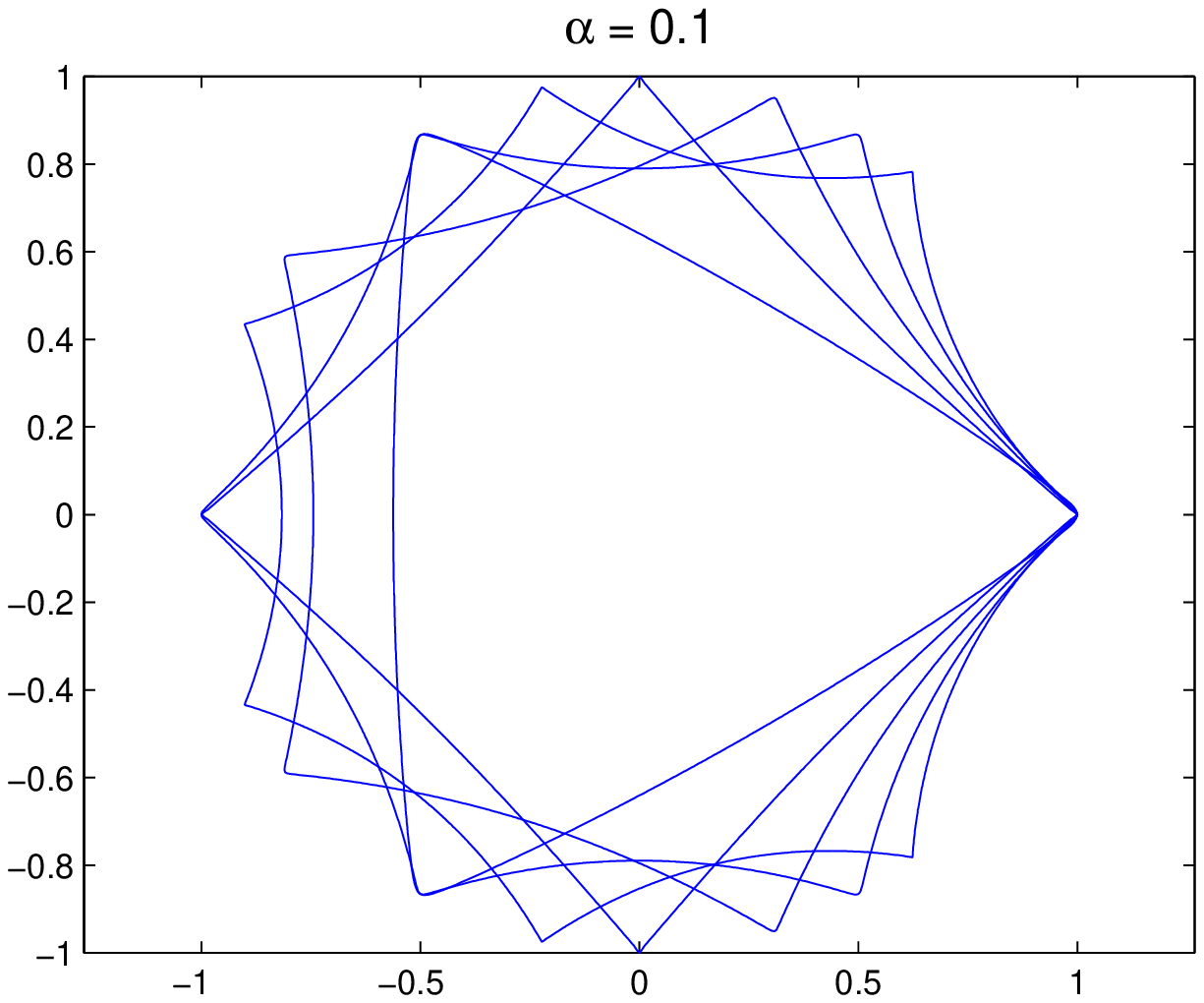}
\includegraphics[width=0.5\textwidth, clip=true]{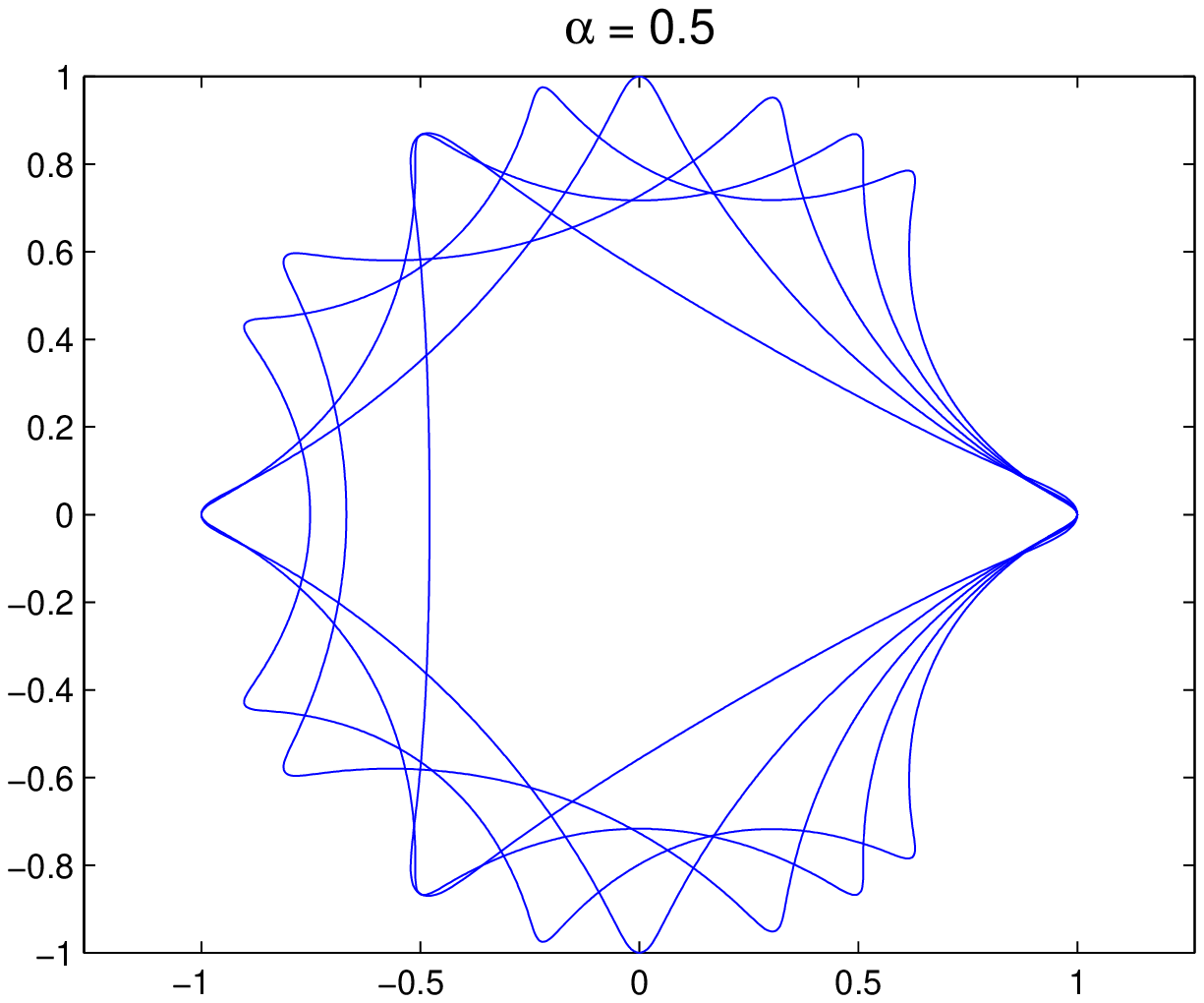}~
\includegraphics[width=0.5\textwidth, clip=true]{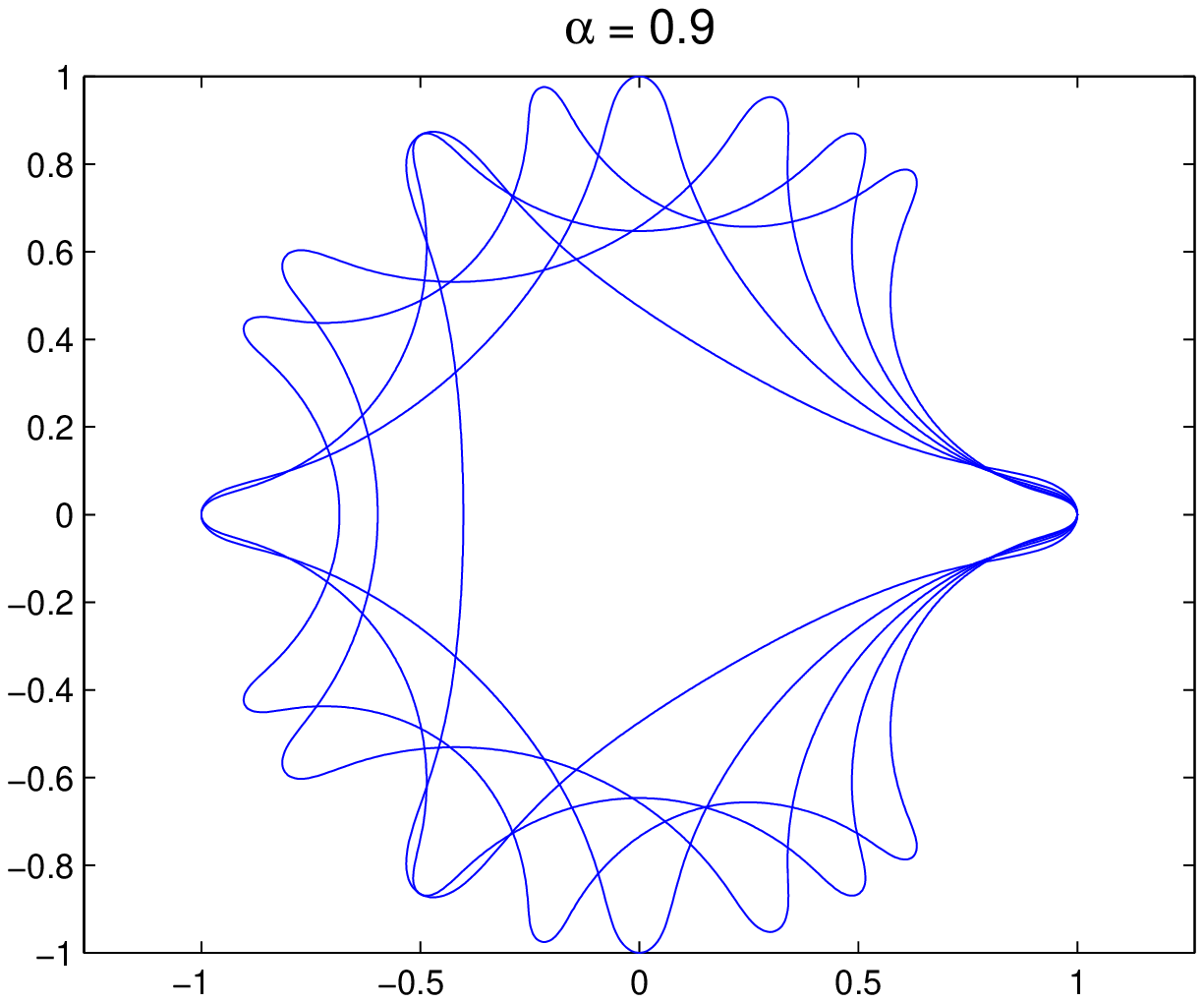}
\caption{\small{$V$-states corresponding to the vortex patch problem (i.e., $\alpha = 0$), and to $\alpha = 0.1$, $\alpha = 0.5$, and $\alpha = 0.9$; with $m = 3, 4, 5, 6, 7$. To facilitate the comparison between different $m$, we have plotted $z(\theta) / \max_\theta(|z(\theta)|)$, instead of $z(\theta)$. We have chosen the smallest possible (four-digit) values of $\Omega$, which are offered in Table \ref{t:LimitingOmegas}.
}}
\label{f:limitingVstates}
\end{figure}

\begin{table}[!htb]
$$
\begin{tabular}{|c|c|c|c|c|c|}
\hline
$\alpha$ / $m$ &3 & 4 & 5 & 6 & 7
    \\
\hline
0 & 0.3013 & 0.3540 & 0.3842 & 0.4040 & 0.4180
    \\
\hline
0.1 & 0.2965 & 0.3544 & 0.3884 & 0.4112 & 0.4275
    \\
\hline
0.5 & 0.2690 & 0.3498 & 0.4021 & 0.4399 & 0.4689
    \\
\hline
0.9 & 0.2191 & 0.3202 & 0.3918 & 0.4476 & 0.4933
    \\
\hline
\end{tabular}
$$
\caption{\small{Values of $\Omega$ for the $V$-states plotted in Figure \ref{f:limitingVstates}. The case $\alpha = 0$ corresponds to the vortex patch problem.}}
\label{t:LimitingOmegas}
\end{table}

Figure \ref{f:limitingVstates} confirms graphically that the angles developed by the five limiting $V$-states in the vortex patch problem are identical. On the other hand, when $\alpha = 0.1$, the $V$-states depicted are very similar to the limiting $V$-states in the vortex patch problem, whereas, for $\alpha = 0.5$, and especially for $\alpha = 0.9$, it is unclear whether any singularity has happened at all. In order to shed some light on this, we have plotted in Figure \ref{f:bifurcation} the respective bifurcation diagrams of $a_1$ in \eqref{e:z0cos} with respect to $\Omega$. In the vortex patch problem, we have a family of monotonic curves already shown in \cite{Saf}. When $\alpha = 0.1$, the curves are very similar to those in the vortex patch problem, but slightly bigger and more spaced. Then, as $\alpha$ grows, the curves become bigger and bigger, and more and more spaced. Furthermore, when $\alpha = 0.9$, the curves are partially superposed; for example, there are 2-fold and 3-fold $V$-states with the same $\Omega$. This phenomenon also happens in the last four curves, when $\alpha = 0.5$. However, the most striking fact from Figure \ref{f:bifurcation} is that all the fifteen curves, corresponding to $\alpha = 0.1$, $\alpha = 0.5$ and $\alpha = 0.9$, lose monotonicity at their left ends. In fact, especially for $\alpha = 0.9$, incipient hooks are clearly visible. This seems to suggest the presence of saddle-node bifurcation points (see for instance \cite{Kil}) at a certain $\Omega = \Omega_c$, which is indeed the case, as we will show in the following lines.

\begin{figure}[!htb]
\center
\includegraphics[width=0.5\textwidth, clip=true]{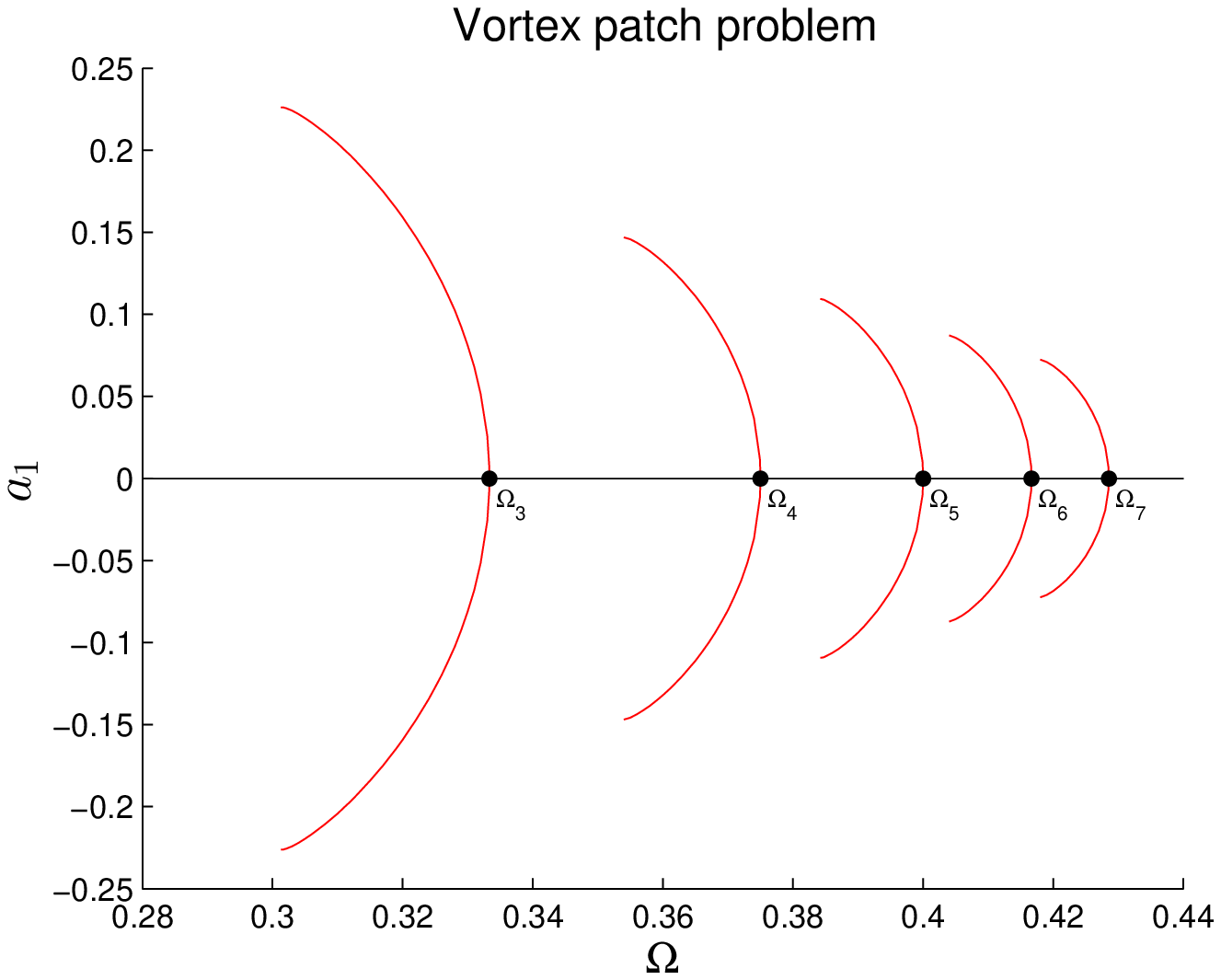}~
\includegraphics[width=0.5\textwidth, clip=true]{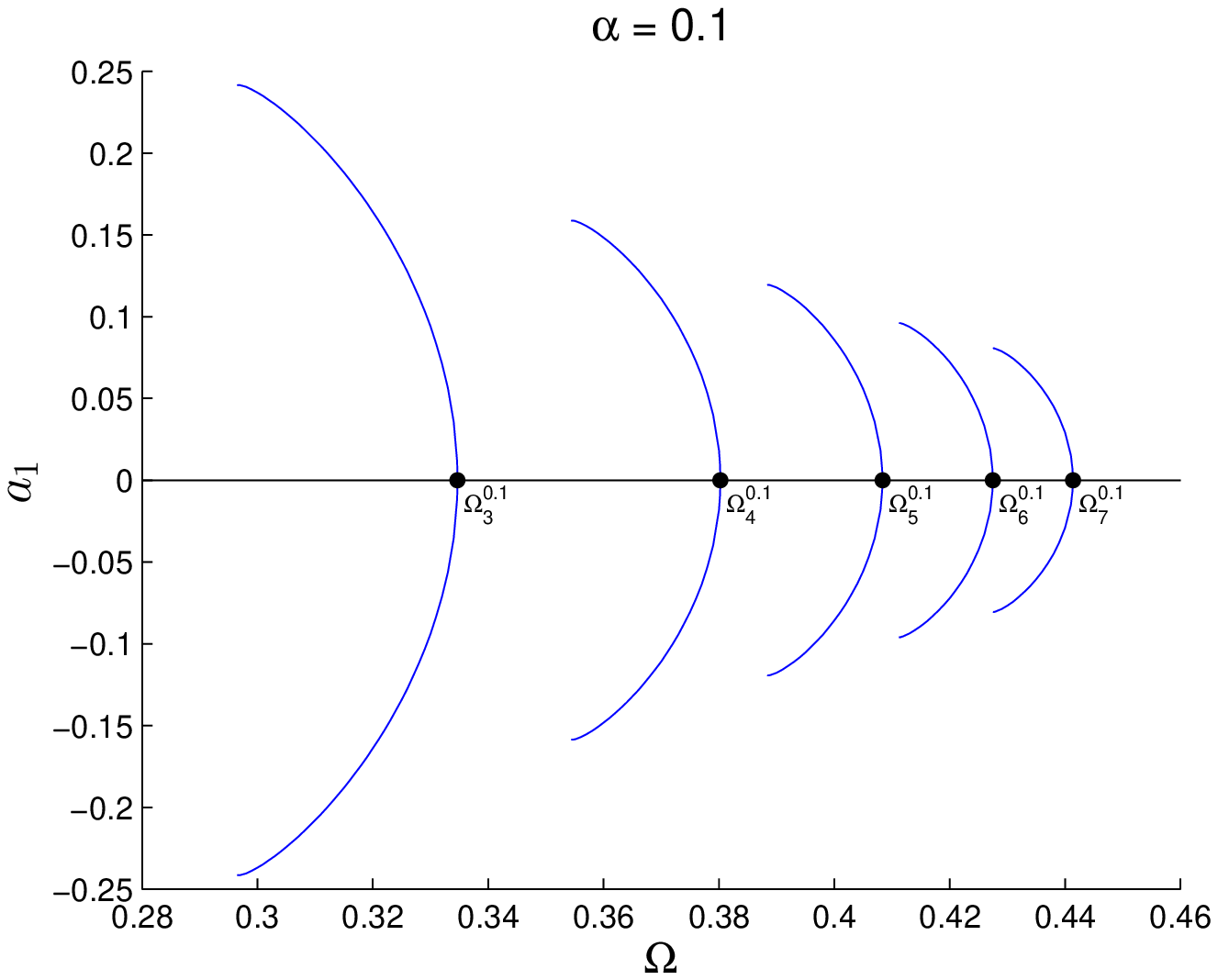}
\includegraphics[width=0.5\textwidth, clip=true]{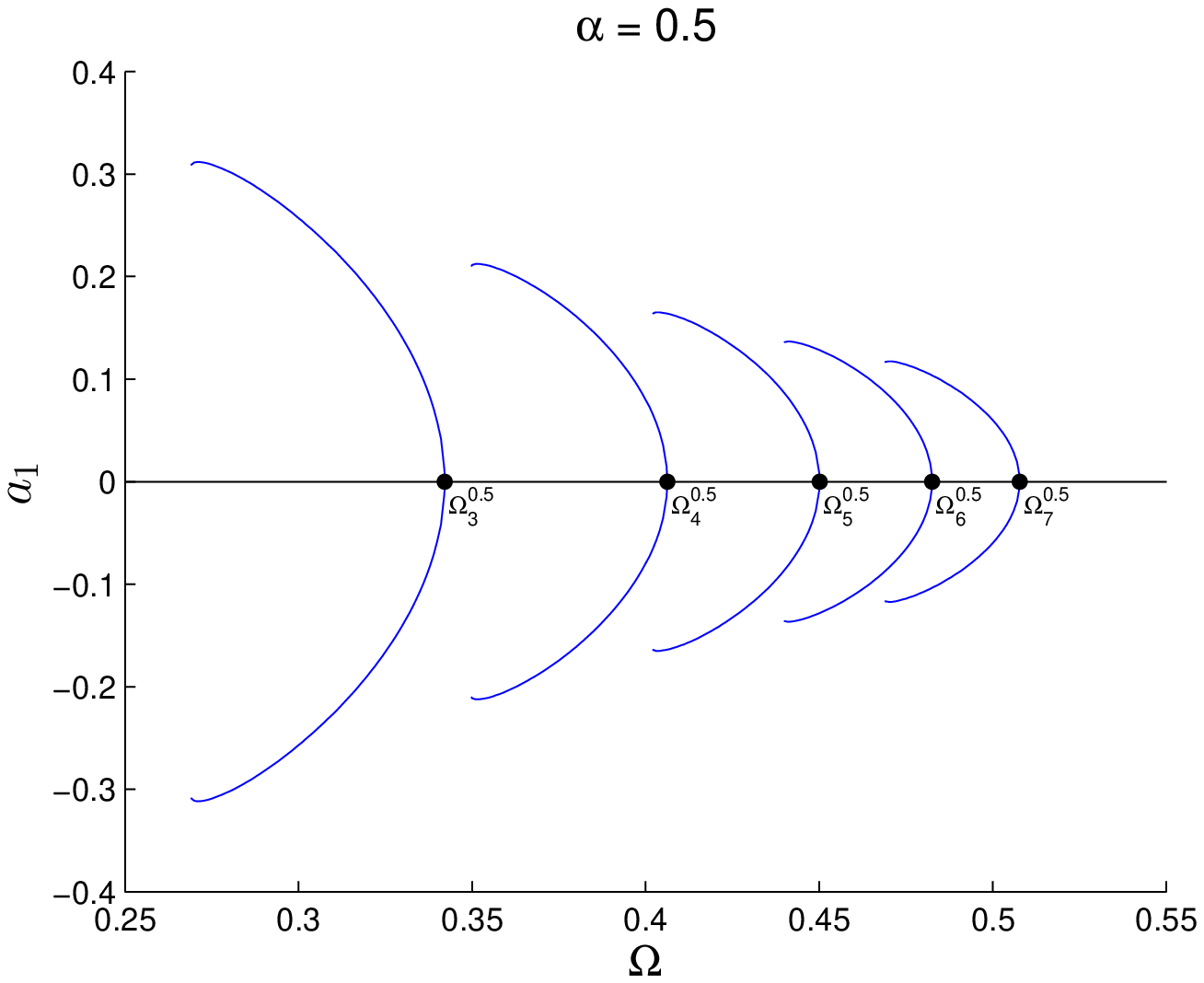}~
\includegraphics[width=0.5\textwidth, clip=true]{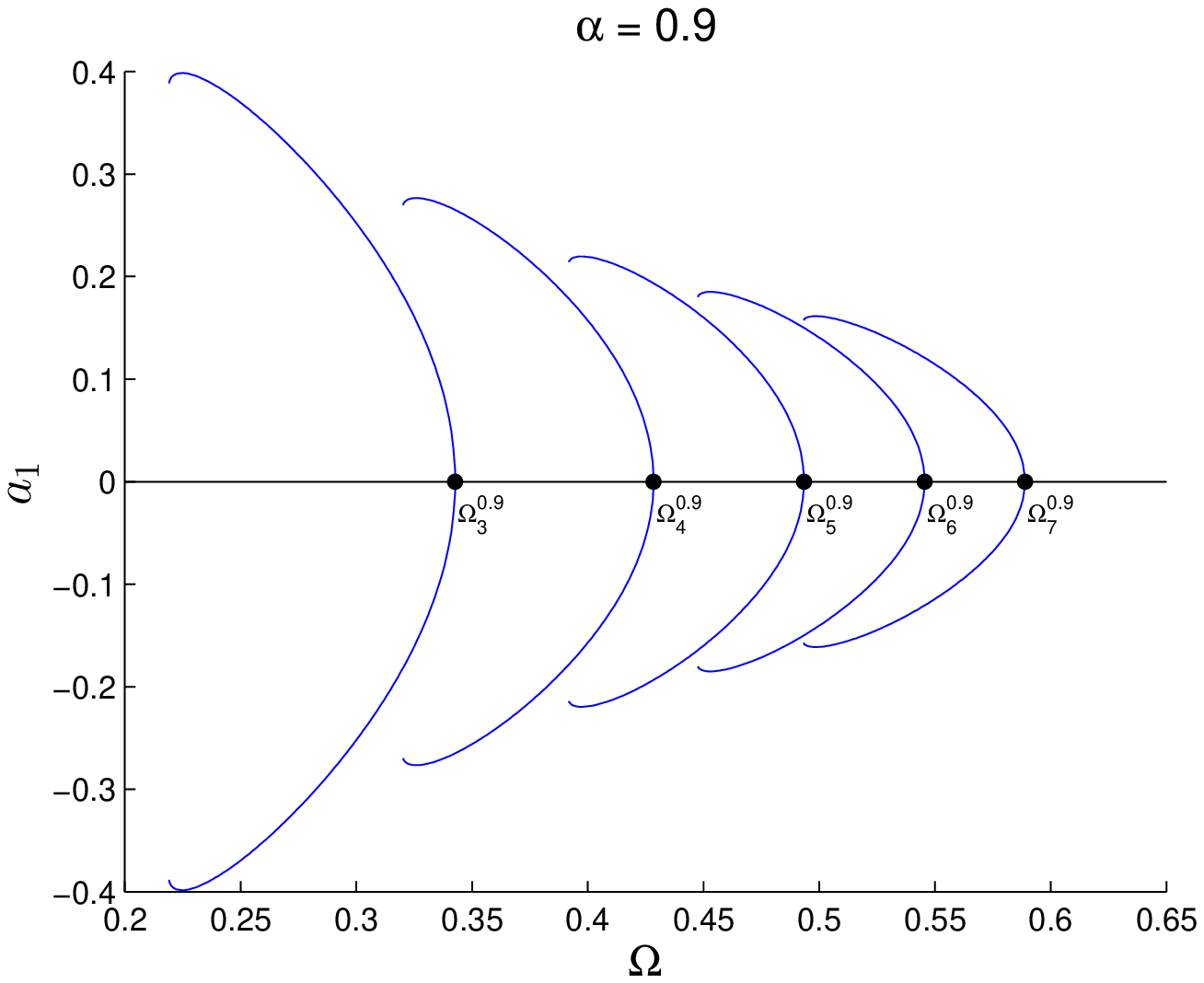}
\caption{\small{Bifurcations diagrams, for the vortex patch problem (i.e., $\alpha = 0$), and for $\alpha = 0.1$, $\alpha = 0.5$, and $\alpha = 0.9$; with $m = 3, 4, 5, 6, 7$.}}
\label{f:bifurcation}
\end{figure}

We work with the bifurcation curve corresponding to $\alpha = 0.9$ and $m = 3$ in Figure \ref{f:bifurcation}, because it has the most pronounced hook, but everything that follows is applicable to the other bifurcation curves as well. We need to estimate the corresponding $\Omega_c$ with enough accuracy. In our case, we have taken $\Omega_c = 0.21904$. Then, given $0<\epsilon\ll1$ (we have taken here $\epsilon = 10^{-4}$), we calculate the $V$-states corresponding to $\Omega^{(A)}\equiv\Omega_c + 4\epsilon$, $\Omega^{(B)}\equiv\Omega_c + 3\epsilon$, $\Omega^{(C)}\equiv\Omega_c + 2\epsilon$, $\Omega^{(D)}\equiv\Omega_c + \epsilon$, and $\Omega^{(E)}\equiv\Omega_c$, whose coefficients in \eqref{e:z0cos} are respectively denoted as $\{a_k^{(A)}\}$, $\{a_k^{(B)}\}$, $\{a_k^{(C)}\}$, $\{a_k^{(D)}\}$, and $\{a_k^{(E)}\}$. The main idea is to introduce a new parameter $\lambda$, instead of $\Omega$, in such a way that $\lambda = \lambda^{(A)}$ corresponds to $\Omega = \Omega^{(A)}$, and so on. We set $\lambda^{(A)} = 0$, $\lambda^{(B)} = \lambda^{(A)} + [(\Omega^{(B)} - \Omega^{(A)})^2 + (a_1^{(B)} - a_1^{(A)})^2]^{1/2}$, $\lambda^{(C)} = \lambda^{(B)} + [(\Omega^{(C)} - \Omega^{(B)})^2 + (a_1^{(C)} - a_1^{(B)})^2]^{1/2}$, $\lambda^{(D)} = \lambda^{(C)} + [(\Omega^{(D)} - \Omega^{(C)})^2 + (a_1^{(D)} - a_1^{(C)})^2]^{1/2}$, $\lambda^{(E)} = \lambda^{(D)} + [(\Omega^{(E)} - \Omega^{(D)})^2 + (a_1^{(E)} - a_1^{(D)})^2]^{1/2}$. Therefore, our problem is reduced to finding the $V$-state corresponding to some $\lambda$ slightly larger than $\lambda^{(E)}$; and a fairly good initial guess for that $V$-state can be obtained by means of a four-degree Lagrange interpolation polynomial. More precisely, let us define \begin{equation}
c^{(A)} = \frac{(\lambda - \lambda^{(B)})(\lambda - \lambda^{(C)})(\lambda - \lambda^{(D)})(\lambda - \lambda^{(E)})}{(\lambda^{(A)} - \lambda^{(B)})(\lambda^{(A)} - \lambda^{(C)})(\lambda^{(A)} - \lambda^{(D)})(\lambda^{(A)} - \lambda^{(E)})},
\end{equation}

\noindent and, in a similar way, $c^{(B)}$, $c^{(C)}$, $c^{(D)}$, and $c^{(E)}$. Then,
\begin{align}
\label{e:omega(lambda)}\Omega(\lambda) & = c^{(A)}\Omega^{(A)} + c^{(B)}\Omega^{(B)} + c^{(C)}\Omega^{(C)} + c^{(D)}\Omega^{(D)} + c^{(E)}\Omega^{(E)},
    \\
a_k(\lambda) & = c^{(A)}a_k^{(A)} + c^{(B)}a_k^{(B)} + c^{(C)}a_k^{(C)} + a_k\Omega^{(D)} + a_k\Omega^{(E)}, \quad k = 1,\ldots, M.
\end{align}

\noindent Remark that a couple of trials may be needed until a good choice of $\Omega^{(A)},\ldots$, and of $\lambda$ is found, i.e., values that enable us to continue the bifurcation curve, and not to come back to some already known $V$-state. In our case, we have chosen $\lambda$ equal to $\lambda^{(E)}$ plus the mean of the four previous increments of $\lambda$, i.e.,
\begin{align}
\lambda & = \lambda^{(E)} + \frac{(\lambda^{(E)} - \lambda^{(D)}) + (\lambda^{(D)} - \lambda^{(C)}) + (\lambda^{(C)} - \lambda^{(B)}) + (\lambda^{(B)} - \lambda^{(A)})}{4}
    \cr
    & = \frac{5\lambda^{(E)} - \lambda^{(A)}}{4}.
\end{align}

\noindent After applying this technique just once, we have successfully obtained a $V$-state corresponding to $\Omega = 0.219054\ldots$, i.e., a $V$-state beyond the critical point. It may be useful (and sometimes even convenient) to iterate several times the procedure, after updating $\lambda^{(A)} = \lambda^{(B)}$, $\lambda^{(B)} = \lambda^{(C)}$,$\lambda^{(C)} = \lambda^{(C)}$, $\lambda^{(C)} = \lambda^{(D)}$, and $\lambda^{(E)} = \lambda$. In fact, it can be even applied from the very beginning, to obtain all the bifurcation curves in \ref{f:bifurcation} in their integrity, with an important spare of computational time. In Figure \ref{f:bifurcation3updated}, we plot on the left-hand side the completed bifurcation curve until $\Omega = 0.236$; the piece of curve beyond the saddle-node bifurcation point, absent in Figure \ref{f:bifurcation}, is shown in thicker stroke.

It is possible to still continue the bifurcation curve, although the results are to be taken with prudence, because higher spatial resolution is needed. Further numerical experiments, which include the use of alternative parameterizations of $z$, would suggest the eventual formation of cusp-shaped singularities at the corners. They would also suggest the presence of additional saddle-node bifurcation points, in such a way that the bifurcation curve in \ref{f:bifurcation3updated} would show spiral-like structures at its ends. Nevertheless, since our results are still inconclusive, we postpone this challenging issue for the future.
\begin{figure}[!htb]
\center
\includegraphics[width=0.5\textwidth, clip=true]{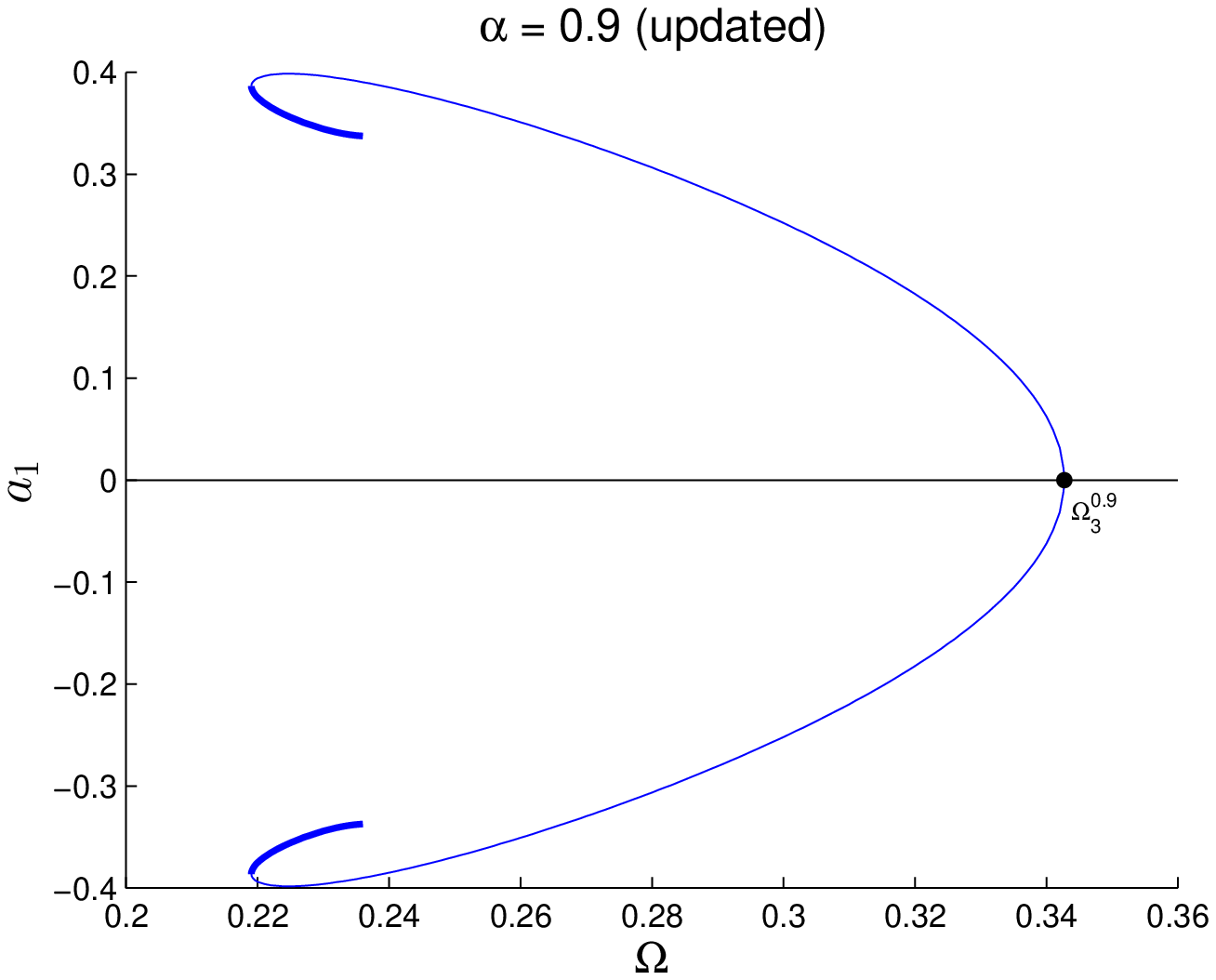}~
\includegraphics[width=0.5\textwidth, clip=true]{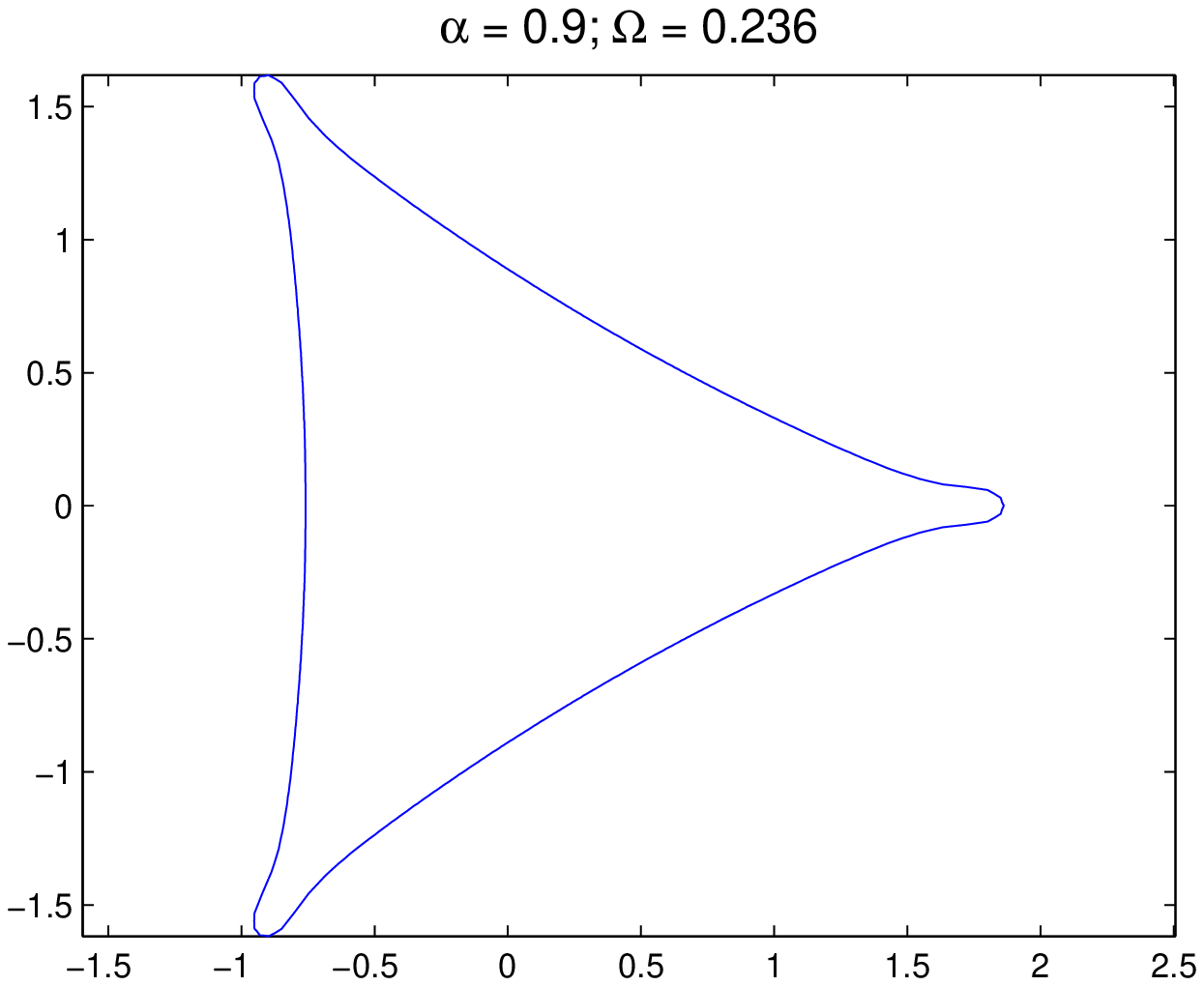}
\caption{\small{Left: Extended bifurcation curve, for $m = 3$, and $\alpha = 0.9$. Right: $V$-state beyond the saddle-node bifurcation point.}}
\label{f:bifurcation3updated}
\end{figure}

On the other hand, based on the previous pages and on additional numerical experiments that we have carried on, we conjecture the existence of saddle-node bifurcation points for all $m\ge3$ and for all $\alpha\in(0,1)$. However, as $\alpha$ decreases, smaller and smaller structures are expected at the ends of the bifurcation curves, until $\alpha = 0$, when they disappear. Indeed, in the vortex patch problem, as mentioned above, the bifurcation curves are always monotonic.

We cannot finish this section, without saying something about the case $m = 2$, which has a pretty different behavior and is interesting per se. In Figure \ref{f:VState0holen2Alpha0_01Omegas}, we have plotted 2-fold $V$-states for $\alpha = 0.01$ and different values of $\Omega$, starting from $\Omega = 0.2496$, which is close to $\Omega_2^{0.01} = 0.249667\ldots$, so the corresponding $V$-state, in black, is practically a unit circumference. Since $\alpha$ is small, we might expect to have a similar behavior to that in the vortex patch problem, where the $V$-states tend to degenerate to a segment as $\Omega$ decreases. However, although this is true for $\Omega$ close enough to $\Omega_2^{0.01}$, there is an instant when convexity is lost, and the $V$-states get a more and more pronounced $\infty$-shape, as $\Omega$ decreases. Remark that $\Omega$ cannot be smaller than a certain value, which corresponds approximately to $\Omega = 0.1057$, and whose corresponding $V$-state is plotted in red. Let us mention that the situation is very similar for greater $\alpha$, even for those close to 1.

\begin{figure}[!htb]
\center
\includegraphics[width=0.5\textwidth, clip=true]{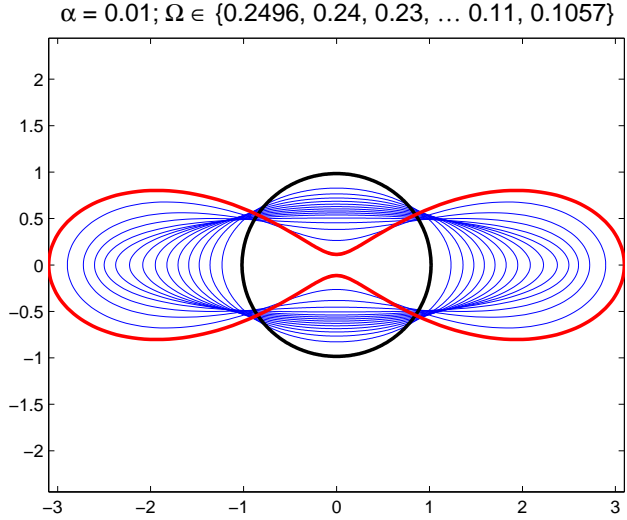}
\caption{\small{2-fold $V$-states corresponding to $\alpha = 0.2$, for different values of $\Omega$.}}
\label{f:VState0holen2Alpha0_01Omegas}
\end{figure}

In Figure \ref{f:VState0holen2Alpha0_01Omegas}, the $V$-state in red seems to have developed no singularity. Again, insight into what is happening is given by the bifurcation curve of $a_1$ in \eqref{e:z0cos} with respect to $\Omega$, for $\alpha = 0.01$, which is plotted in Figure \eqref{e:z0cos}. In that figure, we have also plotted the bifurcation curves for $\alpha = 0.1, 0.2, 0.3$, being the four curves very similar to each other.

\begin{figure}[!htb]
\center
\includegraphics[width=0.5\textwidth, clip=true]{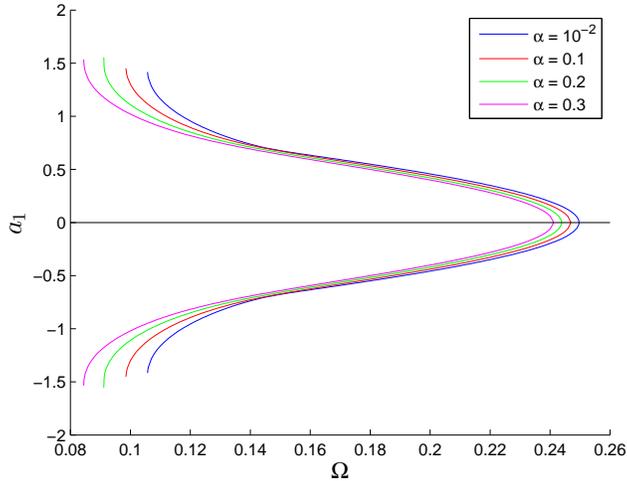}
\caption{\small{Bifurcation curves, for $m = 2$, and $\alpha = 0.01, 0.1, 0.2, 0.3$. We plot the value of the first coefficient $a_1$ in \eqref{e:z0cos} with respect to $\Omega$. The right-most curve corresponds to the smallest $\alpha$, and so on.}}
\label{f:bifurcation2}
\end{figure}

As with $m\ge3$, the bifurcation curves suggest the existence of saddle-node bifurcation points. To see whether this is indeed the case, we have used the previously described continuation method for $\alpha = 0.01$, taking $\Omega_c = 0.10567$, and $\epsilon = 10^{-4}$. Figure \ref{f:bifurcation2updated} confirms our suspicions. On the left-hand side, we plot the completed bifurcation curve until $\Omega = 0.1061$; the piece of curve beyond the saddle-node bifurcation point, absent in Figure \ref{f:bifurcation2}, is shown in thicker stroke. Remark that, unlike Figure \ref{f:bifurcation3updated}, a zoom is necessary in order for the hook to be appreciated. On the right-hand side, we plot the new $V$-state corresponding to a slightly larger $\Omega$, i.e., $\Omega = 0.1061407$ (and such that $\Omega = 0.1061408$ is unstable), with twice as many nodes, i.e, $N = 512\times2$. Apparently, a self-intersection has happened, although a powerful zoom shows that the distance between the two inner pieces of curve is approximately $5.9\times10^{-3}$; moreover, there are apparently enough nodes in that region, so it seems that we could decrease that distance even further, by increasing the eight decimal of $\Omega$, and so on. Even if we rather think that a self-intersection will eventually occur (see \cite{Cerr} and \cite{Luz} for similar phenomena in the vortex patch problem), further study is required here.

\begin{figure}[!htb]
\center
\includegraphics[width=0.5\textwidth, clip=true]{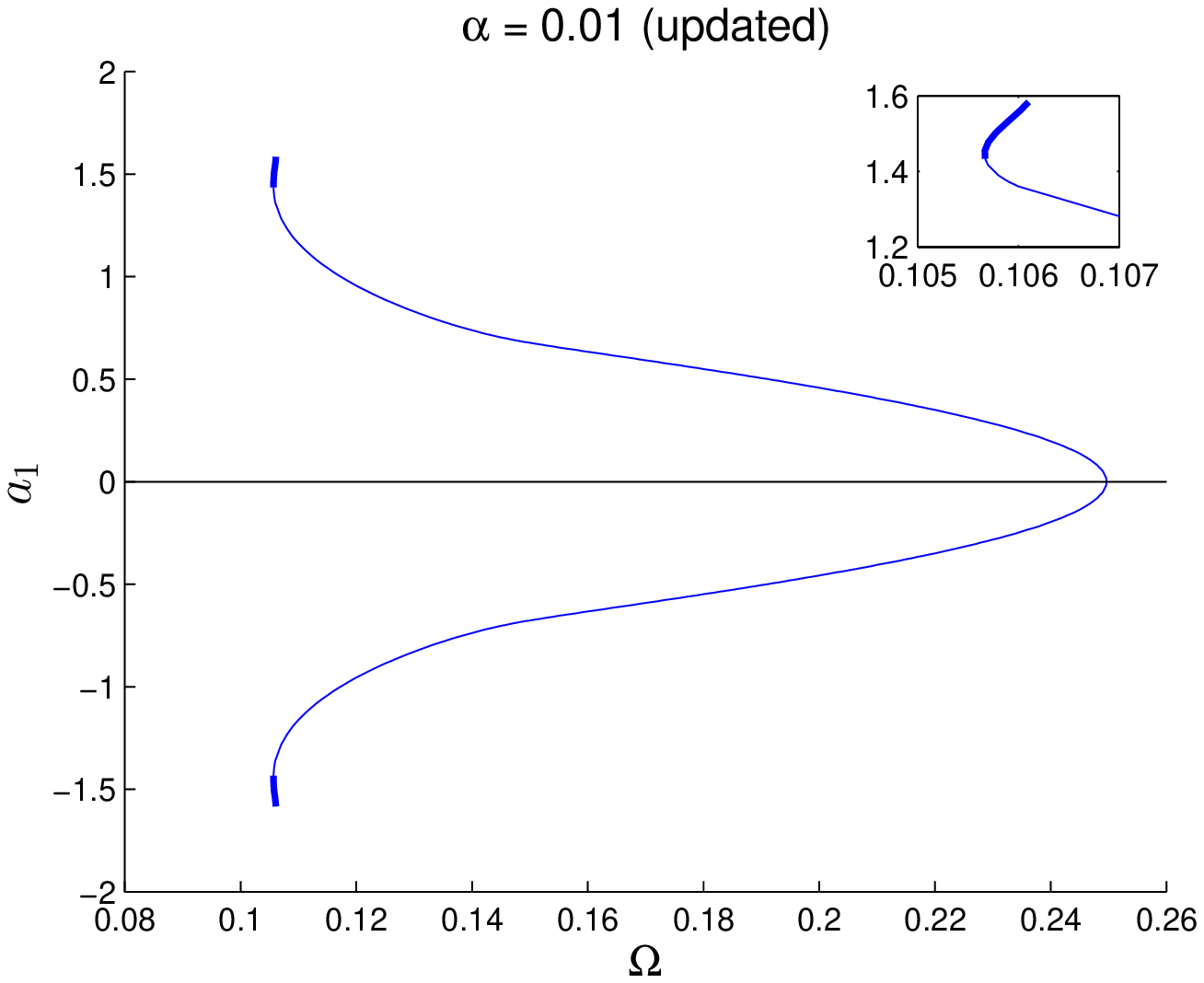}~
\includegraphics[width=0.5\textwidth, clip=true]{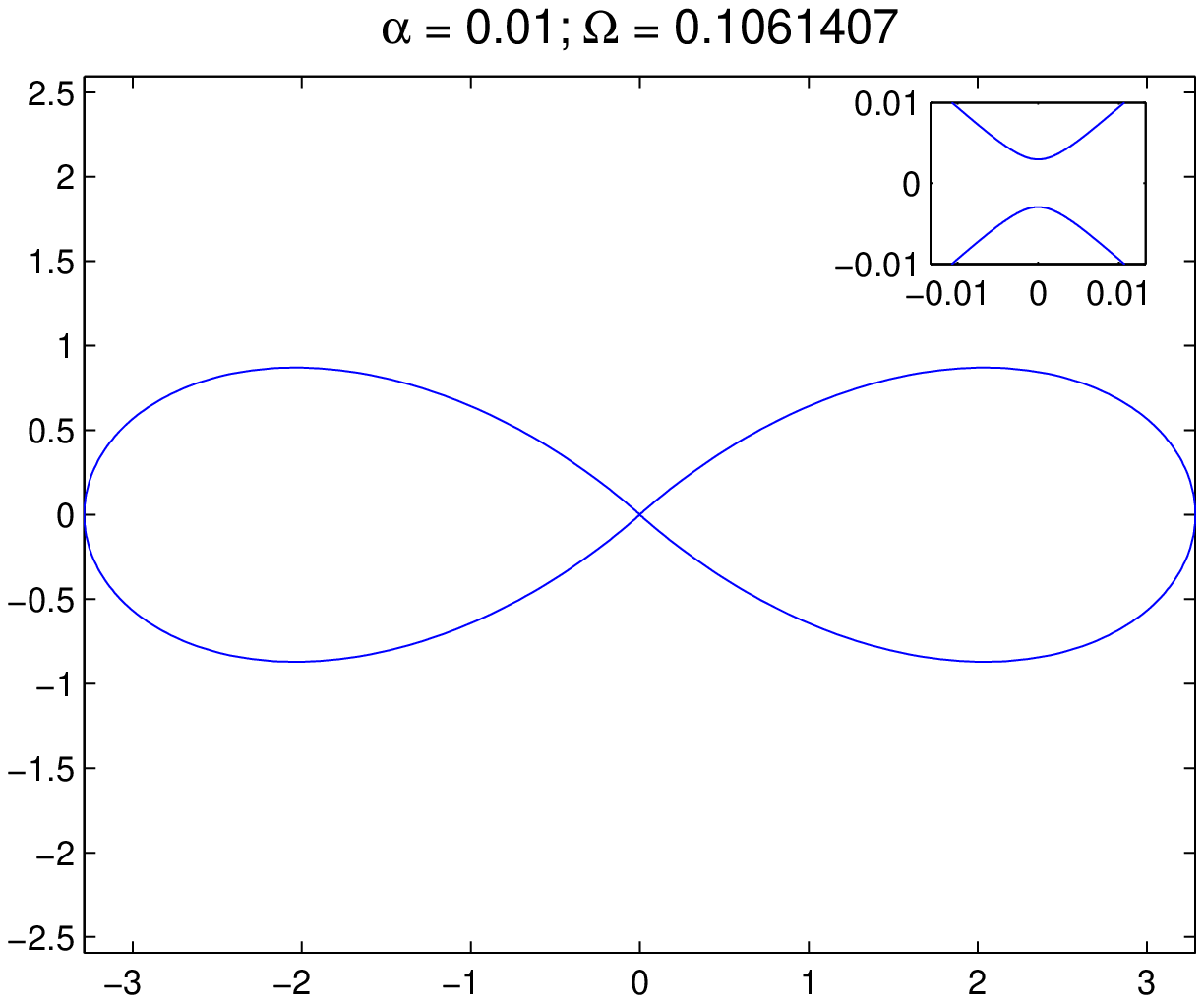}
\caption{\small{Left: Extended bifurcation curve, for $m = 2$, and $\alpha = 0.01$. Right: $V$-state beyond the saddle-node bifurcation point. A self-intersection has almost occurred.}}
\label{f:bifurcation2updated}
\end{figure}

\subsection{Doubly-connected $V$-states}

Given a doubly-connected domain $D$ with outer boundary $z_1(\theta)$ and inner boundary $z_2(\theta)$, where $\theta\in[0,2\pi)$ is the Lagrangian parameter, and $z_1$ and $z_2$ are counterclockwise parameterized, the contour dynamics equations for the quasi-geostrophic problem are
\begin{equation}
\begin{split}
z_{1,t}(\theta, t) & = \frac{C_\alpha}{2\pi}\int_0^{2\pi}\frac{z_{1,\phi}(\phi,t)d\phi}{|z_1(\phi,t) - z_1(\theta,t)|^\alpha} - \frac{C_\alpha}{2\pi}\int_0^{2\pi}\frac{z_{2,\phi}(\phi,t)d\phi}{|z_2(\phi,t) - z_1(\theta,t)|^\alpha},
    \\
z_{2,t}(\theta, t) & = \frac{C_\alpha}{2\pi}\int_0^{2\pi}\frac{z_{1,\phi}(\phi,t)d\phi}{|z_1(\phi,t) - z_2(\theta,t)|^\alpha} - \frac{C_\alpha}{2\pi}\int_0^{2\pi}\frac{z_{2,\phi}(\phi,t)d\phi}{|z_2(\phi,t) - z_2(\theta,t)|^\alpha};
\end{split}
\end{equation}

\noindent where, as in the simply-connected case, we will omit the explicit dependence on $t$.

The doubly-connected domain $D$ is a $V$-state if and only if its boundaries satisfy the following equations:
\begin{align}
\label{e:QGcondition1}
\operatorname{Re}\bigg[\bigg(\Omega z_1(\theta) - \frac{C_\alpha}{2\pi i}\int_0^{2\pi}\frac{z_{1,\phi}(\phi)d\phi}{|z_1(\phi) - z_1(\theta)|^\alpha} + \frac{C_\alpha}{2\pi i}\int_0^{2\pi}\frac{z_{2,\phi}(\phi)d\phi}{|z_2(\phi) - z_1(\theta)|^\alpha}\bigg)\overline{z_{1,\theta}(\theta)}\bigg] & = 0,
    \\
\label{e:QGcondition2}
\operatorname{Re}\bigg[\bigg(\Omega z_2(\theta) - \frac{C_\alpha}{2\pi i}\int_0^{2\pi}\frac{z_{1,\phi}(\phi)d\phi}{|z_1(\phi) - z_2(\theta)|^\alpha} + \frac{C_\alpha}{2\pi i}\int_0^{2\pi}\frac{z_{2,\phi}(\phi)d\phi}{|z_2(\phi) - z_2(\theta)|^\alpha}\bigg)\overline{z_{2,\theta}(\theta)}\bigg] & = 0.
\end{align}

\noindent However, as we did in \eqref{e:QGcondition0a}, it is convenient to rewrite them in the following equivalent form:
\begin{align}
\label{e:QGcondition1a}
\operatorname{Re}\bigg[\bigg(\Omega z_1(\theta) & - \frac{C_\alpha}{2\pi i}\int_0^{2\pi}\frac{(z_{1,\phi}(\phi) - z_{1,\theta}(\theta))d\phi}{|z_1(\phi) - z_1(\theta)|^\alpha}
    \cr
& + \frac{C_\alpha}{2\pi i}\int_0^{2\pi}\frac{z_{2,\phi}(\phi)d\phi}{|z_2(\phi) - z_1(\theta)|^\alpha}\bigg)\overline{z_{1,\theta}(\theta)}\bigg] = 0,
    \\
\label{e:QGcondition2a}
\operatorname{Re}\bigg[\bigg(\Omega z_2(\theta) & - \frac{C_\alpha}{2\pi i}\int_0^{2\pi}\frac{z_{1,\phi}(\phi)d\phi}{|z_1(\phi) - z_2(\theta)|^\alpha}
    \cr
& + \frac{C_\alpha}{2\pi i}\int_0^{2\pi}\frac{(z_{2,\phi}(\phi) - z_{2,\theta}(\theta))d\phi}{|z_2(\phi) - z_2(\theta)|^\alpha}\bigg)\overline{z_{2,\theta}(\theta)}\bigg] = 0.
\end{align}

\noindent We use again a pseudo-spectral method to find $V$-states. We discretize $\theta\in[0,2\pi)$ in $N$ equally spaced nodes $\theta_i = 2\pi i/N$, $i = 0, 1, \ldots, N-1$. Since $z_1$ and $z_2$ never intersect, the second integral in \eqref{e:QGcondition1a} and the first integral in \eqref{e:QGcondition2a} can be evaluated numerically with spectral accuracy at a node $\theta = \theta_i$ by means of the trapezoidal rule; e.g.,
\begin{equation}
\frac{1}{2\pi}\int_0^{2\pi}\frac{z_{2,\phi}(\phi)d\phi}{|z_2(\phi) - z_1(\theta_i)|^\alpha} \approx \frac{1}{N}\sum_{j = 0}^{N-1}\frac{z_{2,\phi}(\phi_j)}{|z_2(\phi_j) - z_1(\theta_i)|^\alpha}.
\end{equation}

\noindent Likewise, the first integral in \eqref{e:QGcondition1a} and the second integral in \eqref{e:QGcondition2a} can also be evaluated with spectral accuracy, reasoning as in \eqref{e:limit0}; e.g.,
\begin{equation}
\frac{1}{2\pi}\int_0^{2\pi}\frac{(z_{1,\phi}(\phi) - z_{1,\theta}(\theta_i))d\phi}{|z_1(\phi) - z_1(\theta_i)|^\alpha}
\approx \frac{1}{N}\mathop{\sum_{j = 0}^{N-1}}_{j\not=i}\frac{z_{1,\phi}(\phi_j) - z_{1,\theta}(\theta_i)}{|z_1(\phi_j) - z_1(\theta_i)|^\alpha}.
\end{equation}

\noindent Therefore, for $N$ large enough, we can evaluate with spectral accuracy all the terms in \eqref{e:QGcondition1a}-\eqref{e:QGcondition2a} at $\theta = \theta_i$.

In order to obtain doubly connected $m$-fold $V$-states, we approximate $z_1$ and $z_2$ as in \eqref{e:z0cos}:
\begin{equation}
\label{e:z1z2cos}
z_1(\theta) = e^{i\theta}\left[1 + \sum_{k = 1}^M a_{1,k}\cos(m\,k\,\theta)\right], \qquad
z_2(\theta) = e^{i\theta}\left[b + \sum_{k = 1}^M a_{2,k}\cos(m\,k\,\theta)\right],
\end{equation}

\noindent where $\theta\in[0,2\pi)$, the mean outer and inner radii are respectively $1$ and $b$; and we are imposing that $z_1(-\theta) = \overline{ z_1(\theta)}$ and $z_2(-\theta) = \overline{z_2(\theta)}$, i.e., we are looking for $V$-states symmetric with respect to the $x$-axis. Again, if we choose $N$ of the form $N = m2^r$, then $M = \lfloor (m2^r-1)/(2m)\rfloor = 2^{r-1}-1$.

We introduce \eqref{e:z1z2cos} into \eqref{e:QGcondition1a}-\eqref{e:QGcondition2a}, and approximate the error in those equations by an $M$-term sine expansion:
\begin{equation}
\label{e:V-Stateconditions}
\begin{split}
\operatorname{Re}\bigg[\bigg(\Omega z_1(\theta) & - \frac{C_\alpha}{2\pi i}\int_0^{2\pi}\frac{(z_{1,\phi}(\phi) - z_{1,\theta}(\theta))d\phi}{|z_1(\phi) - z_1(\theta)|^\alpha}
    \cr
& + \frac{C_\alpha}{2\pi i}\int_0^{2\pi}\frac{z_{2,\phi}(\phi)d\phi}{|z_2(\phi) - z_1(\theta)|^\alpha}\bigg)\overline{z_{1,\theta}(\theta)}\bigg] = \sum_{k = 1}^M b_{1,k}\sin(m\,k\,\theta),
    \\
\operatorname{Re}\bigg[\bigg(\Omega z_2(\theta) & - \frac{C_\alpha}{2\pi i}\int_0^{2\pi}\frac{z_{1,\phi}(\phi)d\phi}{|z_1(\phi) - z_2(\theta)|^\alpha}
    \cr
& + \frac{C_\alpha}{2\pi i}\int_0^{2\pi}\frac{(z_{2,\phi}(\phi) - z_{2,\theta}(\theta))d\phi}{|z_2(\phi) - z_2(\theta)|^\alpha}\bigg)\overline{z_{2,\theta}(\theta)}\bigg] = \sum_{k = 1}^M b_{2,k}\sin(m\,k\,\theta).
\end{split}
\end{equation}

\noindent As in \eqref{e:FaO}, this last system of equations can be represented in a very compact way as
\begin{equation}
\label{e:FbaO}
\mathcal F_{b,\alpha,\Omega}(a_{1,1}, \ldots, a_{1,M}\,,\, a_{2,1}, \ldots, a_{2,M}) = (b_{1,1}, \ldots, b_{1,M}\,,\, b_{2,1}, \ldots, b_{2,M}),
\end{equation}

\noindent for a certain $\mathcal F_{\alpha,\Omega}\ : \ \mathbb{R}^{2M}\to\mathbb{R}^{2M}$.

Remark that, for any value of the parameters $b$, $\alpha$ and $\Omega$, we have trivially $\mathcal F_{b,\alpha,\Omega}(\mathbf 0) = \mathbf 0$, i.e., any circular annulus is a solution of the problem. Therefore, the obtention of a doubly-connected $V$-state is reduced to finding numerically $\{a_{1,k}\}$ and $\{a_{2,k}\}$, such that $(a_{1,1}, \ldots, a_{1,M}\,,\, a_{2,1}, \ldots, a_{2,M})$ is a nontrivial root of \eqref{e:FbaO}. To do so, we discretize the $(2M\times2M)$-dimensional Jacobian matrix $\mathcal J$ of $\mathcal F_{b,\alpha,\Omega}$ as in \eqref{e:derivative0}, taking $h = 10^{-9}$:
\begin{equation}
\label{e:derivative}
\begin{split}
\frac{\partial}{\partial a_{1,1}} & \mathcal F_{b,\alpha,\Omega}(a_{1,1}, \ldots, a_{1,M}\,,\, a_{2,1}, \ldots, a_{2,M})
    \cr
& \approx \frac{\mathcal F_{b,\alpha,\Omega}(a_{1,1} + h, a_{1,2}, \ldots, a_{1,M}\,,\, a_{2,1}, \ldots, a_{2,M}) - \mathcal F_{b,\alpha,\Omega}(a_{1,1}, \ldots, a_{1,M}\,,\, a_{2,1}, \ldots, a_{2,M})}{h},
\end{split}
\end{equation}

\noindent Then, the sine expansion of \eqref{e:derivative} gives us the first row of $\mathcal J$, and so on. Hence, if the $n$-th iteration is denoted by $(a_{1,1}, \ldots, a_{1,M}\,,\, a_{2,1}, \ldots, a_{2,M})^{(n)}$, then the $(n+1)$-th iteration is given by
\begin{equation}
\begin{split}
(a_{1,1}, & \ldots, a_{1,M}\,,\, a_{2,1}, \ldots, a_{2,M})^{(n+1)}
    \cr
& = (a_{1,1}, \ldots, a_{1,M}\,,\, a_{2,1}, \ldots, a_{2,M})^{(n)} - \mathcal F_{b,\alpha,\Omega}\left((a_{1,1}, \ldots, a_{1,M}\,,\, a_{2,1}, \ldots, a_{2,M})^{(n)}\right)\cdot [\mathcal J^{(n)}]^{-1},
\end{split}
\end{equation}

\noindent where $[\mathcal J^{(n)}]^{-1}$ denotes the inverse of the Jacobian matrix at $(a_{1,1}, \ldots, a_{1,M}\,,\, a_{2,1}, \ldots, a_{2,M})^{(n)}$. To make this iteration converge, it is usually enough to perturb the annulus by assigning a small value to ${a_{1,1}}^{(0)}$ or ${a_{2,1}}^{(0)}$, and leave the other coefficients equal to zero. Our stopping criterion is
\begin{equation}
\max\left|\sum_{k = 1}^M b_{1,k}\sin(m\,k\,\theta)\right| < tol \quad \wedge \quad \max\left|\sum_{k = 1}^M b_{2,k}\sin(m\,k\,\theta)\right| < tol,
\end{equation}

\noindent where $tol = 10^{-11}$. As in the vortex patch problem, $a_{1,1}\cdot a_{2,1}<0$, so, for the sake of coherence, we change eventually the sign of all the coefficients $\{a_{1,k}\}$ and $\{a_{2,k}\}$, in order that, without loss of generality, $a_{1,1}>0$ and $a_{2,1}<0$.

\subsubsection{Numerical experiments}

As we have seen, the procedure to find doubly-connected $m$-fold $V$-states is very similar in the vortex patch and in the quasi-geostrophic problems. However, as evidenced in the simply-connected case, the numerical study of the $V$-states for the quasi-geostrophic problem reveals itself as a much richer task. Indeed, unlike in the vortex patch problem, where we had just two parameters $b$ and $\Omega$, we have now a third parameter $\alpha$. Furthermore, since \eqref{e:QGcondition1a} and \eqref{e:QGcondition2a} are not homogeneous, choosing the mean outer radius not to be equal to one, as we are doing, would introduce a fourth parameter. Therefore, we will limit ourselves here to exposing a few relevant facts.

Theorem \ref{main} states that, for any $\alpha\in(0,1)$, the inner mean radius $b$ must be greater than $b_0$, which is the unique solution of the equation
\begin{equation}
\label{e:b0}
b^2\Lambda_1(b) - \Lambda_1(1) + \frac{1}{2} = 0, \qquad \mbox{with } \Lambda_1(1) = \frac{\Gamma(1 - \alpha)}{2^{1-\alpha}\Gamma^2(1-\frac{\alpha}{2})}\frac{\Gamma(1 + \frac{\alpha}{2})}{\Gamma(2 - \frac{\alpha}{2})}.
\end{equation}

\noindent This is an important difference with the vortex patch problem, where no lower bound for $b$ exists.

In order to obtain $b_0$ (and other relevant quantities), we need to compute $\Lambda_n(b)$ accurately:
\begin{align}
\Lambda_n(b) & = \frac{\Gamma(\frac{\alpha}{2})}{2^{1-\alpha}\Gamma(1 - \frac{\alpha}{2})}\frac{(\frac{\alpha}{2})_n}{n!}b^{n-1}F\left(\frac{\alpha}{2}, n + \frac{\alpha}{2}, n + 1, b^2\right)
    \cr
& = \frac{b^{n-1}\Gamma(\frac{\alpha}{2} + n)}{2^{1-\alpha}\Gamma(1 - \frac{\alpha}{2})n!}F\left(\frac{\alpha}{2}, n + \frac{\alpha}{2}, n + 1, b^2\right).
\end{align}

\noindent The hypergeometric function $F$ is commonly implemented in most scientific packages; for instance, in \textsc{Matlab\copyright}, $F(a, b, c, z)$ can be evaluated by means of the command \verb"hypergeom([a, b], c, z)", so we can find the only value $b_0$ satisfying \eqref{e:b0} efficiently and with the greatest possible accuracy, by means a simple bisection technique. In Figure \ref{f:alphab0}, we have plotted $b_0$ against 200 different values of $\alpha$, i.e., $\alpha = 10^{-4}, 10^{-3}$, and $\alpha = 0.005, 0.01, \ldots, 0.995$. Observe that $b_0$ tends very quickly to $1$; for example, $b_0(0.76) > 0.99$. On the other hand, $b_0(0) = 0$, which is coherent with the vortex patch problem, where there is no lower bound for $b$.
\begin{figure}[!htb]
\center
\includegraphics[width=0.5\textwidth, clip=true]{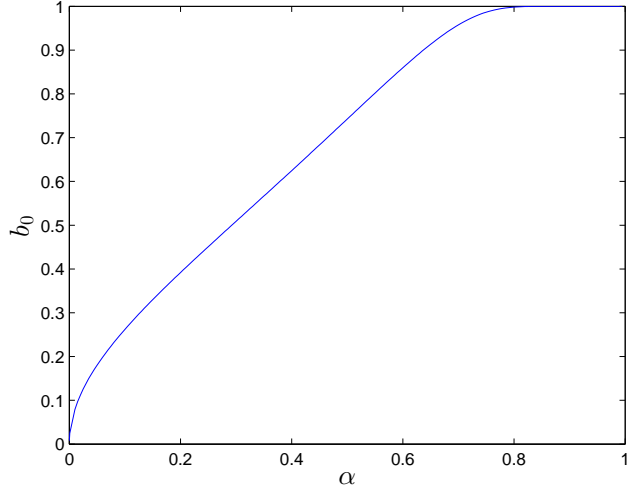}
\caption{\small{Unique solution of \eqref{e:b0}, $b_0$, in function of $\alpha$.}}
\label{f:alphab0}
\end{figure}

However, we have discovered in our numerical experiments that \eqref{e:b0} is not sharp. Indeed, we have been able to find $V$-states with $b$ much smaller than $b_0$. In all the numerical experiments in this section, we take $N = 64\times m$ nodes. In Figure \ref{f:VState1holeAlpha0_9b0_2Omegas}, we have plotted $V$-states corresponding to $\alpha = 0.9$ and $b =0.2$, where $b_0(0.9) - 1 = \mathcal O(10^{-9})$. On the left-hand side, we have started to bifurcate from $\Omega_4^+(0.2) = 0.4077\ldots$. Observe that $z_2$ is practically a circumference, for all $\Omega$. Moreover, the $V$-state corresponding to $\Omega = 0.4076$, in black, is practically a circular annulus, whereas the outer boundary of the $V$-state corresponding to $\Omega = 0.296$, in red, has a marked star shape. For $\Omega$ slightly smaller than $\Omega = 0.296$, the numerical experiments become instable. On the right-hand side, we have started to bifurcate from $\Omega_4^-(0.2) = -1.3055\ldots$. The most remarkable fact is that $\Omega$ is always negative, i.e, the $V$-states rotate clockwise, which is an important difference with respect to the vortex patch problem. On the other hand, $z_1$ is practically a circumference, for all $\Omega$. The $V$-state corresponding to $\Omega = -1.305$, in black, is practically a circular annulus, whereas the inner boundary of the $V$-state corresponding to $\Omega = -0.849$, in red, has a marked star shape. For $\Omega$ slightly larger than $\Omega = -0.849$, the numerical experiments become instable.

\begin{figure}[!htb]
\center
\includegraphics[width=0.5\textwidth, clip=true]{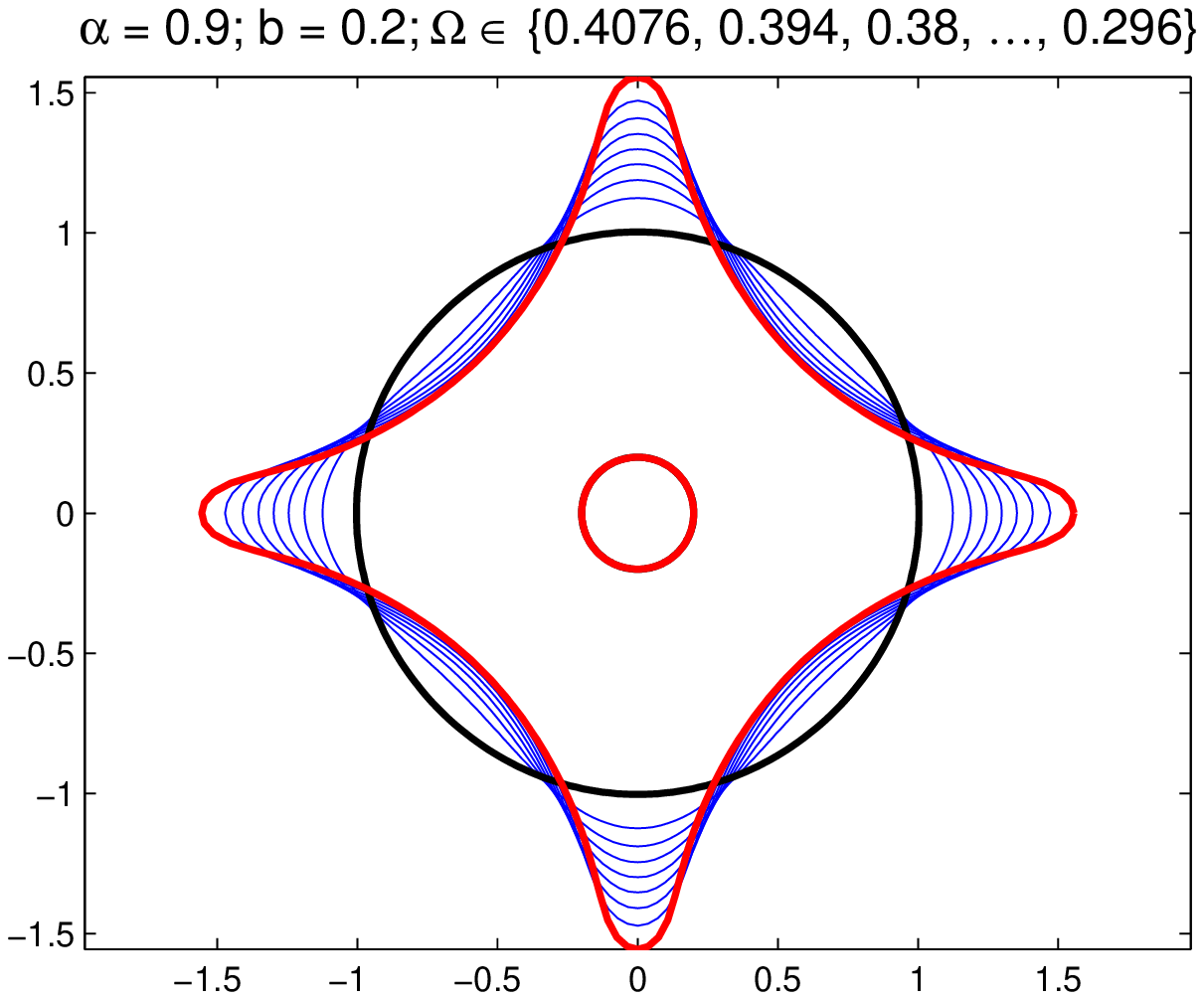}~
\includegraphics[width=0.5\textwidth, clip=true]{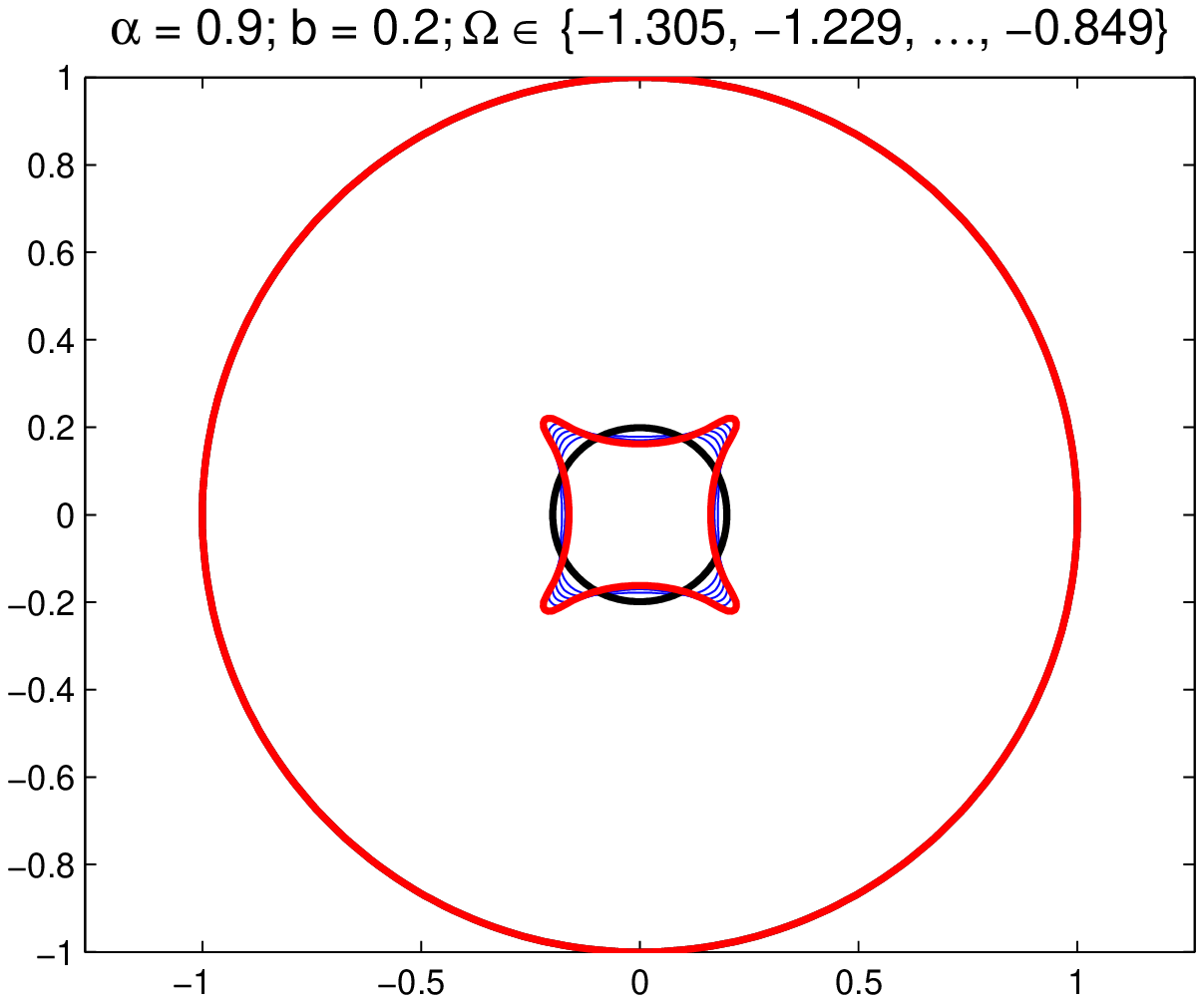}
\caption{\small{4-fold $V$-states corresponding to $\alpha = 0.9$, for different values of $\Omega$. On the left-hand side, we have started to bifurcate from $\Omega_4^+(0.2) = 0.4077\ldots$; on the right-hand side, from $\Omega_4^-(0.2) = -1.3055\ldots$.}}
\label{f:VState1holeAlpha0_9b0_2Omegas}
\end{figure}

Figure \ref{f:VState1holeAlpha0_9b0_2Omegas} shows the obvious parallelism with \cite{H-F-M-V} in the vortex patch problem: as $b$ becomes smaller, bifurcating at $\Omega_m^+(b)$ yields doubly-connected $V$-states closer and closer to simply-connected $V$-states; whereas, bifurcating at $\Omega_m^-(b)$ yields double-connected $V$-states closer and closer to the unit circumference. This explains why there are two bifurcation values of $\Omega$ in the doubly-connected case, and just one single bifurcation value of $\Omega$ in the simply-connected case. Nonetheless, unlike what would have happened in the vortex patch problem, the $V$-states corresponding to $\Omega = 0.296$ and to $\Omega = -0.849$ seem to have developed no singularity. Again, as in the simply-connected case, the explanation is given by the loss of monotonicity in the bifurcation curves of $a_{1,1}$ and $a_{2,1}$ in \eqref{e:z1z2cos} with respect to $\Omega$, plotted in Figure \ref{f:bifurcation1hole}, which predicts the existence of saddle-node bifurcation points. Observe also that, when we bifurcate from $\Omega_4^+(0.2)$, $a_{1,1}$ clearly dominates; whereas, when we bifurcate from $\Omega_4^-(0.2)$, $a_{2, 1}$ dominates and $a_{1,1}$ is of the order of $\mathcal O(10^{-6})$. This confirms our previous observations.

\begin{figure}[!htb]
\center
\includegraphics[width=0.5\textwidth, clip=true]{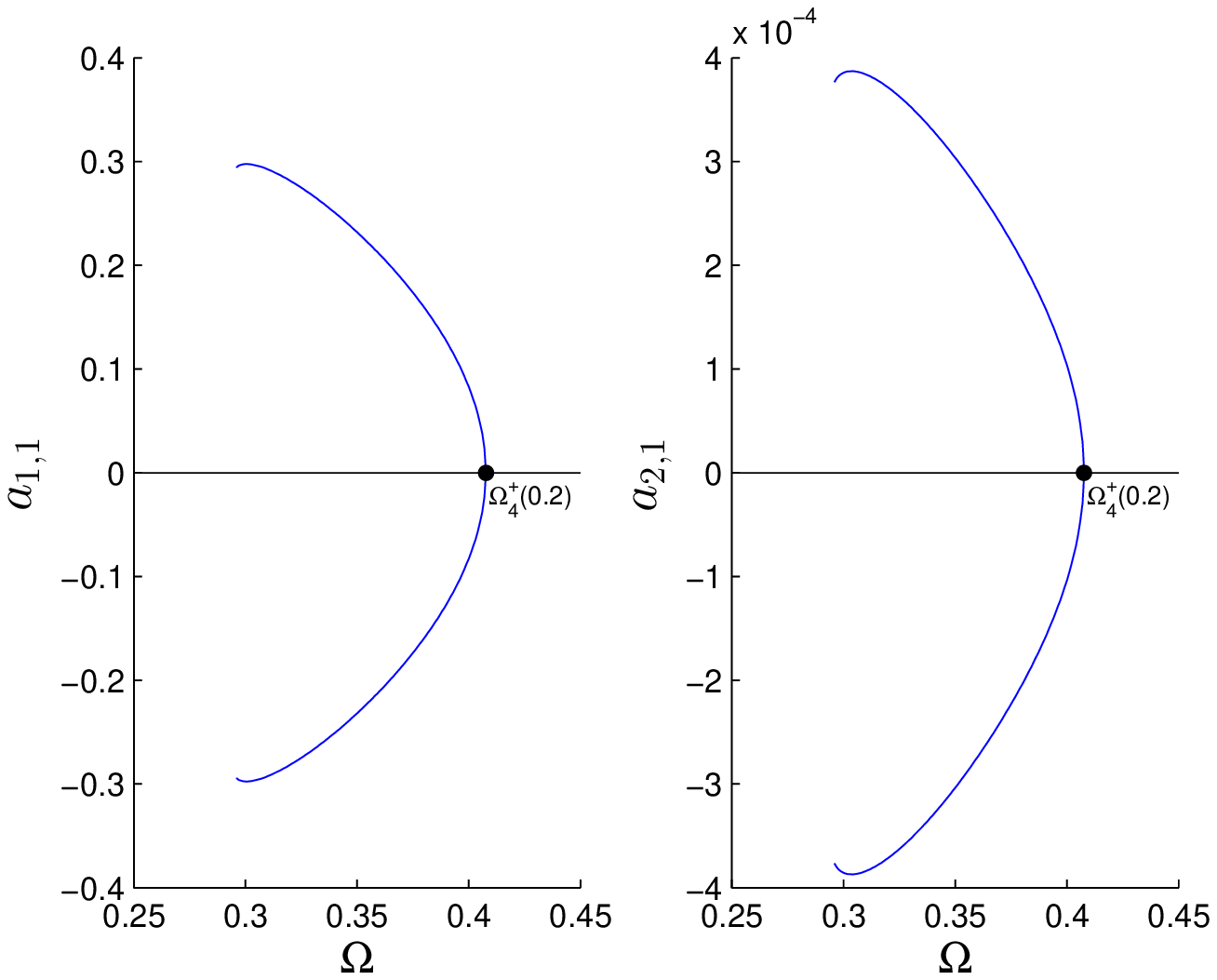}~
\includegraphics[width=0.5\textwidth, clip=true]{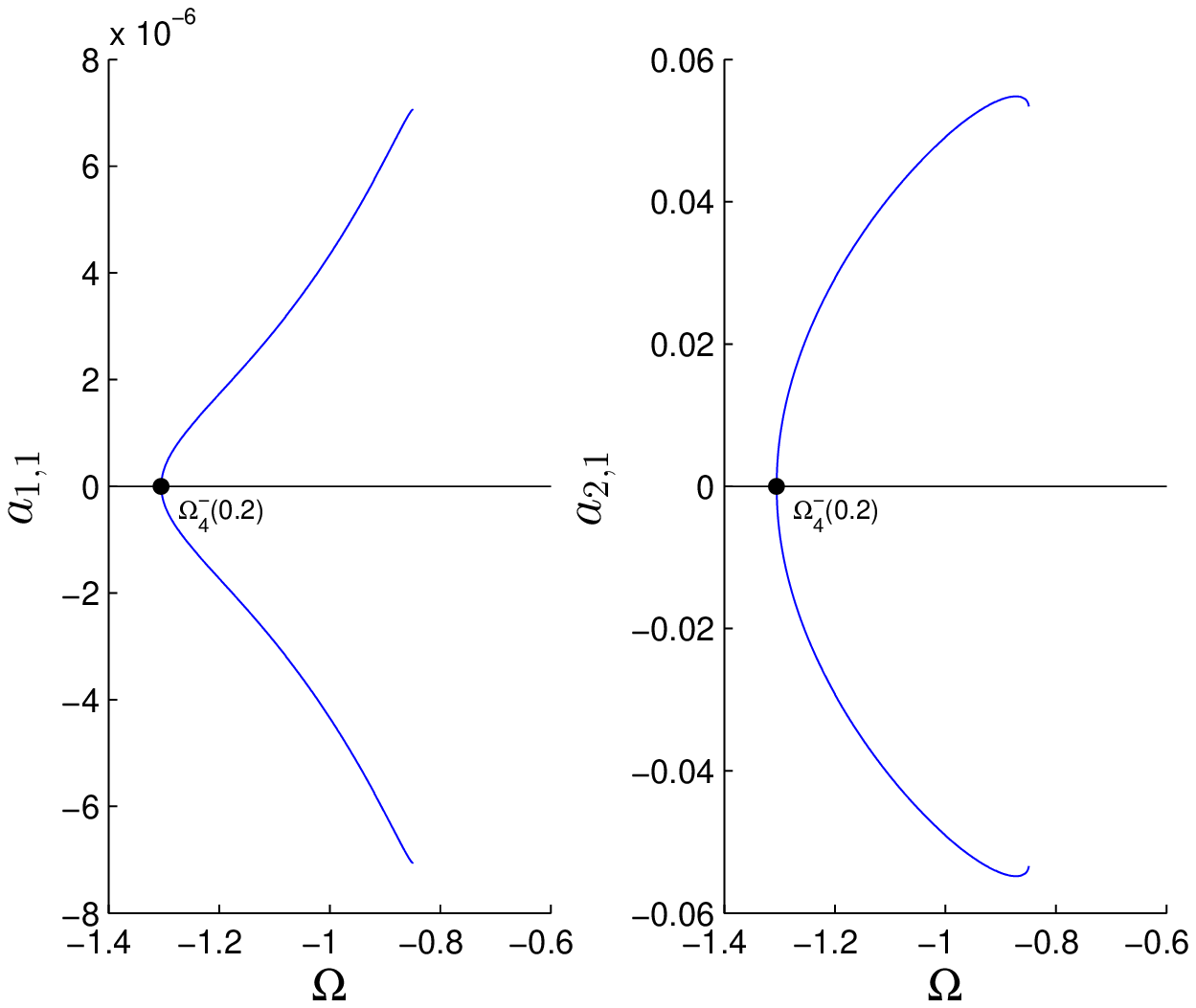}
\caption{\small{Bifurcation curves, for $m = 4$, $\alpha = 0.9$, and $b = 0.2$. The leftmost two correspond to a bifurcation from $\Omega_4^+(0.2)$, whereas the rightmost two correspond to a bifurcation from $\Omega_4^-(0.2)$. Remember that $a_{1,1}$ and $a_{2,1}$ are such that $a_{1,1}\cdot a_{2,1} < 0$.}}
\label{f:bifurcation1hole}
\end{figure}

Following the procedure explained in the simply-connected case (where we set $\lambda^{(A)} = 0$, $\lambda^{(B)} = \lambda^{(A)} + [(\Omega^{(B)} - \Omega^{(A)})^2 + (a_{1,1}^{(B)} - a_{1,1}^{(A)})^2 + (a_{2,1}^{(B)} - a_{2,1}^{(A)})^2]^{1/2}$, and so on), we have continued the bifurcation curves at $\Omega = 0.29507$, and $\Omega = -0.84878$, until $\Omega = 0.3$ and $\Omega = -0.86$, respectively, as is shown in Figure \ref{f:bifurcation1holeupdated}. It is still possible to go a bit further, but, in order not to lose accuracy, a larger number of nodes is convenient. The pieces of curve beyond the saddle-node bifurcation points are shown in thicker stroke.

\begin{figure}[!htb]
\center
\includegraphics[width=0.5\textwidth, clip=true]{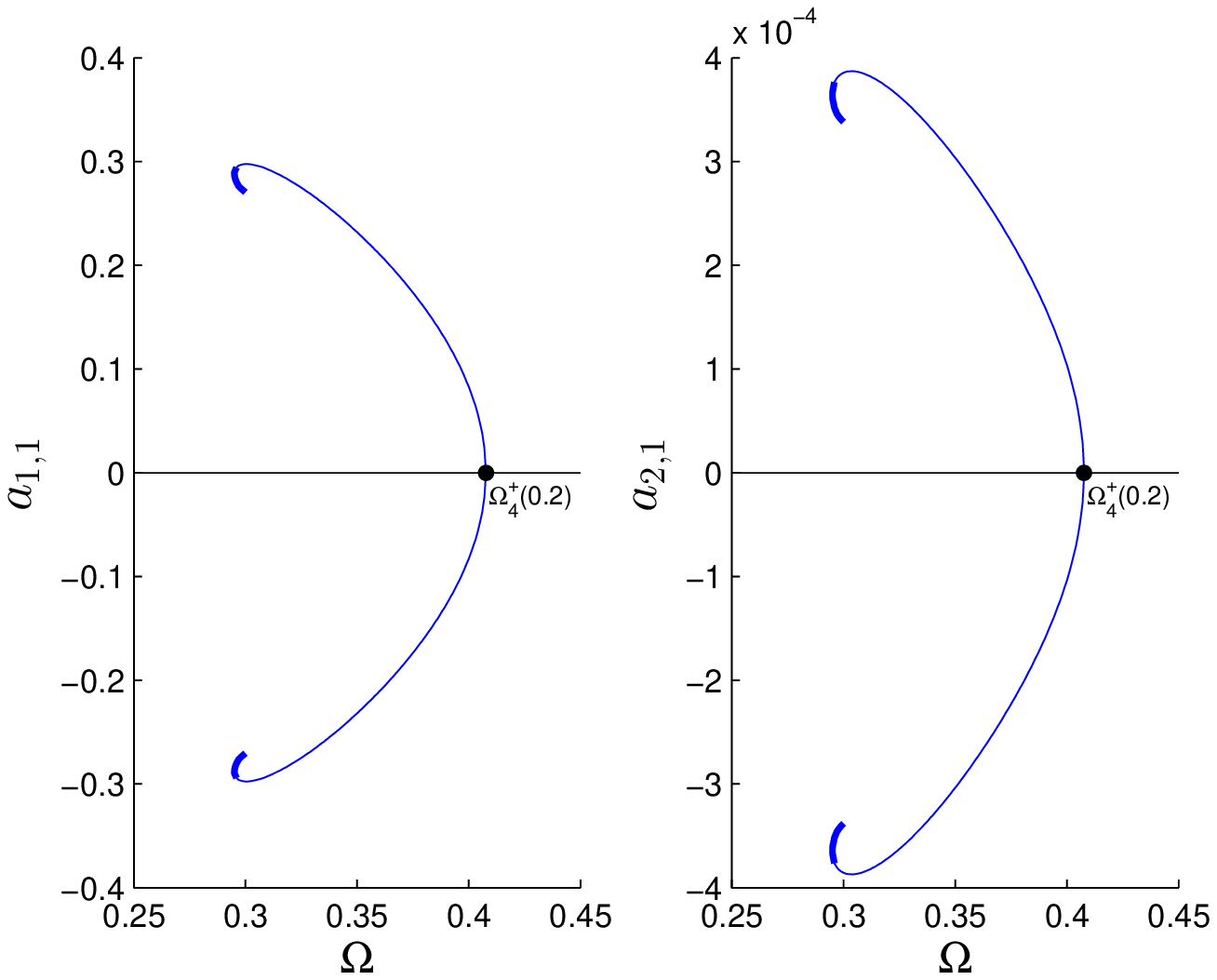}~
\includegraphics[width=0.5\textwidth, clip=true]{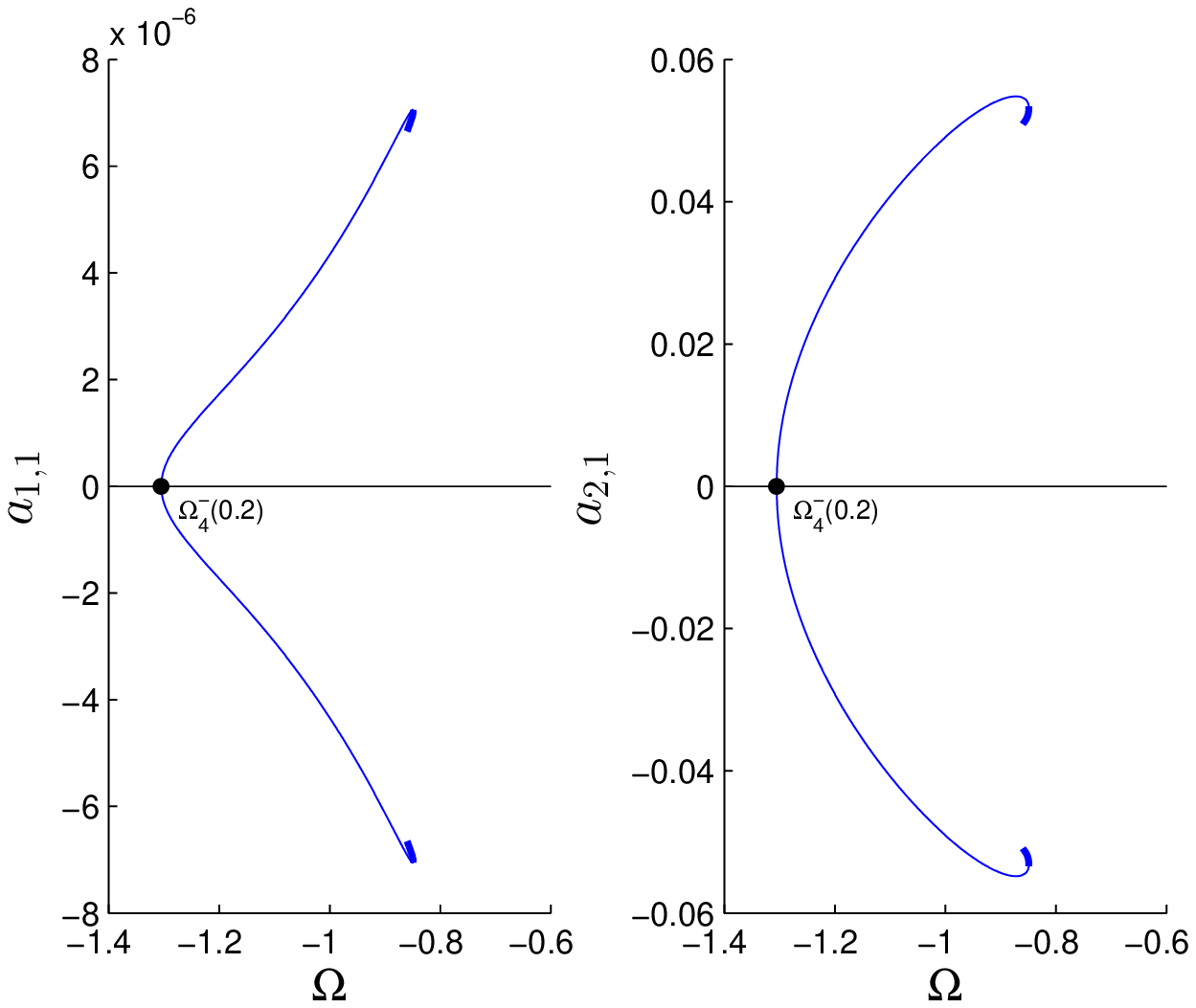}
\caption{\small{Extended bifurcation curves, for $m = 4$, $\alpha = 0.9$, and $b = 0.2$. The leftmost two correspond to a bifurcation from $\Omega_4^+(0.2)$, whereas the rightmost two correspond to a bifurcation from $\Omega_4^-(0.2)$. Remember that $a_{1,1}$ and $a_{2,1}$ are such that $a_{1,1}\cdot a_{2,1} < 0$.}}
\label{f:bifurcation1holeupdated}
\end{figure}

In Figure \ref{f:VState1holeAlpha0_9b_2}, we have plotted the $V$-states corresponding to $\Omega = 3$ and $\Omega = -0.86$, but beyond the saddle-node bifurcation points. The differences with respect to the $V$-states in red from Figure \ref{f:VState1holeAlpha0_9b0_2Omegas} are evident. It would be interested to calculate which kind of limiting $V$-states develops. Finally, whether the lower bound restriction for $b$ can be ignored completely or not is another relevant question.

\begin{figure}[!htb]
\center
\includegraphics[width=0.5\textwidth, clip=true]{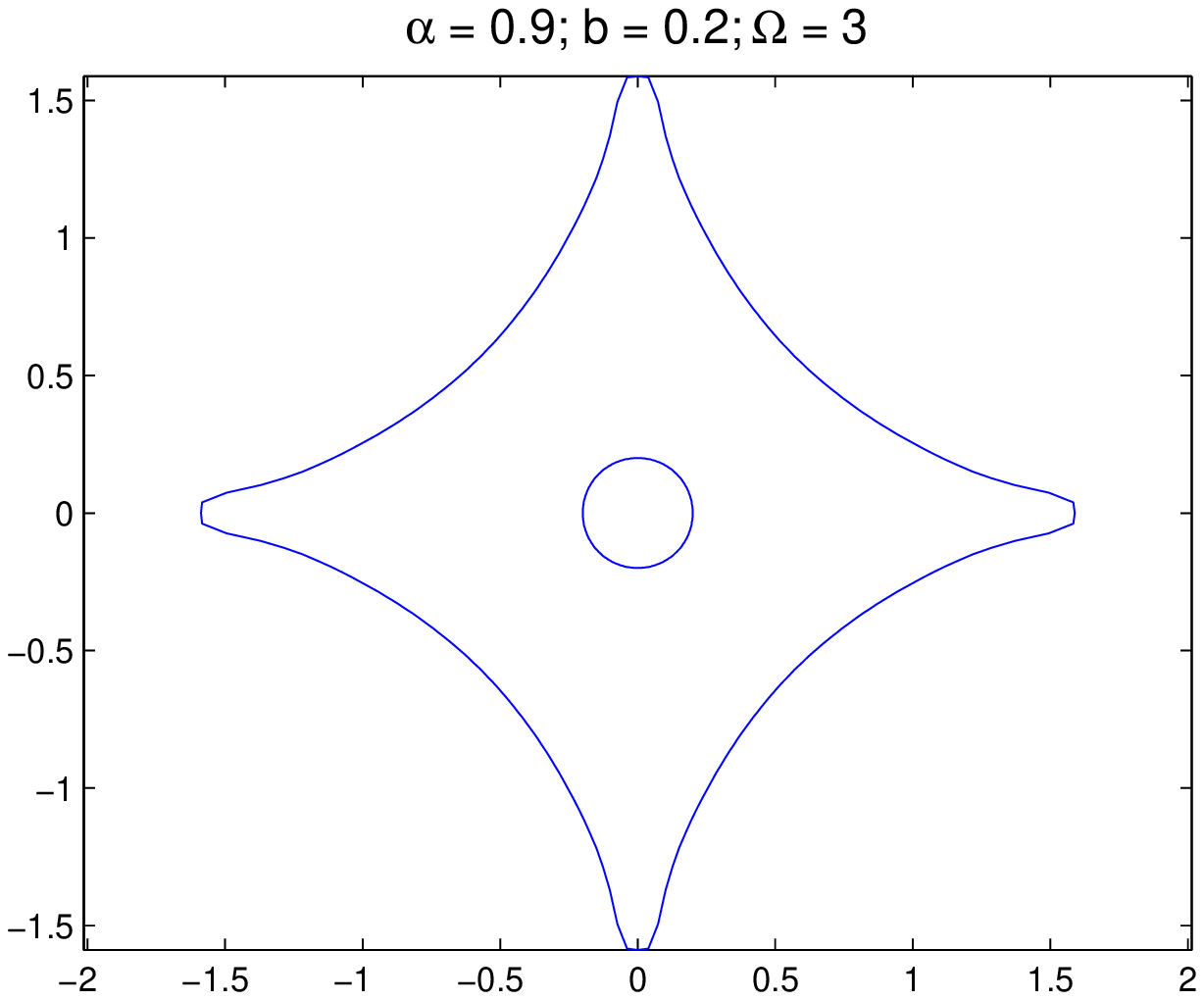}~
\includegraphics[width=0.5\textwidth, clip=true]{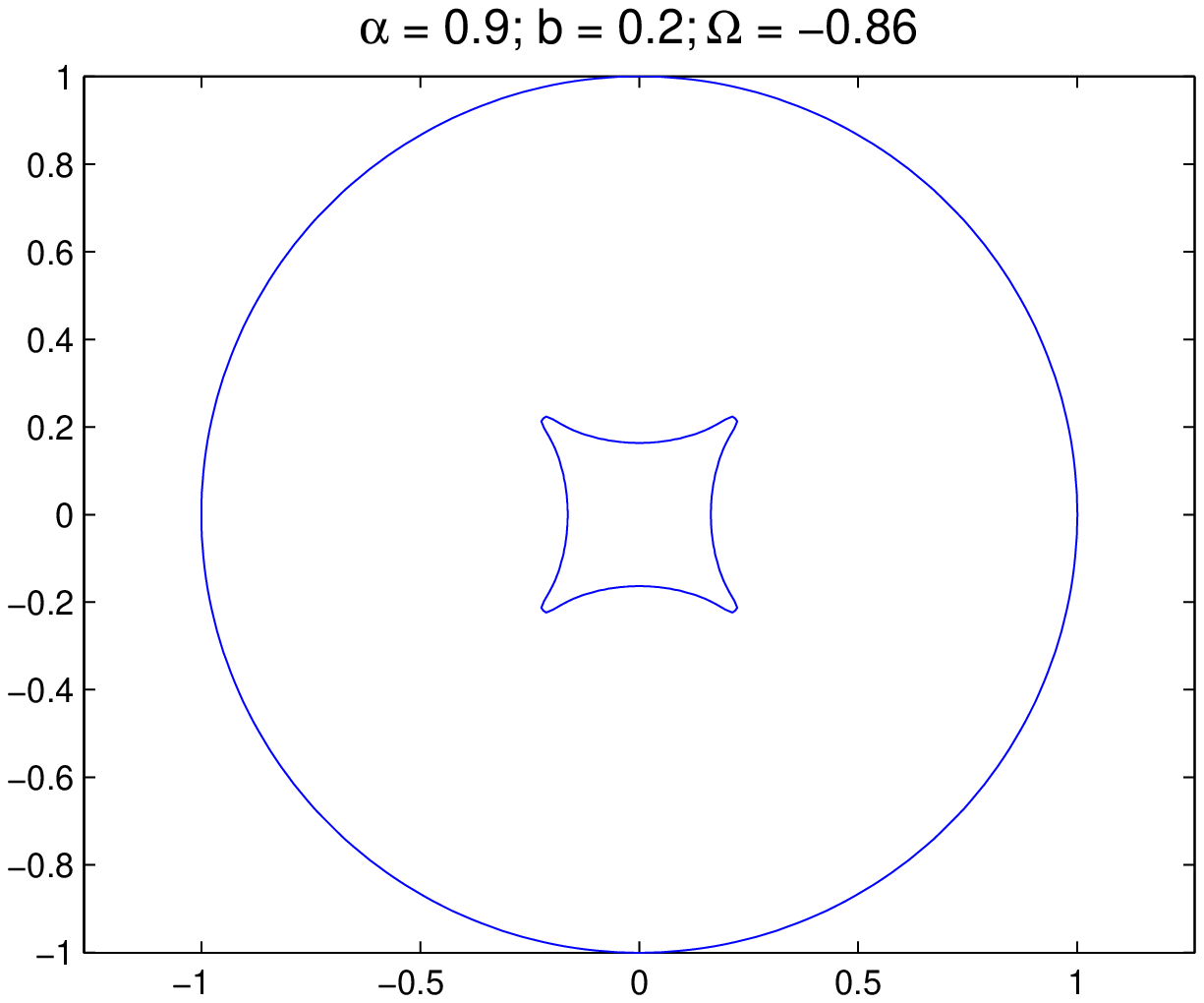}
\caption{\small{$V$-states beyond the saddle-node bifurcation points.}}
\label{f:VState1holeAlpha0_9b_2}
\end{figure}

Another relevant fact is the possibility to find examples where $V$-states exist for all $\Omega\in[\Omega_m^-, \Omega_m^+]$; such situation was also found in the vortex patch problem. We have considered, for instance, $m = 4$, $\alpha = 0.5$, and $b = 0.65$, with $\Omega_4^+(0.65) = 0.1480\ldots$ and $\Omega_4^-(0.65) = 0.08168\ldots$. Observe that $b_0(0.5) = 0.7424\ldots$, so we are violating the restriction on $b_0$ again. We have computed successfully all the $V$-states with $\Omega = 0.082, 0.083, \ldots, 0.148$. On the left-hand side of Figure \ref{f:bifurcation1whole}, we plot the $V$-states corresponding to $\Omega = 0.082, 0.093, \ldots, 0.148$. The $V$-state corresponding to $\Omega = 0.148$ (in red), and to $\Omega = 0.082$ (in black), are practically circular annuli (we use no thicker stroke here, because all the $V$-states are very close to each other). On the right-hand side of Figure \ref{f:bifurcation1whole}, we plot the bifurcation curves of $a_{1,1}$ and $a_{2,1}$ in \eqref{e:z1z2cos}, with respect to $\Omega$. As expected, the curves are closed; however, the curve corresponding to $a_{2,1}$ is much more symmetrical and reminds us of an ellipse.
\begin{figure}[!htb]
\center
\includegraphics[width=0.5\textwidth, clip=true]{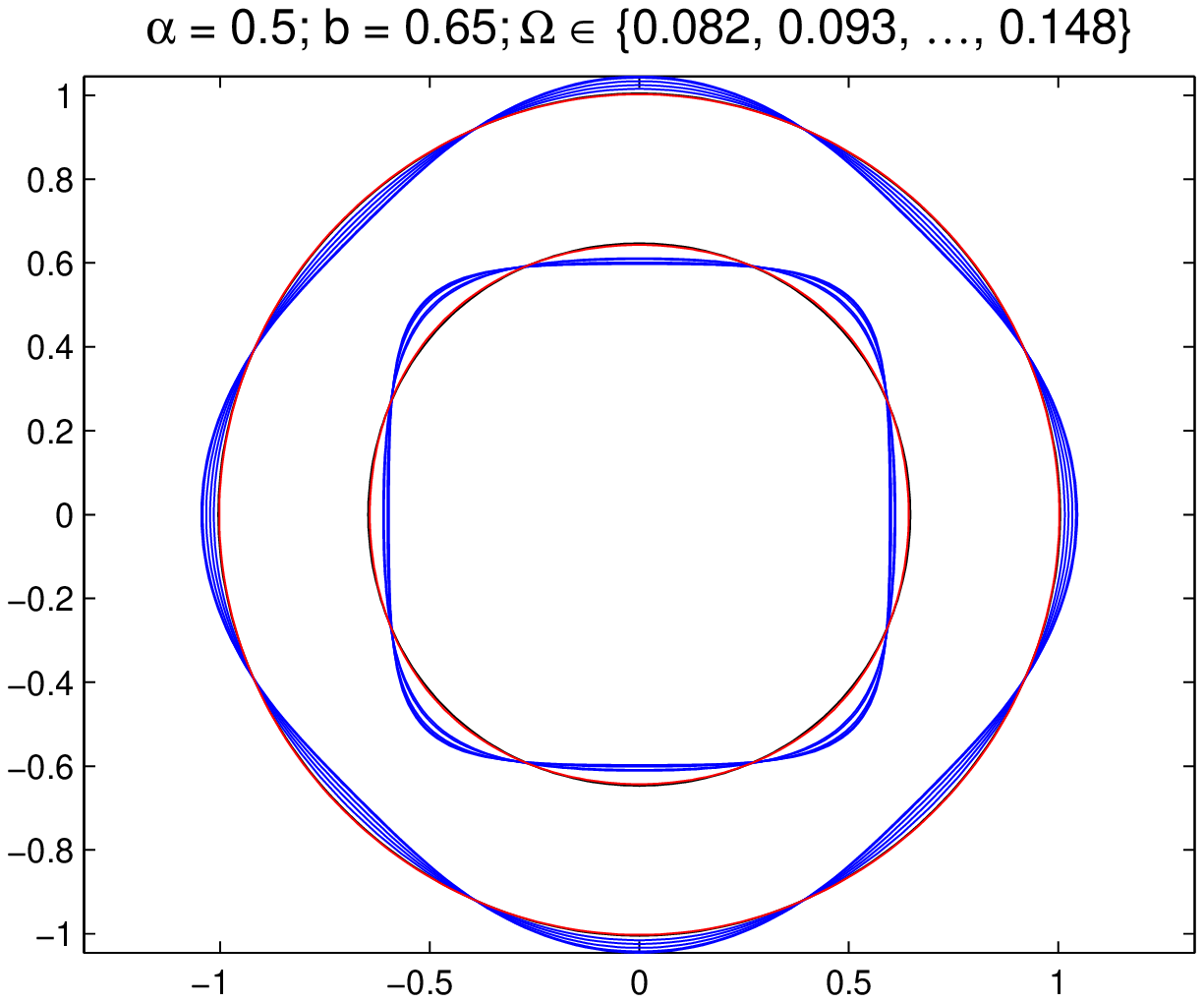}~
\includegraphics[width=0.5\textwidth, clip=true]{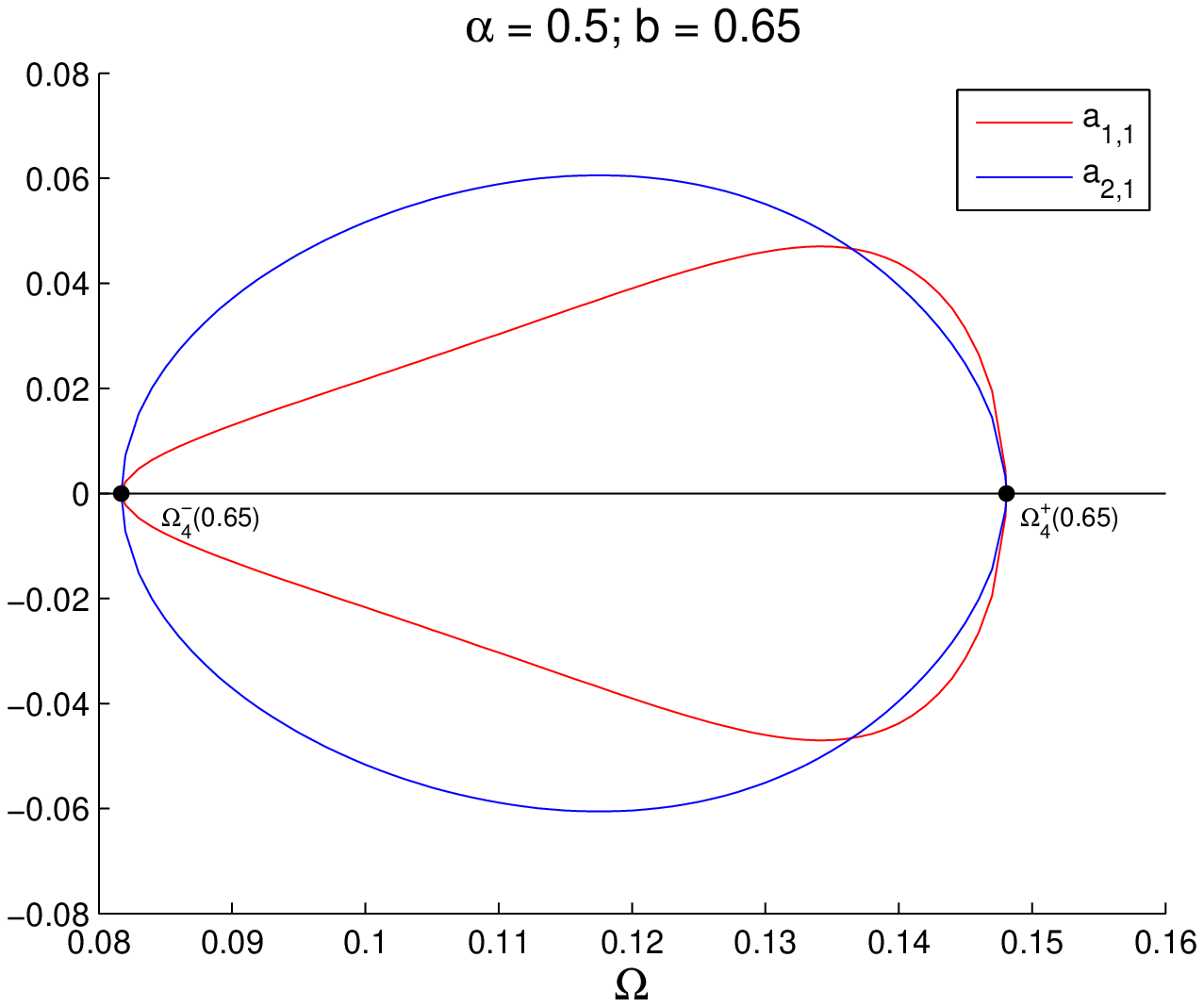}
\caption{\small{Left: 4-fold $V$-states corresponding to $\alpha = 0.5$, $b = 0.65$, for different values of $\Omega$; there are $V$-states for all $\Omega\in[\Omega_m^-, \Omega_m^+]$. Right: bifurcation curves of $a_{1,1}$ and $a_{2,1}$ in \eqref{e:z1z2cos} with respect to $\Omega$. The curve corresponding to $a_{2,1}$ is much more symmetrical. Remember that $a_{1,1}$ and $a_{2,1}$ are such that $a_{1,1}\cdot a_{2,1} < 0$.}}
\label{f:bifurcation1whole}
\end{figure}

Let us finish this section by mentioning the existence of stationary doubly-connected $V$-states when $\alpha > 0$, i.e., $V$-states with $\Omega = 0$. Like the examples with $\Omega < 0$ shown above, they have no particularity from a numerical point of view, yet they are a completely new phenomenon with respect to the vortex patch problem. To obtain them, it is necessary to choose $m$, $\alpha$ and $b$, such that $\Omega_m^- < 0$, but $|\Omega_m^-| \ll 1$, since we bifurcate from the annulus at $\Omega = \Omega_m^-$. We have chosen $m = 4$ and $\alpha = 0.5$, as in the last experiment, but with a $b$ even smaller, $b = 0.5$, in such a way that $\Omega_4^-(0.5) = -0.02760\ldots$. Figure \ref{f:stationary} shows the corresponding stationary $V$-state.
\begin{figure}[!htb]
\center
\includegraphics[width=0.5\textwidth, clip=true]{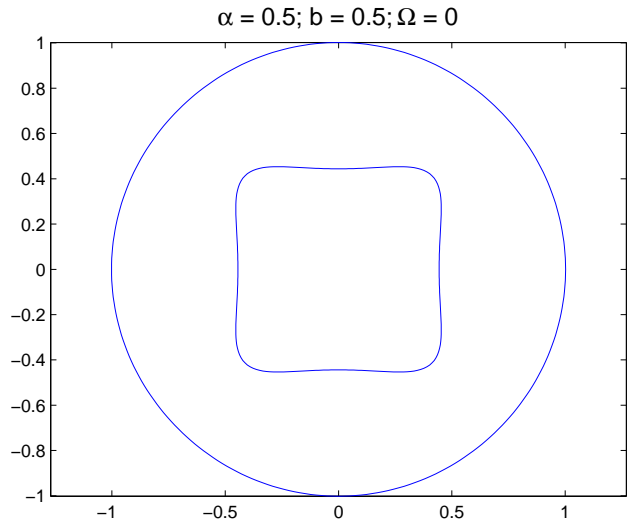}
\caption{\small{Example of a stationary $V$-state.}}
\label{f:stationary}
\end{figure}

\begin{ackname}
Francisco de la Hoz was supported by the Spanish Ministry of Economy and Competitiveness, with the project
MTM2011-24054, and by the Basque Government, with the project IT641-13. Zineb Hassainia and Taoufik Hmidi were partially supported by  the ANR
project Dyficolti ANR-13-BS01-0003-01.

 \end{ackname}

\end{document}